\documentclass[11pt, reqno]{amsart}
\usepackage{txfonts}
\usepackage{amsmath,amssymb,amsthm,mathrsfs,enumerate,bm,xcolor,multirow,pbox}
\usepackage{graphicx,color,framed,tikz,caption,subcaption}
\usepackage{enumitem}
\setlist{leftmargin=9mm}
\usepackage[colorlinks,linkcolor=black,citecolor=black,urlcolor=black]{hyperref}
\allowdisplaybreaks[4]
\numberwithin{equation}{section}
\newcommand{\N}{\mathbb{N}}
\newcommand{\R}{\mathbb{R}}

\newcommand{\pnorm}[2]{\lVert #1\rVert_{#2}}
\newcommand{\bigpnorm}[2]{\big\lVert#1\big\rVert_{#2}}
\newcommand{\biggpnorm}[2]{\bigg\lVert#1\bigg\rVert_{#2}}
\newcommand{\abs}[1]{\lvert#1\rvert}
\newcommand{\bigabs}[1]{\big\lvert#1\big\rvert}
\newcommand{\biggabs}[1]{\bigg\lvert#1\bigg\rvert}
\newcommand{\iprod}[2]{\langle#1,#2\rangle}

\newcommand{\biggiprod}[2]{\bigg\langle#1,#2\bigg\rangle}

\renewcommand{\epsilon}{\varepsilon}

\newcommand{\equald}{\stackrel{d}{=}}
\renewcommand{\hat}{\widehat}

\DeclareMathOperator{\E}{\mathbb{E}}
\DeclareMathOperator{\Prob}{\mathbb{P}}

\DeclareMathOperator{\var}{Var}

\DeclareMathOperator{\op}{op}

\DeclareMathOperator{\gw}{\mathfrak{w}}

\DeclareMathOperator{\nullsp}{null}

\DeclareMathOperator{\proj}{\mathsf{P}}
\DeclareMathOperator{\dis}{dist}

\let\liminf\relax
\DeclareMathOperator*\liminf{\underline{lim}}
\let\limsup\relax
\DeclareMathOperator*\limsup{\overline{lim}}

\DeclareMathOperator*{\argmin}{arg\,min\,}

\newcommand{\beq}{\begin{equation}}
\newcommand{\eeq}{\end{equation}}
\newcommand{\beqa}{\begin{equation} \begin{aligned}}
\newcommand{\eeqa}{\end{aligned} \end{equation}}
\newcommand{\beqas}{\begin{equation*} \begin{aligned}}
\newcommand{\eeqas}{\end{aligned} \end{equation*}}

\newcommand{\bit}{\begin{itemize}}
	\newcommand{\eit}{\end{itemize}}
\newcommand{\bmat}{\begin{bmatrix}}
	\newcommand{\emat}{\end{bmatrix}}

\theoremstyle{definition}\newtheorem{problem}{Problem}[section]
\theoremstyle{definition}
\theoremstyle{remark}

\theoremstyle{remark}
\theoremstyle{definition}
\theoremstyle{plain}\newtheorem{theorem}[problem]{Theorem}
\theoremstyle{plain}\newtheorem{question}{Question}
\theoremstyle{plain}\newtheorem{lemma}[problem]{Lemma}
\theoremstyle{plain}\newtheorem{proposition}[problem]{Proposition}
\theoremstyle{plain}\newtheorem{corollary}[problem]{Corollary}
\theoremstyle{plain}

\AtBeginDocument{%
	\def\MR#1{}
}

\begin{document}

\title[Gaussian random projections of convex cones]{Gaussian random projections of convex cones: Approximate kinematic formulae and applications}
%
\author[Q. Han]{Qiyang Han}

\address[Q. Han]{
Department of Statistics, Rutgers University, Piscataway, NJ 08854, USA.
}
\email{qh85@stat.rutgers.edu}

\author[H. Ren]{Huachen Ren}

\address[H. Ren]{
	Department of Statistics, Rutgers University, Piscataway, NJ 08854, USA.
}
\email{hr306@scarletmail.rutgers.edu}

\date{\today}

\begin{abstract}
Understanding the stochastic behavior of random projections of geometric sets constitutes a fundamental problem in high dimension probability that finds wide applications in diverse fields. This paper provides a kinematic description for the behavior of Gaussian random projections of closed convex cones, in analogy to that of randomly rotated cones studied in \cite{amelunxen2014living}. Formally, let $K$ be a closed convex cone in $\R^n$, and $G\in \R^{m\times n}$ be a Gaussian matrix with i.i.d. $\mathcal{N}(0,1)$ entries. We show that $GK\equiv \{G\mu: \mu \in K\}$ behaves like a randomly rotated cone in $\R^m$ with statistical dimension $\min\{\delta(K),m\}$, in the following kinematic sense: 
for any fixed closed convex cone $L$ in $\R^m$,
\begin{align*}
&\delta(L)+\delta(K)\ll m\, \Rightarrow\, L\cap GK = \{0\} \hbox{ with high probability},\nonumber\\
&\delta(L)+\delta(K)\gg m\, \Rightarrow\, L\cap GK \neq \{0\} \hbox{ with high probability}.
\end{align*}
Similar kinematic descriptions are obtained for Gaussian random pre-images, and certain Gaussian random projections of general closed convex sets.

The practical utility and broad applicability of the prescribed approximate kinematic formulae are demonstrated in a number of distinct problems arising from statistical learning, mathematical programming and asymptotic geometric analysis. In particular, we prove (i) new phase transitions of the existence of cone constrained maximum likelihood estimators in logistic regression, (ii) new phase transitions of the cost optimum of deterministic conic programs with random constraints, and (iii) a local version of the Gaussian Dvoretzky-Milman theorem that describes almost deterministic, low-dimensional behaviors of subspace sections of randomly projected convex sets.

The proofs of our results exploit the full strength of comparison inequalities for Gaussian processes. Compared to the conic integral geometry method in \cite{amelunxen2014living}, our method has the advantage of circumventing the rigid requirement of exact kinematic formulae that are typically unavailable for random projections and general closed convex sets.
\end{abstract}

\maketitle


\vspace{-1em}
\setcounter{tocdepth}{1}
\tableofcontents


\sloppy

\section{Main results}\label{section:main_results}

\subsection{The problem and motivating questions}
Let $m,n$ be two positive integers, $K$ be a closed convex cone in $\R^n$. We reserve the notation $G \in \R^{m\times n}$ for a standard Gaussian matrix in that
\begin{align}\label{def:standard_gaussian_matrix}
G \in \R^{m\times n} \hbox{ contains i.i.d. $\mathcal{N}(0,1)$ entries}. 
\end{align}
We will be interested in the stochastic behavior of the \emph{Gaussian random projection} of $K\subset \R^n$, defined as
\begin{align}\label{def:GT}
GK\equiv \{G \mu: \mu \in K\}\subset \R^m. 
\end{align}
The study of the stochastic behavior of Gaussian random projections has a long history in the high dimensional probability literature; see e.g.,  \cite{pisier1989volume,ledoux2013probability,boucheron2013concentration,artstein2015asymptotic,vershynin2018high} for textbook treatments on this topic. 

The lasting interest in the theory and ideas of (Gaussian) random projections is partly due to their wide applications in problems arising from signal processing, statistics, and computational mathematics. Interested readers are referred to, e.g., \cite{vempala2001random,achlioptas2003database,candes2006near,dasgupta2008random,baraniuk2009random,lopes2011more,dahl2013large,durrant2015random,cannings2017random} for a highly non-exhaustive but diverse list of related applications. 

To provide some concrete applied and statistical context for studying the particular form of (\ref{def:GT}), we consider below two examples arising from convex-constrained statistical estimation problems.

\subsubsection{Cone-constrained minimum-norm interpolators in linear regression}

Consider the linear regression setting: suppose we observe i.i.d.\ data $\{(X_i,Y_i): i\in [m]\}\subset \R^n\times \R$ from the model 
\begin{align}\label{def:linear_model}
Y_i = \iprod{X_i}{\mu_0}+\xi_i,
\end{align}
where $\mu_0 \in \R^n$ is an unknown signal vector of interest, and $\xi_i$’s are unobservable statistical errors. Here we are interested in estimating or recovering a structured signal $\mu_0 \in K$ modeled a closed convex cone $K \subset \R^n$, via the popular minimum-norm interpolator defined as 
\begin{align}\label{def:cone_interpolator}
\hat{\mu}_K \in \argmin_{\mu \in K: Y_i=\iprod{X_i}{\mu},\forall i \in [m]} \pnorm{\mu}{}.
\end{align}
In the special case $K = \R^n$, the unconstrained minimum-norm interpolator $\hat{\mu}_{\R^n}$ is known to be the convergence point of natural gradient descent in overparameterized linear regression. It therefore exhibits a strong connection to the theory of learning in overparameterized models, which has received significant recent attention in the statistics and machine learning literature and has been studied in increasing depth in \cite{bartlett2020benign,wu2020optimal,bartlett2021deep,hastie2022surprises,montanari2022interpolation,montanari2023generalization,han2023distribution}. 

One important reason why the unconstrained minimum-norm interpolator $\hat{\mu}_{\R^n}$ serves as a convenient theoretical model for understanding learning in overparameterized regimes is that its existence undergoes a sharp phase transition: $\hat{\mu}_{\R^n}$ is well defined if and only if the sample size $m$ does not exceed the ambient dimension $n$ (under very general conditions on the design matrix). 

For a general cone constraint $K$, it is reasonable to speculate that the minimum-norm interpolator $\hat{\mu}_K$ coincides with the convergence point of natural projected gradient descent, provided that the sample size does not exceed a certain threshold. This raises the natural question:

\begin{question}\label{question:linear_reg}
	For a general closed convex cone $K\subset \R^n$, what is the phase-transition point of the sample size $m$ for the existence of the cone-constrained minimum-norm interpolator $\hat{\mu}_K$?
\end{question}

\subsubsection{Cone-constrained maximum likelihood estimator in logistic regression}

The phase-transition phenomenon for statistical estimators is not confined to the minimum-norm interpolator (\ref{def:cone_interpolator}) in linear regression. Another well-known example is the maximum likelihood estimator in logistic regression. Concretely, suppose $m$ i.i.d.\ samples $(X_i,Y_i) \in \R^n\times \{\pm 1\}, i \in [m]$, are observed with 
\begin{align}\label{eqn:glm}
\Prob(Y_i=1|X_i)=1-\Prob(Y_i=-1|X_i)= 1/\big(1+e^{-X_i^\top \beta_0}\big)
\end{align}
for some unknown regression vector $\beta_0 \in \R^n$. The statistical problem is to estimate $\beta_0$ based on the observations $\{(X_i,Y_i): i \in [m]\}$. In this logistic regression problem, the log-likelihood is given by 
\begin{align}\label{eqn:logistic_likelihood}
\ell(\beta) = -\sum_{i=1}^m \log \Big[1+\exp\big(-Y_i \cdot X_i^\top \beta\big)\Big]. 
\end{align}
The maximum likelihood estimator (MLE) $\hat{\beta}_{\R^n}$ is then defined as any maximizer of the map $\beta \mapsto  \ell(\beta)$ whenever such a maximizer exists. 

It is well known that the existence of the unconstrained MLE $\hat{\beta}_{\R^n}$ exhibits a sharp phase transition. For instance, under the global null $\beta_0=0$ and in the proportional high-dimensional regime $n/m\to \kappa \in (0,\infty)$, the seminal result of \cite{cover1965geometrical} shows that the unconstrained MLE $\hat{\beta}_{\R^n}$ exists if and only if $\kappa<1/2$.

Here we are interested in a convex-constrained generalization of the MLE defined above. Specifically, for a given closed convex cone $K \subset \R^n$, the \emph{cone-constrained maximum likelihood estimator (MLE)} $\hat{\beta}_K$ is defined as any maximizer of the map $\beta \mapsto \ell(\beta)$ over $\beta \in K$, whenever such a maximizer exists. Such constraints arise naturally in applications; readers are referred to \cite{han2023noisy} for related literature and motivating examples.

An immediate question arises regarding the existence of $\hat{\beta}_K$:
\begin{question}\label{question:logistic_reg}
	For a general closed convex cone $K\subset \R^n$, what is the phase-transition point of the sample size $m$ for the existence of the cone-constrained maximum likelihood estimator $\hat{\beta}_K$ in logistic regression?
\end{question}

As will be clear below, at least when the covariates $X_i$ follow a standard Gaussian design, both Questions \ref{question:linear_reg} and \ref{question:logistic_reg} boil down to understanding the phase-transitional behavior of the random projection $GK$ defined in (\ref{def:GT}) hitting certain structured sets (or its generalized version).

Of course, the utility of a precise understanding of (\ref{def:GT}) extends far beyond the two statistical examples mentioned above. As will also become clear below, we will leverage our theoretical results for $GK$ to develop new findings in both optimization and asymptotic geometric analysis.

\subsection{Some heuristics for the behavior of random projections}
In typical applications related to our setting (\ref{def:GT}), the ambient dimension $n$ in which the convex cone $K$ lives is (much) larger than the projection dimension $m$ in which the random convex cone $GK$ resides. The central question of interest is therefore to understand how the stochastic geometry of the relatively low-dimensional random cone $GK$ can be described by that of the deterministic, but possibly high-dimensional cone $K$. 

One particular easy case to understand the behavior of $GK$ is when $K$ is a linear subspace of $\R^n$. In this case, $GK$ is simply distributed as a uniform random subspace of $\R^m$ with dimension $\min\{\dim(K),m\}$. Clearly, both the notion of dimension and the prescribed exact distribution description are specific to $K$ being subspaces. The former issue is less severe: there is a natural generalization of the notion of dimension for linear subspaces to the so-called \emph{statistical dimension} $\delta(K)$ for general convex cones $K$, formally defined as
\begin{align}\label{def:stat_dim}
\delta(K)\equiv \E \pnorm{\Pi_K(g)}{}^2,\quad g\sim \mathcal{N}(0,I_n).
\end{align}
Here $\pnorm{\cdot}{}$ denotes the Euclidean norm in $\R^n$, and $\Pi_K$ denotes the metric projection onto $K$ defined as $\Pi_K(\cdot)=\argmin_{\mu \in K}\pnorm{\cdot-\mu}{}$.\footnote{The readers are referred to Section \ref{section:preliminary} below and \cite[Section 3]{amelunxen2014living} for more technical discussions on the notion of statistical dimension and other related results. }

The major difficulty, however, appears to be the lack of natural generalizations of the exact distributional descriptions of $GK$ from $K$ being subspaces to general convex cones. Fortunately, the classical theory of Gaussian random projections gives useful hints on the stochastic behavior of $GK$. For instance, when $\delta(K)\gg m$, the Gaussian Dvoretzky-Milman Theorem (cf. \cite{dvoretzky1959theorem,milman1971new,pisier1989volume,artstein2015asymptotic}) and the cone property of $K$ indicate that $GK$ is trivially the full space $\R^m$ with high probability. On the other hand, when $m\gg \delta(K)$, the Johnson-Lindenstrauss Embedding Theorem (cf. \cite{johnson1984extensions,klartag2005empirical}) suggests that $GK\subset \R^m$ is locally almost isomorphic to $K\subset \R^n$. Combined with the rotation invariance of Gaussian random projections, it is then natural to conjecture that
\begin{align}\label{eqn:GK_heuristic}
&\hbox{$GK$ behaves like a randomly rotated cone in $\R^m$}\nonumber\\
&\qquad \hbox{with statistical dimension $\min\{\delta(K),m\}$}.
\end{align}
\subsection{Approximate kinematic formulae: main results}

\subsubsection{Approximate kinematic formulae for random projections}
The first goal of this paper is to formalize the heuristic (\ref{eqn:GK_heuristic}) for an arbitrary closed convex cone $K$, using the kinematic formulation in \cite{amelunxen2014living} that describes the behavior of randomly rotated cones. In particular, let $O_n$ be distributed according to the Haar measure on the orthogonal group $\mathrm{O}(n)$. \cite{amelunxen2014living} describes the kinematic behavior of $O_n K$ as follows: there exists some universal constant $C>0$ such that for any fixed closed convex cone $L\subset \R^n$ and $t\geq 1$, 
\begin{align}\label{eqn:ALMT}
\sqrt{\delta(K)}-\sqrt{n-\delta(L)}\leq -C\sqrt{t}\, &\Rightarrow \, \Prob\big(L\cap O_n K = \{0\}\big)\geq 1-e^{-t},\nonumber\\
\sqrt{\delta(K)}-\sqrt{n-\delta(L)}\geq C\sqrt{t}\, &\Rightarrow \, \Prob\big(L\cap O_n K \neq \{0\}\big)\geq 1-e^{-t}.
\end{align}
The above results (\ref{eqn:ALMT}) are called \emph{approximate kinematic formulae} in \cite{amelunxen2014living}, as they serve as approximations for the \emph{exact kinematic formulae} for $\Prob\big(L\cap O_n K \neq \{0\}\big)$, the form of which can be found in (\ref{eqn:exact_kinematic_formula}) in Section \ref{subsection:ALMT_proof_technique} below.

Using the kinematic formulation (\ref{eqn:ALMT}),  we may now give a formal description of (\ref{eqn:GK_heuristic}) as follows.

\begin{theorem}\label{thm:approx_kinematics}
Suppose that $K \subset \R^n$ and $L\subset \R^m$ are non-trivial closed convex cones. There exists some universal constant $C>0$ such that the following statements hold for $t\geq 1$.
\begin{enumerate}
	\item If $\sqrt{\delta(K)}-\sqrt{m-\delta(L)}\leq -C\sqrt{t}$, then 
	\begin{align*}
	\Prob\big(L\cap GK = \{0\}\big)\geq 1-e^{-t}.
	\end{align*}
	\item If $\sqrt{\delta(K)}-\sqrt{m-\delta(L)}\geq C\sqrt{t}$, then 
	\begin{align*}
	\Prob\big(L\cap GK \neq \{0\}\big)\geq 1-e^{-t}.
	\end{align*}
\end{enumerate}
\end{theorem}

 Remarkably, Theorem \ref{thm:approx_kinematics} provides a complete analogue to the approximate kinematic formulae (\ref{eqn:ALMT}) in its full generality, despite that \emph{exact} formulae for the probability $\Prob(L\cap GK\neq \{0\})$  are apparently unavailable (more technical comments on this can be found in Section \ref{subsection:ALMT_proof_technique} below). Due to the apparent similar appearance to (\ref{eqn:ALMT}), we will call the results in Theorem \ref{thm:approx_kinematics} \emph{approximate kinematic formulae} for $\Prob(L\cap GK\neq \{0\})$.

From a different angle, the main theme between (\ref{eqn:ALMT}) and Theorem \ref{thm:approx_kinematics} studies the proximity of random orthogonal matrices to Gaussian random matrices in the sense that $\Prob(L\cap O_nK\neq \{0\})\approx \Prob(L\cap GK\neq \{0\})$, at least when the closed convex cones $K,L$ are of the same dimensionality. A separate line in the random matrix theory studies similar proximity phenomena for random orthogonal and Gaussian matrices in various other senses; the interested readers are referred to \cite{jiang2006how,chatterjee2007multivariate,tropp2012comparison,thrampoulidis2015isotropically,jiang2019distances} and references therein for further details in this direction.

\subsubsection{Connections to Gordon's Escape Theorem}

When $L$ is taken to be a subspace, our Theorem \ref{thm:approx_kinematics} immediately implies an optimal, two-sided version of the so-called Gordon's Escape Theorem for spherically convex sets. We formally record this result below.

\begin{corollary}\label{cor:gordon_escape}
	Let $1\leq \ell\leq n$ be an integer and $L\subset \R^n$ be any fixed subspace with $\dim(L)=\ell$. Let $K\subset \R^n$ be a non-trivial closed convex cone. There exists some universal constant $C>0$ such that the following statements hold for $t\geq 1$.
	\begin{enumerate}
		\item If $\sqrt{\ell}-\sqrt{n-\delta(K)}\leq -C\sqrt{t}$, then
		\begin{align*}
		\Prob\big(K\cap O_n L= \{0\}\big)\geq 1-e^{-t}.
		\end{align*}
		\item If $\sqrt{\ell}-\sqrt{n-\delta(K)}\geq C\sqrt{t}$, then 
		\begin{align*}
		\Prob\big(K\cap O_nL\neq \{0\}\big)\geq 1-e^{-t}.
		\end{align*}
	\end{enumerate}
	Here recall $O_n$ is uniformly distributed on $\mathrm{O}(n)$. 
\end{corollary}

As will be clear below, the upper estimate part (1) (with more general $K$'s) is proved essentially by \cite{gordon1988milman} via a one-sided comparison principle for Gaussian processes.  The more difficult lower estimate part (2) is proved much later by \cite{amelunxen2014living} via the so-called Crofton formula (\ref{eqn:crofton_formula}) and the concentration of intrinsic volumes. As Corollary \ref{cor:gordon_escape} now follows as an immediate consequence of Theorem \ref{thm:approx_kinematics}, an interesting proof-theoretic implication of our approach is that it provides a purely random process proof for the classical Gordon's Escape Theorem, without resorting to conic integral geometry method as employed in \cite{amelunxen2014living}.

\subsubsection{Approximate kinematic formulae for random pre-images}

Next, we describe an approximate kinematic formula for the \emph{Gaussian random pre-image} $G^{-1}L$ for a fixed closed convex cone $L\subset \R^m$, formally defined by
\begin{align}\label{def:preimage}
G^{-1}L\equiv \{\mu \in \R^n: G\mu \in L\}. 
\end{align}
The approximate kinematic formula below describes the behavior of $G^{-1}L$. 

\begin{theorem}\label{thm:approx_kinematics_preimage}
	Suppose that $K \subset \R^n$ and $L\subset \R^m$ are non-trivial closed convex cones. There exists some universal constant $C>0$ such that the following statements hold for $t\geq 1$.
	\begin{enumerate}
		\item If $\sqrt{\delta(K)}-\sqrt{m-\delta(L)}\leq -C\sqrt{t}$, then 
		\begin{align*}
		\Prob\big(K\cap G^{-1}L = \{0\}\big)\geq 1-e^{-t}.
		\end{align*}
		\item If $\sqrt{\delta(K)}-\sqrt{m-\delta(L)}\geq C\sqrt{t}$, then 
		\begin{align*}
		\Prob\big(K\cap G^{-1}L \neq \{0\}\big)\geq 1-e^{-t}.
		\end{align*}
	\end{enumerate}
\end{theorem}

Despite the similar formulation of the above theorem to that of Theorem \ref{thm:approx_kinematics}, the kinematic formula in the above Theorem \ref{thm:approx_kinematics_preimage} is far from a direct consequence of Theorem \ref{thm:approx_kinematics}. Moreover, the kinematic interpretation of the two approximate formulae in Theorems \ref{thm:approx_kinematics} and \ref{thm:approx_kinematics_preimage} is markedly different. In particular, the informal interpretation (\ref{eqn:GK_heuristic}) for $GK$ requires a crucial modification for the purpose of describing the kinematics of $G^{-1}L$: In the most interesting regime $m\leq n$, Theorem \ref{thm:approx_kinematics_preimage} shows that the random pre-image $G^{-1}L$ should now be regarded as a randomly rotated cone in $\R^n$ with statistical dimension $\delta(L)+n-m$, where the extra `dimension' $n-m$ comes from the null space behavior of $G$.


\subsubsection{An extension to general closed convex sets $K$}

The approximate kinetic formulae in both Theorems \ref{thm:approx_kinematics} and \ref{thm:approx_kinematics_preimage} rely heavily on the cone structure of $K$. Below we shall examine certain aspects of the kinematic behavior of $GK$ where $K$ is a general closed convex set in $\R^n$. In particular, the following approximate kinetic formula determines whether the random set $GK$ hits/misses a given deterministic vector $\xi \in \R^m$, for a given number $m$ of the projection dimension. 


\begin{theorem}\label{thm:generalized_gordon}
	Let $\emptyset \neq K\subset \R^n$ be a non-empty closed convex set, and $\xi \neq 0$. There exists some absolute constant $C>0$ such that the following hold for $t\geq 1$.
	\begin{enumerate}
		\item If $\sqrt{m}\leq \mathscr{T}_K(\xi)-C\sqrt{t}$, then $\Prob(\xi \in GK)\geq 1-e^{-t}$.
		\item If $\sqrt{m}\geq \mathscr{T}_K(\xi)+C\sqrt{t}$, then $\Prob(\xi \notin GK)\geq 1-e^{-t}$. 
	\end{enumerate}
    Here with $\sigma_m^2\equiv \pnorm{\xi}{}^2/m$, 
    \begin{align*}
    \mathscr{T}_K(\xi)\equiv \gw\bigg(\bigg\{\frac{\mu}{\sqrt{\pnorm{\mu}{}^2+\sigma_m^2}}: \mu \in K\bigg\}\bigg),
    \end{align*}
    where $\gw(T)\equiv \E \sup_{t \in T}\iprod{g}{t}$, $g\sim \mathcal{N}(0,I_n)$ is the Gaussian width for (a bounded measurable set) $T \subset \R^n$. 
\end{theorem}

Note that when $K$ is a non-trivial closed convex cone and $\xi \neq 0$, $\xi \in GK$ is equivalent to $GK \cap \{t\xi: t\geq 0\} \neq \{0\}$. In this case, it is easy to see $\mathscr{T}_K(\xi)=\delta(K)$, and therefore the above theorem coincides with Theorem \ref{thm:approx_kinematics}. For general closed convex sets $K$, case-by-case techniques will be needed to compute the precise value of the quantity $\mathscr{T}_K(\xi)$.

From a statistical perspective, Theorem \ref{thm:generalized_gordon} above provides a solution to Question \ref{question:linear_reg} in the more general context of minimum interpolators constrained by an arbitrarily closed convex set $K$. With notation in (\ref{def:linear_model}), as the constraint set in the definition of $\hat{\mu}_K$ in (\ref{def:cone_interpolator}) can be rewritten as
\begin{align*}
\big\{G\mu = Y: \mu \in K\big\} = \big\{G(\mu-\mu_0)=\xi: \mu \in K\big\} = \big\{\xi \in G(K-\mu_0)\big\},
\end{align*}
Theorem \ref{thm:generalized_gordon}-(1) shows that the phase transition point of the sample size for the existence of $\hat{\mu}_K$ in (\ref{def:cone_interpolator}) is given by $\mathscr{T}_{K-\mu_0}(\xi)$, i.e.,
\begin{align*}
\sqrt{m}\leq \mathscr{T}_{K-\mu_0}(\xi)-C\sqrt{t}\, &\Rightarrow\, \Prob(\hat{\mu}_K \hbox{ exists})\geq 1-e^{-t};\\
\sqrt{m}\geq \mathscr{T}_{K-\mu_0}(\xi)+C\sqrt{t}\, &\Rightarrow\, \Prob(\hat{\mu}_K \hbox{ does not exist})\geq 1-e^{-t}.
\end{align*}
 A detailed study of the behavior of $\hat{\mu}_0$ is beyond the scope of this paper and will be pursued elsewhere. We believe the techniques used in the proof of the above theorem will be highly relevant for this problem. 


\subsection{Organization}
The rest of the paper is organized as follows. Applications of the approximate kinetic formulae to Question \ref{question:logistic_reg} in logistic regression and other problems are detailed in Section \ref{section:applications}. Section \ref{section:preliminary} presents an overview of our proof technique in comparison to the existing literature, and technical preliminaries on convex geometry and Gaussian process tools. The approximate kinematic formulae are then proved in Sections \ref{section:proof_approx_kinematics} and \ref{section:proof_generalized_kinetic}. Proofs for the applications in Section \ref{section:applications} are detailed in Section \ref{section:proof_application}. 

\subsection{Some further notation}

For any pair of positive integers $m\leq n$, let $[m:n]\equiv \{m,\ldots,n\}$, and $(m:n]\equiv [m:n]\setminus\{m\}$, $[m:n)\equiv [m:n]\setminus\{n\}$. We often write $[n]$ for $[1:n]$. For $a,b \in \R$, $a\vee b\equiv \max\{a,b\}$ and $a\wedge b\equiv\min\{a,b\}$. For $a \in \R$, let $a_+\equiv a\vee 0$ and $a_- \equiv (-a)\vee 0$. 

For $x \in \R^n$, let $\pnorm{x}{p}=\pnorm{x}{\ell_p(\R^n)}$ denote its $p$-norm $(1\leq p\leq \infty)$ with $\pnorm{x}{2}$ abbreviated as $\pnorm{x}{}$. Let $B_n(r;x)\equiv \{z \in \R^n: \pnorm{z-x}{}\leq r\}$ and recall $B_n(r)= B_n(r;0), B_n= B_n(1)$. For $x,y \in \R^n$, we write $\iprod{x}{y}\equiv \sum_{i=1}^n x_iy_i$. For a matrix $A \in \R^{m\times n}$ and a measurable set $T$, we follow the notation (\ref{def:GT}) to write $AT\equiv \{At: t \in T\}\subset \R^m$. For any subspace $L\subset \R^n$, let $\proj_L: \R^n\to \R^n$ be the orthogonal projection onto $L$ and $\proj_L^\perp\equiv \mathrm{id}-\proj_L$.

We use $C_{x}$ to denote a generic constant that depends only on $x$, whose numeric value may change from line to line unless otherwise specified. $a\lesssim_{x} b$ and $a\gtrsim_x b$ mean $a\leq C_x b$ and $a\geq C_x b$ respectively, and $a\asymp_x b$ means $a\lesssim_{x} b$ and $a\gtrsim_x b$. 

\section{Applications}\label{section:applications}

To demonstrate the broad applicability the approximate kinetic formulae in Section \ref{section:main_results}, we now present applications into three distinct problems arising from statistical learning, mathematical programming, and asymptotic geometric analysis. 

\subsection{Application I: Cone constrained MLEs in logistic regression}

Recall the setup of logistic regression in (\ref{eqn:glm}) with data $\{(X_i,Y_i)\in \R^n\times \{\pm 1\}: i \in [m]\}$, and the cone-constrained maximum likelihood estimator (MLE) $\hat{\beta}_K$, defined as any maximizer of the likelihood $\beta \mapsto \ell(\beta)$ over $\beta \in K$ (the likelihood function $\ell(\cdot)$ is defined in (\ref{eqn:logistic_likelihood})), whenever such a maximizer exists.

With the help of Theorem \ref{thm:approx_kinematics}, we provide a solution to Question \ref{question:logistic_reg} by proving a phase transition for the existence of cone-constrained MLEs under a standard Gaussian design for $\{X_i\}$ and the global null $\beta_0=0$: (i) if $m>2\delta(K)$, the cone-constrained MLE $\hat{\beta}_K$ exists with high probability; (ii) if $m<2\delta(K)$, cone-constrained MLEs do not exist, i.e., $\beta\mapsto \ell(\beta), \beta \in K$, attains its maximum at $\infty$ with high probability. The formal statement is as follows.

\begin{theorem}\label{thm:logistic_cone_MLE}
	Suppose that $\{(X_i,Y_i)\in \R^n\times \{\pm 1\}: i \in [m]\}$ are i.i.d. samples generated from the model (\ref{eqn:glm}). Suppose further that $X_1,\ldots,X_m$ are i.i.d. $\mathcal{N}(0,I_n)$ and $\beta_0=0$. There exists some universal constant $C>0$ such that the following statements hold for $t\geq 1$.
	\begin{enumerate}
		\item If $\sqrt{m}\leq \sqrt{2\delta(K)}-C\sqrt{t}$, then 
		\begin{align*}
		\Prob\big(\hbox{The cone constrained MLE $\hat{\beta}_K$ does not exist}\big)\geq 1-e^{-t},
		\end{align*}
		\item If $\sqrt{m}\geq \sqrt{2\delta(K)}+C\sqrt{t}$, then 
		\begin{align*}
		\Prob\big(\hbox{The cone constrained MLE $\hat{\beta}_K$ exists}\big)\geq 1-e^{-t}.
		\end{align*}
	\end{enumerate}
\end{theorem}

Here are two concrete examples:
\begin{itemize}
	\item Consider the non-negative orthant constraint $K=\R_{\geq 0}^n$. In this case, $\delta(K)=\delta(\R_{\geq 0}^n)=n/2$ (cf. \cite[Section 3.2]{amelunxen2014living}), and therefore the precise phase transition point for the existence of $\beta_{ \R_{\geq 0}^n}$ is given by $m\sim n$.
	\item Consider the monotone cone $K=\mathcal{M}_{\uparrow}\equiv \{\mu\in \R^n: \mu_1\leq \cdots \leq \mu_n\}$. In this case, $\delta(K)\sim \log n$ (as $n \to \infty$, cf. \cite[Section 3.5]{amelunxen2014living}), and therefore the precise phase transition point for the existence of $\beta_{\mathcal{M}_{\uparrow}}$ is given by $m\sim 2\log n$.
\end{itemize}

The key to link the above result to the approximate kinetic formulae in Section \ref{section:main_results} is the approximate equivalence
\begin{align}\label{ineq:logistic_MLE_equiv}
\big\{\hbox{The cone constrained MLE $\hat{\beta}_K$ does not exist}\big\}\approx \big\{GK\cap \R_{\geq 0}^m\neq \{0\}\big\}. 
\end{align}
See Proposition \ref{prop:logistic_exist} for a precise statement of (\ref{ineq:logistic_MLE_equiv}). 

Clearly, Theorem \ref{thm:logistic_cone_MLE} recovers the asymptotic phase transitions in \cite{cover1965geometrical,candes2020phase} for the non-/existence of the unconstrained MLE $\hat{\beta}_{\R^n}$ under the global null $\beta_0=0$, with an optimal probability estimate. Extending the phase transitions in Theorem \ref{thm:logistic_cone_MLE} for the constrained MLE $\hat{\beta}_K$ to general $\beta_0$'s is beyond the scope of a direct application of the results in Section \ref{section:main_results}, and will therefore be pursued elsewhere.

\subsection{Application II: Conic programs with random constraints}

Let $K\subset \R^n$ be a closed convex cone. For given $x \in \R^n$ and $A \in \R^{m\times n}$ and $b \in \R^m$, consider the following standard form of conic program (CP):
\begin{align}\label{def:conic_program}
&\max_{\mu \in \R^n}\, \iprod{x}{\mu}\quad \hbox{subject to}\quad A\mu = b,\, \mu \in K.
\end{align}
Several canonical examples of conic programs include:
\begin{itemize}
	\item linear programs with $K=\R_{\geq 0}^n$,
	\item second-order cone programs with $K=\{(x,y)\in \R^{(n-1)+1}: \pnorm{x}{}\leq y\}$, and 
	\item semi-definite programs with $K=\{X \in \R^{n_0\times n_0}: X=X^\top, X\geq 0\}\subset \R^{n_0\times n_0}\cong \R^n$, where $n=n_0^2$ for some integer $n_0 \in \mathbb{Z}_{\geq 1}$.
\end{itemize}
The readers are referred to \cite{bental2001lectures,boyd2004convex} for more details on CPs.

Here we are interested in the `average-case' feasibility of the CP~(\ref{def:conic_program}), where, following the standard terminology in mathematical programming, the CP~(\ref{def:conic_program}) is called \emph{feasible} if the constraint set $\{\mu \in K : A\mu = b\}$ is nonempty, and \emph{infeasible} if this constraint set is empty. More concretely, we will show that under a Gaussian random constraint $A = G$ and a deterministic choice of $x \in \partial B_n$, the CP~(\ref{def:conic_program}) exhibits different phase transitions depending on whether $b = 0$ or $b \neq 0$. The behavior can be characterized by the following geometric quantity associated with $K$ and $x$: 
\begin{align}\label{def:K_x}
K_x \equiv K \cap \{\mu \in \R^n : \iprod{x}{\mu} \geq 0\}.
\end{align}
Clearly, $K_x \subset K$ and $\delta(K_x) \leq \delta(K)$.

\begin{theorem}\label{thm:conic_program}
	Fix two positive integers $m,n\in \N$, $x \in \partial B_n$ and a non-trivial closed convex cone $K\subset \R^n$. Consider the conic programming (\ref{def:conic_program}) with the standard Gaussian random constraint $A=G \in \R^{m\times n}$. There exists some universal constant $C>0$ such that the following statements hold for $t\geq 1$:
	\begin{itemize}
		\item (\textbf{Homogeneous case $b=0$})
		\begin{enumerate}
			\item If $\sqrt{m}\geq \sqrt{\delta(K_x)}+C\sqrt{t}$, then 
			\begin{align*}
			\Prob\big(\hbox{The value of the CP (\ref{def:conic_program})}=0\big)\geq 1-e^{-t}.
			\end{align*}
			\item If $\sqrt{m}\leq \sqrt{\delta(K_x)}-C\sqrt{t}$, then 
			\begin{align*}
			\Prob\big(\hbox{The value of the CP (\ref{def:conic_program})}=\infty\big)\geq 1-e^{-t}.
			\end{align*}
		\end{enumerate}
		\item  (\textbf{In-homogeneous case $b\neq 0$})
		\begin{enumerate}
			\item If $\sqrt{m}\geq \sqrt{\delta(K)}+C\sqrt{t}$, then 
			\begin{align*}
			\Prob\big(\hbox{The CP (\ref{def:conic_program}) is infeasible}\big)\geq 1-e^{-t}.
			\end{align*}
			\item If $\sqrt{\delta(K_x)}+C\sqrt{t}\leq \sqrt{m}\leq \sqrt{\delta(K)}-C\sqrt{t}$, then 
			\begin{align*}
			\Prob\big(\hbox{The value of the CP (\ref{def:conic_program})}\leq 0\big)\geq 1-e^{-t}.
			\end{align*}
			\item If $\sqrt{m}\leq \sqrt{\delta(K_x)}-C\sqrt{t}$, then 
			\begin{align*}
			\Prob\big(\hbox{The value of the CP (\ref{def:conic_program})}=\infty\big)\geq 1-e^{-t}.
			\end{align*}
		\end{enumerate}
	\end{itemize}
\end{theorem}

The above theorem is closely related to \cite{amelunxen2015intrinsic} that studies the CP (\ref{def:conic_program}) with the same Gaussian random constraint $A=G$, and the additional assumption $x\sim \mathcal{N}(0,I_n)$. Under this additional Gaussian assumption on $x$, \cite[Theorems 3.1 and 3.2]{amelunxen2015intrinsic} provide exact kinematic formulae for (slight variations of) some probabilities in the above theorem, expressed via the intrinsic volumes associated with $K$. Using concentration of intrinsic volumes, \cite[Section 8]{amelunxen2014living} then shows that in the in-homogeneous case $b\neq 0$ with $x\sim \mathcal{N}(0,I_n)$, the phase transition reads:
\begin{enumerate}
	\item If $\sqrt{m}\geq \sqrt{\delta(K)}+C\sqrt{t}$, then the CP (\ref{def:conic_program}) is infeasible with probability at least $1-e^{-t}$.
	\item If $\sqrt{m}\leq \sqrt{\delta(K)}-C\sqrt{t}$, then the value of the CP (\ref{def:conic_program}) is $\infty$ with probability at least $1-e^{-t}$.
\end{enumerate}
Compared to the results for the in-homogeneous case of Theorem \ref{thm:conic_program}, it is immediately seen that a deterministic choice of $x \in \partial B_n$ significantly changes the phase transition behavior of the CP (\ref{def:conic_program}). In particular, a new intermediate regime $\delta(K_x)\ll m\ll \delta(K)$ appears in which the CP (\ref{def:conic_program}) is feasible but its value is finite.

We now mention the connection between the above Theorem \ref{thm:conic_program} and the approximate kinetic formulae in Section \ref{section:main_results}. The proof of Theorem \ref{thm:conic_program} heavily borrows from the idea used in the proof of Theorem \ref{thm:approx_kinematics} (cf. the estimates (\ref{ineq:upper_estimate_supp_fcn})-(\ref{ineq:lower_estimate_supp_fcn}) ahead). In particular, for any non-trivial closed convex cone $L \subset \R^m$, the following two-sided estimate holds:
\begin{align}\label{ineq:conic_program_reduction}
\max_{\mu \in K\cap B_n: G\mu \in L}\iprod{x}{\mu}\stackrel{\Prob}{\approx} \sup_{\mu \in K\cap B_n,\iprod{g}{\mu}\geq \sup\limits_{v\in L^\circ \cap \partial B_m}\iprod{h}{v} }\iprod{x}{\mu}.
\end{align}
See Proposition \ref{prop:support_fcn_cgmt} for a precise formulation of the above probabilistic equivalence. The above estimate also plays a crucial role in the proof of Theorem \ref{thm:approx_kinematics_preimage}.

\subsection{Application III: A local Gaussian Dvoretzky-Milman Theorem}\label{section:local_dm}

The  classical Gaussian Dvoretzky-Milman Theorem (\cite{dvoretzky1959theorem,dvoretzky1961some,milman1971new}), when adapted to our setting, says that a sufficiently low dimensional image of the possibly high dimension random set $G(K\cap B_n)$ is an approximately deterministic round ball with radius $\sqrt{\delta(K)}$. Its precise form reads as follows: there exists some universal constant $c>0$ such that for any $\epsilon \in (0,1/2)$, if $k$ is an integer with $1\leq k \leq c\epsilon^2 \delta(K)$, then with probability at least $1-\exp(-c\epsilon^2 \delta(K))$, 
\begin{align}\label{eqn:DM}
(1-\epsilon)\cdot B_k\big(\sqrt{\delta(K)}\big) \subset \Pi_{m\to k} \big(G K_n\big)\subset (1+\epsilon)\cdot B_k\big(\sqrt{\delta(K)}\big).
\end{align}
Here $K_n\equiv K\cap B_n$, and $\Pi_{m\to k}:\R^m\to \R^k$ is the natural projection map onto the first $k$-coordinates\footnote{Note that while we may simply write $\Pi_{m\to k}(G K_n)$ as $G_{k,n}K_n$ where $G_{k,n} \in \R^{k\times n}$ contains i.i.d. $\mathcal{N}(0,1)$. Here we opt to the form $\Pi_{m\to k}(G K_n)$ in (\ref{eqn:DM}), as it connects naturally to the `local' version we examine in Theorem \ref{thm:local_dm}.}, whereas $B_k(r)$ denotes an $\ell_2$ ball of $\R^k$ with radius $r$. For more general versions/other variations of (\ref{eqn:DM}), the readers are referred to e.g., \cite[Theorem 4.4]{pisier1989volume}, \cite[Theorem 9.2.6]{artstein2015asymptotic}, and \cite[Theorem 11.3.3]{vershynin2018high}. A tight version of the Gaussian Dvoretzky-Milman Theorem with optimal constants is recently obtained in \cite[Theorem 2.5]{han2024exact}.

Below we provide a `local' version of (\ref{eqn:DM}) by looking at low-dimensional images of the random set $L\cap GK_n$, where $L$ is a fixed subspace. In particular, the following theorem shows that the shape of $\Pi_{m\to k}(L\cap GK_n)$ is again approximately a deterministic round ball, but with a possible shrinkage in the radius. 

\begin{theorem}\label{thm:local_dm}
	Suppose that $K \subset \R^n$ is a non-trivial closed convex cone and $L\equiv\{x \in \R^m: x|_{(\ell:m]}=0\}$ for some integer $\ell$ with $\max\{1, (m-\delta(K))_+\}\leq \ell\leq m$. Further assume that there exists some $\tau \in (0,1/2)$ such that $(m-\ell)\leq (1-\tau)\cdot \delta(K)$. Then there exists some small constant $c=c(\tau)>0$ with the following statement holding true. For any $\epsilon \in (0,1/2)$, if $k$ is an integer with $1\leq k\leq \min\{\ell, c\cdot \mathscr{L}\big(\epsilon^2\delta(K)\big)\}$, then with probability at least $1-\exp\big(-c \epsilon^2 \delta(K)\big)$,
	\begin{align*}
	(1-\epsilon)\cdot B_k(\sqrt{\delta(K)-m+\ell}) \subset \Pi_{m\to k} \big(L\cap G K_n\big)\subset (1+\epsilon)\cdot B_k(\sqrt{\delta(K)-m+\ell}).
	\end{align*}
	Here $K_n$, $\Pi_{m\to k}:\R^m\to \R^k$ are as above, and $\mathscr{L}(x)\equiv x/\log\big(e \vee (1/x)\big)$. 
\end{theorem}

Clearly, the standard form of Gaussian Dvoretzky-Milman Theorem (\ref{eqn:DM}) can be recovered using the above theorem by setting $L=\R^m$ (with a logarithmically worse dependence of the dimension $k$ with respect to $\epsilon$). 

From a different perspective, Theorem \ref{thm:local_dm} above also provides further understanding to the stochastic geometry of $L\cap GK_n$ beyond the scope of Theorem \ref{thm:approx_kinematics}, at least when $L$ is a subspace. In particular, in the most interesting regime $\ell+\delta(K)\gg m$, Theorem \ref{thm:approx_kinematics}-(2) only proves that $L\cap GK_n \neq \{0\}$ is a non-trivial random set with high probability. Here Theorem \ref{thm:local_dm} provides a substantially refined, almost deterministic low dimensional description of the stochastic geometry of $L\cap G K_n$: with high probability, $\Pi_{m\to k}(L\cap GK_n)\approx B_k(\sqrt{\delta(K)-m+\ell})$.

The proof of Theorem \ref{thm:local_dm} is based on the following improved two-sided estimate for (\ref{ineq:upper_estimate_supp_fcn})-(\ref{ineq:lower_estimate_supp_fcn}) used in the proof of Theorem \ref{thm:approx_kinematics} that exploits the additional subspace structure of $L$:
\begin{align}\label{ineq:dm_two_sided_estimate}
\mathsf{h}_{L\cap GK_n}(x)\stackrel{\Prob}{\approx}  \bigg\{\bigg(\sup_{\mu \in K\cap B_n}\iprod{g}{\mu}\bigg)^2-\bigg(\sup_{v \in L^\circ\cap B_m} \iprod{h}{v}\bigg)^2\bigg\}_+^{1/2}.
\end{align}
In particular, we prove in Propositions \ref{prop:DM_upper} and \ref{prop:DM_lower} a precise version of (\ref{ineq:dm_two_sided_estimate}) with optimal Gaussian tail estimates for each and every $x\in L\cap \partial B_m$. The claim of Theorem \ref{thm:local_dm} then follows by a standard $\epsilon$-net argument.

\section{Overview of the proof and preliminaries}\label{section:preliminary}


\subsection{Brief review of proof method in \cite{amelunxen2014living}}\label{subsection:ALMT_proof_technique}
Before detailing the proof techniques of our approximate kinetic formulae above, let us first review the method of proof in \cite{amelunxen2014living} for (\ref{eqn:ALMT}) with a similar appearance. The proof of (\ref{eqn:ALMT}) in \cite{amelunxen2014living} is based on the \emph{exact kinematic formula} (cf. \cite[Theorem 6.5.6]{schneider2008stochastic}):
\begin{align}\label{eqn:exact_kinematic_formula}
\Prob\big(K\cap O_n L\neq \{0\}\big)=\sum_{i=0}^n \big(1+(-1)^{i+1}\big)\sum_{j=i}^n \mathscr{V}_j(K)\cdot \mathscr{V}_{n+i-j}(L).
\end{align}
Here $\{\mathscr{V}_j(K)\}\subset [0,1]$'s are the so-called \emph{intrinsic volumes} of the cone $K$ satisfying $\sum_{j=0}^n \mathscr{V}_j(K)=1$; see \cite[Section 5]{amelunxen2014living} for more background knowledge of the notion of intrinsic volumes.

The key step in \cite{amelunxen2014living} to derive (\ref{eqn:ALMT}) from the exact formula (\ref{eqn:exact_kinematic_formula}) is to prove `concentration' of $\{\mathscr{V}_j(K)\}$ viewed as a probability measure on $\{0,\ldots,n\}$. 
To illustrate the main idea, suppose further $L$ is a subspace of dimension $\ell$. Then $\mathscr{V}_j(L)=\bm{1}_{j=\ell}$, and  (\ref{eqn:exact_kinematic_formula}) reduces to the \emph{Crofton formula}:
\begin{align}\label{eqn:crofton_formula}
\Prob\big(K\cap O_n L\neq \{0\}\big)= 2\sum_{\substack{j=n-\ell+1,\\ j-(n-\ell+1)\textrm{ is even}}}^n \mathscr{V}_j(K)\equiv 2\mathscr{H}_{n-\ell+1}(K).
\end{align}
As a consequence of (i) the concentration of $\{\mathscr{V}_j(K)\}$ around its mean $\delta(K)$ (cf. \cite[Theorem 6.1]{amelunxen2014living}, \cite[Corollary 4.10]{mccoy2014from}), and (ii) the interlacing property $\sum_{j=n-\ell+1}^n \mathscr{V}_j(K) \leq 2\mathscr{H}_{n-\ell+1}(K)\leq \sum_{j=n-\ell}^n \mathscr{V}_j(K)$ (cf. \cite[Proposition 5.9]{amelunxen2014living}), the right hand side of the above display (\ref{eqn:crofton_formula}) is either close to $0$ or $1$, according to whether $\ell+\delta(K)\ll n$ or $\ell+\delta(K)\gg n$.

As have been clear now, the conic integral geometry approach of \cite{amelunxen2014living} to prove (\ref{eqn:ALMT}) relies heavily on the existence of the exact formula (\ref{eqn:exact_kinematic_formula}); see also \cite{goldstein2017gaussian}. On the other hand, such exact kinematic formulae are proved by exploiting, in an essential way, the fact that $O_n$ induces the unique invariant probability (Haar) measure on $\mathrm{O}(n)$ (cf. \cite[Section 5]{schneider2008stochastic}), and therefore unfortunately do not admit a direct extension to the probability $\Prob(L\cap GK\neq \{0\})$ of interest in Theorem \ref{thm:approx_kinematics}, where $L,K$ are cones of possibly different dimensions. Moreover, such kinematic formulae are also unavailable for general closed convex sets in the setting of Theorem \ref{thm:generalized_gordon}. See \cite[Remark 2.3]{amelunxen2014living} for some related discussions.

\subsection{Our method of proof}
Here we take a random process approach to prove the approximate kinetic formulae in Section \ref{section:main_results}. Our method of proof is based on a two-sided version of Gordon's Gaussian min-max comparison inequality \cite{gordon1985some,gordon1988milman}, known as the convex Gaussian min-max theorem in the statistical learning and information theory literature, cf. \cite{stojnic2013framework,thrampoulidis2018precise}. The following version is taken from \cite[Corollary G.1]{miolane2021distribution}.

\begin{theorem}[Convex Gaussian Min-Max Theorem]\label{thm:CGMT}
	Suppose $D_u \in \R^{n_1+n_2}, D_v \in \R^{m_1+m_2}$ are compact sets, and $Q: D_u\times D_v \to \R$ is continuous. Let $G=(G_{ij})_{i \in [n_1],j\in[m_1]}$ with $G_{ij}$'s i.i.d. $\mathcal{N}(0,1)$, and $g \sim \mathcal{N}(0,I_{n_1})$, $h \sim \mathcal{N}(0,I_{m_1})$ be independent Gaussian vectors. For $u \in \R^{n_1+n_2}, v \in \R^{m_1+m_2}$, write $u_1\equiv u_{[n_1]}\in \R^{n_1}, v_1\equiv v_{[m_1]} \in \R^{m_1}$. Define
	\begin{align*}
	\Phi^{\textrm{p}} (G)& = \max_{u \in D_u}\min_{v \in D_v} \Big( u_1^\top G v_1 + Q(u,v)\Big), \nonumber\\
	\Phi^{\textrm{a}}(g,h)& = \max_{u \in D_u}\min_{v \in D_v} \Big(\pnorm{v_1}{} g^\top u_1 + \pnorm{u_1}{} h^\top v_1+ Q(u,v)\Big).
	\end{align*}
	Then the following hold.
	\begin{enumerate}
		\item For all $t \in \R$, 
		\begin{align*}
		\Prob\big(\Phi^{\textrm{p}} (G)\geq t\big)\leq 2 \Prob\big(\Phi^{\textrm{a}}(g,h)\geq t\big).
		\end{align*}
		\item If $(u,v)\mapsto u_1^\top G v_1+ Q(u,v)$ satisfies the conditions of Sion's min-max theorem for  the pair $(D_u,D_v)$ a.s. (for instance, $D_u,D_v$ are convex, and $Q$ is concave-convex), then for any $t \in \R$,
		\begin{align*}
		\Prob\big(\Phi^{\textrm{p}} (G)\leq t\big)\leq 2 \Prob\big(\Phi^{\textrm{a}}(g,h)\leq t\big).
		\end{align*}
	\end{enumerate}
\end{theorem}
A significant line of recent work explores Theorem \ref{thm:CGMT} in the context of characterizing the $\ell_2$-type errors or distributions of various statistical estimators~\cite{thrampoulidis2015recovering,thrampoulidis2018precise,miolane2021distribution,celentano2023lasso,montanari2023generalization,han2023noisy,han2023universality,han2023distribution}. The main idea underlying this line of work lies in providing exact characterizations for the cost optimum of the associated min–max optimization problems involving $G$, and then linking the statistical estimators to the precise value of this cost optimum by quantifying the suboptimality gap.

In our setting, while the comparison inequality is known to yield one (first) direction of~(\ref{eqn:ALMT}) in several special cases~\cite{rudelson2008sparse,stojnic2009various,oymak2010new,chandrasekaran2012convex}, its utility in producing the second direction of the phase transition remains unclear. It is now well understood that such a difficulty can be circumvented for randomly rotated cones by the proof method developed in~\cite{amelunxen2014living}, as detailed above.

To establish a precise link between the phase transition behavior of randomly projected cones and the objective of the associated min–max optimization problem, the key technical point that makes our random process method fully effective in the general formulation of Theorem~\ref{thm:approx_kinematics} is to work with \emph{sufficiently tight upper and lower bounds, instead of exact characterizations}, for the cost optimum of two different random optimization problems that certify the two directions of the same phase transition point. Conceptually, this is viable because we are primarily interested in the nontriviality of the objective function rather than its exact value.



Specifically, we exploit the Gaussian comparison principle in its full strength, to produce optimal probabilistic upper and lower estimates for two different representation of the \emph{support function} of suitable compact versions of $L\cap G K$. Recall for a generic closed convex set $K_0\subset \R^n$, its support function is defined as
\begin{align}\label{def:support_fcn}
\mathsf{h}_{K_0}(x)\equiv \max_{\mu \in K_0}\iprod{x}{\mu},\quad x \in \R^n.
\end{align}
Basic properties of the support function can be found in, e.g., \cite[Chapter 1]{schneider2014convex}. With $B_n,\partial B_n$ denoting the unit ball and sphere in $\R^n$, and $K_n\equiv K\cap B_n$, the (uniform) upper estimate takes the form
\begin{align}\label{ineq:upper_estimate_supp_fcn}
\sup_{x \in L\cap \partial B_m}\mathsf{h}_{L\cap G K_n}(x)&=\sup_{ \substack{x \in L\cap \partial B_m,\\ v \in L, \mu \in K_n}} \inf_{w \in \R^m} \Big\{\iprod{x}{v}+\iprod{w}{G\mu-v}\Big\}\nonumber\\
&\stackrel{\Prob}{\leq} \sup_{ \substack{x \in L\cap \partial B_m,\\ \mu \in K_n} } \sup_{\substack{v\in L,\\ \iprod{g}{\mu}\geq \pnorm{\,\pnorm{\mu}{}h-v}{}}
}\iprod{x}{v}.
\end{align}
The lower estimate, which holds for individual $x \in L\cap \partial B_m$'s, takes a different form: there exists some universal constant $c>0$, 
\begin{align}\label{ineq:lower_estimate_supp_fcn}
\mathsf{h}_{L\cap G K_n}(x)&=\sup_{\mu \in K_n}\inf_{v \in L^\circ} \Big\{ \iprod{x}{G\mu}-\iprod{v}{G\mu} \Big\}\nonumber\\
&\stackrel{\Prob}{\geq} c\cdot \bigg\{\bigg(\sup_{\mu \in K\cap B_n}\iprod{g}{\mu}\bigg)^2-\bigg(\sup_{v \in L^\circ\cap B_m} \iprod{h}{v}\bigg)^2\bigg\}_+^{1/2}.
\end{align}
See the proofs of Propositions \ref{prop:support_fcn_proj_upper} and \ref{prop:support_fcn_proj_lower} for precise versions of (\ref{ineq:upper_estimate_supp_fcn})-(\ref{ineq:lower_estimate_supp_fcn}). The relevance of the estimates (\ref{ineq:upper_estimate_supp_fcn}) and (\ref{ineq:lower_estimate_supp_fcn}) to Theorem \ref{thm:approx_kinematics} can now be easily seen:
\begin{itemize}
	\item If $\delta(K)+\delta(L)\ll m$, then with high probability the feasible set for the supremum over $v$ in (\ref{ineq:upper_estimate_supp_fcn}) is $\{0\}$, and therefore $\sup_{x \in L\cap \partial B_m}\mathsf{h}_{L\cap GK_n}(x)=0$ with high probability. This is equivalent to $L\cap GK=\{0\}$. 
	\item If $\delta(K)+\delta(L)\gg m$, then for each $x \in L\cap \partial B_m$, with high probability $\mathsf{h}_{L\cap GK_n}(x)>0$, which implies $L\cap GK\neq \{0\}$. 
\end{itemize}
As mentioned above, the utility of the comparison principle in proving the first direction (\ref{ineq:upper_estimate_supp_fcn}) has been recognized in \cite{rudelson2008sparse,stojnic2009various,oymak2010new,chandrasekaran2012convex} in several special cases of interest in applications, and only requires the (almost) original form of Gordon's comparison inequality \cite{gordon1985some,gordon1988milman}. Here we work out the most general form in this direction. The major contribution of our technical development therefore rests in proving the more difficult, second direction (\ref{ineq:lower_estimate_supp_fcn}). Indeed, as will be evident in the proof of Proposition \ref{prop:support_fcn_proj_lower}, the proof of (\ref{ineq:lower_estimate_supp_fcn}) is much more involved that the `easier direction' (\ref{ineq:upper_estimate_supp_fcn}). 

\subsection{Some basic convex geometry}

For any closed convex cone $K\subset\R^n$, its polar cone is defined as
\begin{align}\label{def:polar_cone}
K^\circ\equiv \left\{v\in\R^n: \iprod{v}{\mu}\leq 0, \text{ for all }\mu\in K\right\}.
\end{align}
The following lemma provides a useful characterization of a closed convex cone $K$ via its polar cone $K^\circ$.
\begin{lemma}\label{lem:cone_polarity}
	Let $K$ be a closed convex cone. The following hold:
	\begin{align*}
	&\mu \in K \Leftrightarrow \sup_{v \in K^\circ} \iprod{v}{\mu}=0\Leftrightarrow \inf_{v \in -K^\circ} \iprod{v}{\mu}=0,\nonumber\\
	& \mu \notin K \Leftrightarrow \sup_{v \in K^\circ} \iprod{v}{\mu}=\infty\Leftrightarrow \inf_{v \in -K^\circ} \iprod{v}{\mu}=-\infty.
	\end{align*}
\end{lemma}
\begin{proof}

Suppose $\mu \in K$. By definition of $K^\circ$, $\sup_{v \in K^\circ}\iprod{v}{\mu}\leq 0$ with equality achieved for $v=0$. This proves the direction $\mu \in K\Rightarrow \sup_{v \in K^\circ} \iprod{v}{\mu}=0$. 

For the other direction, as for any $\mu \in \R^n$, $\sup_{v \in K^\circ} \iprod{v}{\mu} \in \{0,\infty\}$, it remains to prove that $\mu \notin K \Rightarrow \sup_{v \in K^\circ} \iprod{v}{\mu}  = \infty$. To see this, suppose $\mu \in K$ is such that $\sup_{v \in K^\circ} \iprod{v}{\mu}$ is finite. This means $\sup_{v \in K^\circ} \iprod{v}{\mu}=0$, or equivalently, $\iprod{v}{\mu}\leq 0$ for all $v \in K^\circ$. Consequently, $\mu \in (K^\circ)^\circ =K$, where the last identity follows by e.g. \cite[Theorem 14.1]{rockafellar1997convex}.
\end{proof}

Another important property of the polar cone is the following orthogonal decomposition known as Moreau's theorem \cite[Theorem 31.5]{rockafellar1997convex}.
\begin{lemma}\label{lem:moreau_thm}
For any $v\in\R^n$, we have the orthogonal decomposition \begin{align}
v = \Pi_K(v) + \Pi_{K^\circ}(v) \, \hbox{ with }\, \iprod{\Pi_K(v)}{\Pi_{K^\circ}(v)} = 0.
\end{align}
\end{lemma}

The notion of statistical dimension defined in (\ref{def:stat_dim}) is intrinsically related to the so-called Gaussian width. A formal definition is given as follows. For a compact convex set $K_0\subset \R^n$, we define its \emph{Gaussian width} $\gw(K_0)$ by
\begin{align}\label{def:Gauss_width}
\gw(K_0)\equiv \E \sup_{\mu \in K_0}\iprod{g}{\mu},\quad g\sim \mathcal{N}(0,I_n).
\end{align}

We shall frequently use the following properties of the statistical dimension (\ref{def:stat_dim}) and Gaussian width (\ref{def:Gauss_width}) in the proofs ahead.
\begin{proposition}\label{prop:stat_dim}
Let $K\subset \R^n$ be a closed convex cone.
\begin{enumerate}
	\item $\delta(K)=\E\big(\sup_{\mu \in K\cap B_n}\iprod{g}{\mu}\big)^2$, where $g\sim \mathcal{N}(0,I_n)$.
	\item $\delta(K)+\delta(K^\circ)=n$. 
	\item For nontrivial $K\neq \{0\}$, $\gw(K\cap \partial B_n) \leq \gw(K\cap B_n)\leq \gw(K\cap \partial B_n)+1$. 
	\item $\gw(K\cap B_n)\leq \sqrt{\delta(K)}\leq \gw(K\cap B_n)+1 $.
\end{enumerate}
\end{proposition}
\begin{proof}
\noindent (1)-(2). The claims follow from \cite[Proposition 3.1]{amelunxen2014living}.

\noindent (3). Note that 
\begin{align*}
\sup_{\mu \in K\cap \partial B_n}\iprod{g}{\mu}&\leq \sup_{\mu \in K\cap B_n}\iprod{g}{\mu} = \sup_{0\leq r\leq 1} \sup_{\mu \in K, \pnorm{\mu}{}=r}\iprod{g}{\mu} \\
&= \sup_{0\leq r\leq 1} \bigg(r\cdot \sup_{\mu \in K\cap \partial B_n}\iprod{g}{\mu}\bigg) = 0 \vee \sup_{\mu \in K\cap \partial B_n}\iprod{g}{\mu}. 
\end{align*}
On the other hand, using the $1$-Lipschitz property of $g\mapsto \sup_{\mu \in K\cap \partial B_n}\iprod{g}{\mu}$, we obtain by Gaussian-Poincar\'e inequality that
\begin{align*}
\var\bigg(\sup_{\mu \in K\cap \partial B_n}\iprod{g}{\mu}\bigg)= \E\bigg(\sup_{\mu \in K\cap \partial B_n}\iprod{g}{\mu}\bigg)^2-\gw^2(K\cap \partial B_n)\leq 1. 
\end{align*}
Consequently, using the above two inequalities and the fact that $\gw(K\cap \partial B_n)\geq 0$,
\begin{align*}
\gw(K\cap \partial B_n)&\leq \gw(K\cap B_n)\leq \E \bigg[0 \vee \sup_{\mu \in K\cap \partial B_n}\iprod{g}{\mu}\bigg]\\
&\leq \E^{1/2}\bigg(\sup_{\mu \in K\cap \partial B_n}\iprod{g}{\mu}\bigg)^2 \leq \gw(K\cap \partial B_n)+1. 
\end{align*}
\noindent (4). We may conclude the claim by using the representation in (1) and Gaussian-Poincar\'e inequality similarly as in the proof of (2). Details are omitted. 
\end{proof}

\subsection{Further Gaussian process tools}


The following Gaussian concentration inequality will be used frequently; its statement is taken from e.g., \cite[Section 1.1]{ledoux2013probability}, or \cite[Theorem 2.2.6]{gine2015mathematical}.
\begin{theorem}\label{thm:gaussian_concentration}
Let $g\sim \mathcal{N}(0,I_n)$ and $F:\R^n\to \R$ be a $1$-Lipschitz map. Then 
\begin{align*}
\Prob\big(\abs{F(g)-\mathrm{Med}(F(g))}\geq t\big)\leq 2e^{-t^2/2},\quad t\geq 0.
\end{align*}
Here $\mathrm{Med}(F(g))$ denotes the median of $F(g)$. 
\end{theorem}

We will mostly use the following form of the Gaussian concentration inequality in the proofs below.
\begin{proposition}\label{prop:gaussian_conc_generic}
There exists some universal constant $C>0$ such that for any compact set $K_0\subset \R^n$ and all $t\geq 1$, with probability at least $1-e^{-t}$,
\begin{align*}
\bigabs{\sup_{\mu \in K_0}\iprod{g}{\mu}- \E \sup_{\mu \in K_0}\iprod{g}{\mu}  }\leq C \cdot\sup_{\mu \in K_0}\pnorm{\mu}{}\cdot  \sqrt{t}.
\end{align*}
\end{proposition}
\begin{proof}
By rescaling, we may assume without loss of generality that $K_0\subset B_n$. Let $F(g)\equiv \sup_{\mu \in K_0}\iprod{g}{\mu}$. It is easy to see that $\abs{F(g)-F(g')}\leq \pnorm{g-g'}{}$ holds for any pair $g,g' \in \R^n$, so $F$ is $1$-Lipschitz. Using the Gaussian concentration inequality in Theorem \ref{thm:gaussian_concentration}, for any $t>0$, with probability at least $1-2e^{-t/2}$ we have 
\begin{align*}
\abs{F(g)-\mathrm{Med}(F(g))}\leq \sqrt{t}. 
\end{align*}
Integrating the tail, we obtain
\begin{align*}
\abs{\E F(g)-\mathrm{Med}(F(g))} \leq C
\end{align*}
for some universal constant $C>0$. Combining the above two displays, for any $t>0$, with probability at least $1-2e^{-t/2}$ we have 
\begin{align*}
\abs{F(g)- \E F(g)}\leq \sqrt{t}+C. 
\end{align*}
Now adjusting the constants to conclude. 
\end{proof}

\subsection{A continuity lemma}

\begin{lemma}\label{lem:sup_set_conv}
	Let $f: \R^n\to \R$ be a continuous function.
	\begin{enumerate}
		\item Suppose $\{S_k: k\in \N\}$ is a sequence of non-increasing compact sets in $\R^n$. Then with $S_\infty\equiv \cap_{k\in \N} S_k$, we have $\sup_{x \in S_k} f(x)\downarrow \sup_{x \in S_\infty} f(x)$.
		\item Suppose $\{T_k: k\in \N\}$ is a sequence of non-decreasing compact sets in $\R^n$ such that $T_k\subset T$ for some compact set $T\subset \R^n$ for all $k \in \N$. Then with $T_\infty\equiv \mathrm{cl}\big(\cup_{k\in \N} T_k\big)$, we have $\sup_{x \in T_k} f(x)\uparrow \sup_{x \in T_\infty} f(x)$.
	\end{enumerate}

\end{lemma}
\begin{proof}
	\noindent (1). 	As $f$ is continuous and $S_k$ is compact, we may find some $x_k \in S_k$ such that $f(x_k)=\sup_{x \in S_k}f(x)$. As the sequence $\{x_k: k\in \N\}$ is contained in some compact set, say, $S_1$, we may assume without loss generality that $x_k \to x_\infty$ for some $x_\infty \in S_1$. As all subsequential limits of $\{S_k\}$ are contained in $S_\infty$ due to the monotonicity of $S_k$, we have  $x_\infty \in S_\infty$. This means $\sup_{x \in S_k} f(x)=f(x_k)\to f(x_\infty)\leq \sup_{x \in S_\infty} f(x)$, i.e., $\limsup_k \sup_{x \in S_k} f(x)\leq \sup_{x \in S_\infty} f(x)$. The other direction is trivial. 
	
	\noindent (2). 	For any $y_\infty \in \cup_{k\in \N} T_k$, there exists some $k \in \N$ such that $y_\infty \in T_k$. This means $\sup_{y \in T_k} f(y)\geq f(y_\infty)$. As $\{T_k\}$ is non-decreasing, we have $\liminf_k \sup_{y \in T_k} f(y)\geq f(y_\infty)$. Maximizing over $y_\infty \in \cup_{k\in \N} T_k$ and using the continuity of $f$, we have $
	\liminf_k \sup_{y \in T_k} f(y)\geq \sup_{y_\infty \in \cup_{k\in \N} T_k} f(y_\infty)=\sup_{y \in T_\infty} f(y)$. 
	The other direction is again trivial. 
\end{proof}

\section{Proofs of Theorem \ref{thm:approx_kinematics} and Corollary \ref{cor:gordon_escape}}\label{section:proof_approx_kinematics}

\subsection{Proof of Theorem \ref{thm:approx_kinematics}: small $\delta(K)$ regime}

The following proposition shows that for `small' values of $\delta(K)$, the support function of $L\cap G(K\cap B_n)$ will be  trivial \emph{uniformly in all directions} with high probability. 

\begin{proposition}\label{prop:support_fcn_proj_upper}
Suppose that $K \subset \R^n$ and $L\subset \R^m$ are closed convex cones with $K\neq \{0\}, L\notin \{\{0\},\R^m\}$. Fix $x \in L\cap \partial B_m$. Then there exists some universal constant $C>0$ such that for any $t\geq 1$,
\begin{align*}
\Prob\Big(\sup_{x \in L\cap \partial B_m}\mathsf{h}_{L\cap G K_n}(x)\neq 0\Big)\leq \bm{1}\Big(\sqrt{\delta(K)}> \sqrt{\delta(L^\circ)}-C\sqrt{t}\Big)+e^{-t}.
\end{align*}
Here $K_n\equiv K\cap B_n$.
\end{proposition}

\begin{proof}
We shall drop the subscript in $\mathsf{h}_{L\cap G K_n}$ for notational convenience in the proof. Let $E(R)\equiv \{\pnorm{G}{\op}\leq R\}$. Note that
\begin{align}\label{ineq:support_fcn_proj_sup_upper_0}
\mathsf{h}(x) = \sup_{\mu \in K_n, G\mu \in L}\iprod{x}{G\mu}=\sup_{v \in L, \mu \in K_n} \inf_{w \in \R^m} \Big\{\iprod{x}{v}+\iprod{w}{G\mu-v}\Big\},
\end{align}
and on the event $E(R)$, the supremum above over $v \in L$ can be restricted to $v \in L\cap B_m(R)$ uniformly for all $x \in L\cap \partial B_m$. This means on  the event $E(R)$,
\begin{align*}
\sup_{x \in L\cap \partial B_m} \mathsf{h}(x) = \sup_{\substack{x \in L\cap \partial B_m, \\v \in L \cap B_m(R), \mu \in K_n}} \inf_{w \in \R^m} \Big\{\iprod{x}{v}+\iprod{w}{G\mu-v}\Big\},
\end{align*}
By the one-sided Gaussian min-max theorem (cf. Theorem \ref{thm:CGMT}-(1)), for any $z \in \R$ and $R,R_1>0$, we have
\begin{align}\label{ineq:support_fcn_proj_sup_upper_1}
&\Prob\bigg(\sup_{x \in L\cap \partial B_m} \mathsf{h}(x) >z\bigg)-\Prob(E(R)^c)\\
&\leq \Prob\bigg(\sup_{\substack{x \in L\cap \partial B_m, \\v \in L \cap B_m(R), \mu \in K_n}} \inf_{w \in B_m(R_1)} \Big\{\iprod{x}{v}+\iprod{w}{G\mu-v}\Big\}>z\bigg)\nonumber\\
&\leq 2\Prob\bigg(\sup_{\substack{x \in L\cap \partial B_m, \\v \in L \cap B_m(R), \mu \in K_n}} \inf_{w \in B_m(R_1)} \Big\{\iprod{x}{v}-\iprod{w}{v}+\pnorm{\mu}{}\iprod{h}{w}+\pnorm{w}{}\iprod{g}{\mu}\Big\}>z\bigg)\nonumber.
\end{align}
On the other hand, 
\begin{align}\label{ineq:support_fcn_proj_sup_upper_2}
&\sup_{\substack{x \in L\cap \partial B_m, \\v \in L \cap B_m(R), \mu \in K_n}} \inf_{w \in B_m(R_1)} \Big\{\iprod{x}{v}-\iprod{w}{v}+\pnorm{\mu}{}\iprod{h}{w}+\pnorm{w}{}\iprod{g}{\mu}\Big\}\nonumber\\
& = \sup_{\substack{x \in L\cap \partial B_m, \\v \in L \cap B_m(R), \mu \in K_n}} \inf_{0\leq \beta\leq R_1} \Big\{ \beta\cdot  \Big(\iprod{g}{\mu}-\bigpnorm{ \pnorm{\mu}{}h-v }{}\Big)+\iprod{x}{v}\Big\}.
\end{align}
As the above max-min problem must take non-negative cost optimum, and $\iprod{x}{v}\leq R$ for $x \in L\cap \partial B_m,v \in L \cap B_m(R)$, it follows that any pair of $(\mu,v) $ such that $\iprod{g}{\mu}-\pnorm{\, \pnorm{\mu}{}h-v }{}<-R/R_1$ is not a feasible maximizer of (\ref{ineq:support_fcn_proj_sup_upper_2}). This means that the supremum of (\ref{ineq:support_fcn_proj_sup_upper_2}) can be further restricted as follows:
\begin{align}\label{ineq:support_fcn_proj_sup_upper_3}
(\ref{ineq:support_fcn_proj_sup_upper_2})&=\sup_{\substack{x \in L\cap \partial B_m, \\v \in L \cap B_m(R), \mu \in K_n,\\ \iprod{g}{\mu}-\pnorm{ \,\pnorm{\mu}{}h-v }{}\geq-R/R_1}} \inf_{0\leq \beta\leq R_1} \Big\{ \beta\cdot  \Big(\iprod{g}{\mu}-\bigpnorm{ \pnorm{\mu}{}h-v }{}\Big)+\iprod{x}{v}\Big\}\nonumber\\
& \stackrel{(\ast)}{\leq} \sup_{\substack{x \in L\cap \partial B_m, \\v \in L \cap B_m(R), \mu \in K_n}} \iprod{x}{v}\,\, \hbox{ subject to } \iprod{g}{\mu}-\bigpnorm{ \pnorm{\mu}{}h-v }{}\geq -R/R_1.
\end{align}
The inequality $(\ast)$ in the above display follows by taking $\beta=0$ in the infimum. Using Moreau's theorem (cf. Lemma \ref{lem:moreau_thm}), we may write $h=\Pi_L(h)+\Pi_{L^\circ}(h)$ with $\iprod{\Pi_L(h)}{\Pi_{L^\circ}(h)}=0$, so
\begin{align*}
\bigpnorm{ \pnorm{\mu}{}h-v }{}^2&= \bigpnorm{ \pnorm{\mu}{}\Pi_L(h)-v+ \pnorm{\mu}{}\Pi_{L^\circ}(h) }{}^2\\
& = \bigpnorm{ \pnorm{\mu}{}\Pi_L(h)-v}{}^2+ \pnorm{\mu}{}^2\cdot \pnorm{\Pi_{L^\circ}(h)}{}^2+2\iprod{ \pnorm{\mu}{}\Pi_L(h)-v}{\pnorm{\mu}{}\Pi_{L^\circ}(h) }\\
&\stackrel{(\ast\ast)}{\geq} \bigpnorm{ \pnorm{\mu}{}\Pi_L(h)-v}{}^2+ \pnorm{\mu}{}^2\cdot \pnorm{\Pi_{L^\circ}(h)}{}^2,
\end{align*}
where in the inequality $(\ast\ast)$ we used that $\iprod{-v}{\Pi_{L^\circ}(h)}\geq 0$. Combining the above inequality with (\ref{ineq:support_fcn_proj_sup_upper_1})-(\ref{ineq:support_fcn_proj_sup_upper_3}), we then have
\begin{align}\label{ineq:support_fcn_proj_sup_upper_4}
\Prob\bigg(\sup_{x \in L\cap \partial B_m} \mathsf{h}(x) >z\bigg)
 \leq 2\Prob\bigg(\sup_{\substack{x \in L\cap \partial B_m, v \in \mathscr{Q}(R,R_1) } } \iprod{x}{v}>z\bigg)+\Prob(E(R)^c),
\end{align}
where
\begin{align*}
\mathscr{Q}(R,R_1)&\equiv \mathscr{Q}(R,R_1;g,h)\equiv \bigg\{v\in L\cap B_m(R):\exists \mu \in K_n,\\
&\qquad\qquad \hbox{ s.t. }\iprod{g}{\mu}\geq \sqrt{\bigpnorm{ \pnorm{\mu}{}\Pi_L(h)-v}{}^2+ \pnorm{\mu}{}^2\pnorm{\Pi_{L^\circ}(h)}{}^2}-\frac{R}{R_1}\bigg\}.
\end{align*}
Clearly $\mathscr{Q}(R,R_1)$ is compact. Note that $R_1 \mapsto \mathscr{Q}(R,R_1)$ is non-increasing as $R_1\uparrow \infty$, so with $\overline{\mathscr{Q}}(R)\equiv \cap_{R_1>0}\mathscr{Q}(R,R_1)\subset \mathscr{Q}(R,\infty)$,  by using Lemma \ref{lem:sup_set_conv} we have 
\begin{align*}
\sup_{\substack{x \in L\cap \partial B_m, v \in \mathscr{Q}(R,R_1) } } \iprod{x}{v}\downarrow \sup_{\substack{x \in L\cap \partial B_m, v \in \overline{\mathscr{Q}}(R) } } \iprod{x}{v}\leq \sup_{\substack{x \in L\cap \partial B_m, v \in \mathscr{Q}(R,\infty) } } \iprod{x}{v}.
\end{align*}
Now taking the limit $R_1\uparrow \infty$ on the right hand side of (\ref{ineq:support_fcn_proj_sup_upper_4}), we have for all $z\in \R$ and $R>0$, 
\begin{align}\label{ineq:support_fcn_proj_sup_upper_5}
&\Prob\bigg(\sup_{x \in L\cap \partial B_m} \mathsf{h}(x) >z\bigg)
\leq 2\Prob\bigg(\sup_{\substack{x \in L\cap \partial B_m, v \in \mathscr{Q}(R,\infty) } } \iprod{x}{v}\geq z\bigg)+\Prob(E(R)^c).
\end{align}
By Gaussian concentration in Proposition \ref{prop:gaussian_conc_generic}, there exists some universal constant $C_0>0$ such that for all $t\geq 1$, on an event $E_0(t)$ with probability at least $1-e^{-t}$,
\begin{align}\label{def:event_E0_conc_gw}
\bigabs{\sup_{\mu \in K\cap \partial B_n}\iprod{g}{\mu}-\gw(K\cap \partial B_n)}\vee \bigabs{ \pnorm{\Pi_{L^\circ}(h)}{}- \gw(L^\circ\cap \partial B_m)}\leq C_0\sqrt{t}.
\end{align}
Here, for the concentration claim concerning $\pnorm{\Pi_{L^\circ}(h)}{}$, we used (i) the simple fact that $\pnorm{\Pi_{L^\circ}(h)}{}=\sup_{v \in L^\circ \cap B_m}\iprod{h}{v}$ (see, e.g., \cite[Appendix B.4]{amelunxen2014living} for a proof of this fact), and (ii) Proposition \ref{prop:stat_dim}-(3) to replace $\gw(L^\circ\cap B_m)$ with $\gw(L^\circ\cap \partial B_m)$ using the condition that $L^\circ$ is non-trivial.

We shall now show that on the event $E_0(t)$ defined in (\ref{def:event_E0_conc_gw}) above, 
\begin{align}\label{ineq:support_fcn_proj_sup_upper_6}
\gw(K\cap \partial B_n)<\gw(L\cap \partial B_m)-2C_0\sqrt{t}\,\, \Rightarrow \,\, \mathscr{Q}(R,\infty)=\{0\}.
\end{align}
To see this, if there exists some $v \in \mathscr{Q}(R,\infty)$ with $v\neq 0$, then there must exist some $\mu \in K_n\setminus \{0\}$ such that 
\begin{align*}
\iprod{g}{\mu}\geq \sqrt{\bigpnorm{ \pnorm{\mu}{}\Pi_L(h)-v}{}^2+ \pnorm{\mu}{}^2\pnorm{\Pi_{L^\circ}(h)}{}^2}\geq \pnorm{\mu}{} \pnorm{\Pi_{L^\circ}(h)}{}.
\end{align*}
This means that
\begin{align*}
\Big\{\mathscr{Q}(R,\infty)\setminus \{0\}\neq \emptyset\Big\}\cap E_0(t)&\subset \Big\{\sup_{\mu \in K\cap \partial B_n} \iprod{g}{\mu}\geq \pnorm{\Pi_{L^\circ}(h)}{}\Big\}\cap E_0(t)\\
&\subset \Big\{\gw(K\cap \partial B_n)\geq \gw(L\cap \partial B_m)-2C_0\sqrt{t}\Big\}\cap E_0(t),
\end{align*}
proving the claim (\ref{ineq:support_fcn_proj_sup_upper_6}). Now using (\ref{ineq:support_fcn_proj_sup_upper_5}), for any $z>0$, we have 
\begin{align*}
\Prob\bigg(\sup_{x \in L\cap \partial B_m} \mathsf{h}(x) >z\bigg)\leq \bm{1}\Big(\gw(K\cap \partial B_n)\geq \gw(L\cap \partial B_m)-2C_0\sqrt{t}\Big)+2e^{-t}+\Prob(E(R)^c).
\end{align*}
Finally taking $z\downarrow 0$ and $R\uparrow \infty$ to conclude, upon noting that (i) $\sup_{x \in L\cap \partial B_m} \mathsf{h}(x) \neq 0$ is equivalent to $\sup_{x \in L\cap \partial B_m} \mathsf{h}(x) > 0$, and (ii) the quantities  $\gw(K\cap \partial B_n), \gw(L\cap \partial B_m)$ can be replaced by $\sqrt{\delta(K)}$ and $\sqrt{\delta(L)}$ for non-trivial $K,L$ via an application of Proposition \ref{prop:stat_dim}-(3)(4).
\end{proof}

Below we shall give a different proof for a `pointwise' version of Proposition \ref{prop:support_fcn_proj_upper}: Under the same setup as in Proposition \ref{prop:support_fcn_proj_upper}, we will show that there exists some universal constant $C>0$ such that for any $x \in L\cap \partial B_m$ and $t\geq 1$,
\begin{align}\label{ineq:pointwise_support_fcn}
\Prob\Big(\mathsf{h}_{L\cap G K_n}(x)\neq 0\Big)\leq \bm{1}\Big(\sqrt{\delta(K)}> \sqrt{\delta(L^\circ)}-C\sqrt{t}\Big)+e^{-t}.
\end{align}
While the conclusion (\ref{ineq:pointwise_support_fcn}) per se is weaker, the proof below exploits a different max-min representation of the support function $\mathsf{h}_{L\cap G K_n}$ from the one used in (\ref{ineq:support_fcn_proj_sup_upper_0}). In particular, using Lemma \ref{lem:cone_polarity}, we may write
\begin{align}\label{ineq:support_fcn_proj_upper_0}
\mathsf{h}_{L\cap G K_n}(x) &= \sup_{\mu \in K_n, G\mu \in L}\iprod{x}{G\mu} = \sup_{\mu \in K_n}\inf_{v \in L^\circ} \Big\{ \iprod{x}{G\mu}-\iprod{v}{G\mu} \Big\}.
\end{align}
This representation (\ref{ineq:support_fcn_proj_upper_0}) will be more convenient to work with in the `large $\delta(K)$ regime' in Proposition \ref{prop:support_fcn_proj_lower} below, as well as in the proofs of  Theorems \ref{thm:conic_program} and \ref{thm:local_dm} in Sections \ref{section:proof_conic_program} and \ref{section:proof_DM} ahead.

\begin{proof}[Proof of the pointwise version (\ref{ineq:pointwise_support_fcn}) via the representation (\ref{ineq:support_fcn_proj_upper_0})]
We again write $\mathsf{h}\equiv\mathsf{h}_{L\cap G K_n}$ for notational convenience in the proof. By the one-sided Gaussian min-max theorem (cf. Theorem \ref{thm:CGMT}-(1)), for any $z \in \R$ and $R>0$,
\begin{align}\label{ineq:support_fcn_proj_upper_1}
\Prob\big(\mathsf{h}(x)>z\big)&\leq \Prob\bigg( \sup_{\mu \in K_n}\inf_{v \in L^\circ \cap B_m(R)} \iprod{x-v}{G\mu} >z  \bigg)\nonumber\\
&\leq 2 \Prob\bigg(\sup_{\mu \in K_n}\inf_{v \in L^\circ \cap B_m(R)} \Big\{-\pnorm{\mu}{}\iprod{h}{v-x}+\pnorm{v-x}{}\iprod{g}{\mu} \Big\}>z   \bigg).
\end{align}
Using that $\{v \in L^\circ: \pnorm{v-x}{}\leq R-1\}\subset L^\circ\cap B_m(R)$, for any $R>2$, with 
\begin{align}\label{ineq:support_fcn_proj_upper_1_1}
L^\circ(\beta;x)\equiv \{(v-x)/\pnorm{v-x}{}:v \in L^\circ, \pnorm{v-x}{}=\beta\}\subset \partial B_m,
\end{align}
we have
\begin{align}\label{ineq:support_fcn_proj_upper_1_2}
& \sup_{\mu \in K_n}\inf_{v \in L^\circ \cap B_m(R)} \Big\{-\pnorm{\mu}{}\iprod{h}{v-x}+\pnorm{v-x}{}\iprod{g}{\mu} \Big\}\nonumber\\
& \leq \sup_{\mu \in K_n} \inf_{1\leq \beta \leq R-1} \beta \bigg\{-\pnorm{\mu}{}\sup_{w \in L^\circ(\beta;x)} \iprod{h}{w}+\iprod{g}{\mu}\bigg\}.
\end{align}
Now taking $\beta=R-1$ in the infimum, and using the cone property of $K$, we have
\begin{align}\label{ineq:support_fcn_proj_upper_2}
& \sup_{\mu \in K_n}\inf_{v \in L^\circ \cap B_m(R)} \Big\{-\pnorm{\mu}{}\iprod{h}{v-x}+\pnorm{v-x}{}\iprod{g}{\mu} \Big\}\nonumber\\
&\leq (R-1)\cdot \sup_{\mu \in K_n}\bigg(\iprod{g}{\mu}- \pnorm{\mu}{}\sup_{w \in L^\circ(R-1;x)} \iprod{h}{w} \bigg)\nonumber\\
&=(R-1)\cdot\bigg(\sup_{\mu \in K\cap \partial B_n} \iprod{g}{\mu}-\sup_{w \in L^\circ(R-1;x)} \iprod{h}{w}\bigg)_+.
\end{align}
By Gaussian concentration in Proposition \ref{prop:gaussian_conc_generic}, for $t\geq 1$, on an event $E_R(t)$ with probability $1-e^{-t}$,
\begin{align}\label{ineq:support_fcn_proj_upper_4}
&\bigabs{\sup_{\mu \in K\cap \partial B_n} \iprod{g}{\mu}-\gw(K\cap \partial B_n)}\nonumber\\
&\qquad \vee \bigabs{ \sup_{w \in L^\circ(R-1;x)} \iprod{h}{w} -\gw\big(L^\circ(R-1;x)\big)}\leq C\sqrt{t}.
\end{align}
Combined with (\ref{ineq:support_fcn_proj_upper_2}), on the event $E_R(t)$ we have,
\begin{align}\label{ineq:support_fcn_proj_upper_3}
& \sup_{\mu \in K_n}\inf_{v \in L^\circ \cap B_m(R)} \Big\{-\pnorm{\mu}{}\iprod{h}{v-x}+\pnorm{v-x}{}\iprod{g}{\mu} \Big\}\nonumber\\
&\leq (R-1) \cdot \Big(\gw(K\cap \partial B_n)- \gw\big(L^\circ(R-1;x)\big)+C\sqrt{t}\Big)_+.
\end{align}
On the other hand, as for $\beta>1$,
\begin{align*}
&\bigabs{\sup_{w \in L^\circ(\beta;x)} \iprod{h}{w}- \sup_{v \in L^\circ \cap \partial B_m} \iprod{h}{v}}\leq \pnorm{h}{}\sup_{v \in L^\circ, \pnorm{v}{}\geq \beta-1}\biggpnorm{ \frac{v-x}{\pnorm{v-x}{}}-\frac{v}{\pnorm{v}{}} }{}\leq  \frac{2\pnorm{h}{}}{\beta-1},
\end{align*}
we have almost surely $\lim_{\beta\uparrow \infty} \sup_{w \in L^\circ(\beta;x)} \iprod{h}{w}= \sup_{v \in L^\circ \cap \partial B_m} \iprod{h}{v}$, and therefore by integrability of the suprema of Gaussian processes,
\begin{align*}
\lim_{\beta\uparrow \infty} \gw\big(L^\circ(\beta;x) \big)= \gw\big(L^\circ \cap \partial B_m\big). 
\end{align*}
This means that for $t\geq 1$, we may find some large enough deterministic $R_t>0$ such that for all $R\geq R_t$, it holds that
\begin{align*}
\abs{\gw\big(L^\circ(R-1;x) \big)-\gw\big(L^\circ \cap \partial B_m\big)}\leq \sqrt{t}. 
\end{align*}
Combined with (\ref{ineq:support_fcn_proj_upper_3}), for $R\geq R_t$, on the event $E_R(t)$, 
\begin{align*}
& \sup_{\mu \in K_n}\inf_{v \in L^\circ \cap B_m(R)} \Big\{-\pnorm{\mu}{}\iprod{h}{v-x}+\pnorm{v-x}{}\iprod{g}{\mu} \Big\}\nonumber\\
&\leq (R-1) \cdot \Big(\gw(K\cap \partial B_n)- \gw\big(L^\circ\cap \partial B_m\big)+C\sqrt{t}\Big)_+.
\end{align*}
Now using (\ref{ineq:support_fcn_proj_upper_1}) with $z=0$, for any $t\geq 1$ and choosing $R\geq R_t$, 
\begin{align*}
&\Prob\big(\mathsf{h}(x)>z=0\big)\\
&\leq \bm{1}\bigg((R-1) \cdot \Big(\gw(K\cap \partial B_n)- \gw\big(L^\circ\cap \partial B_m\big)+C\sqrt{t}\Big)_+>0\bigg)+2\Prob(E_R(t)^c)\\
& \leq  \bm{1}\Big(\gw(K\cap \partial B_n)- \gw\big(L^\circ\cap \partial B_m\big)+C\sqrt{t}>0\Big)+2e^{-t}.
\end{align*}
Finally using $\mathsf{h}(x)>0\Leftrightarrow \mathsf{h}(x)\neq 0$, and Proposition \ref{prop:stat_dim}-(3)(4) to conclude via adjusting constants. 
\end{proof}

\subsection{Proof of Theorem \ref{thm:approx_kinematics}: large $\delta(K)$ regime}

The following proposition shows that for `large' values of $\delta(K)$, the support function of $L\cap G(K\cap B_n)$ will be non-trivial with high probability.

\begin{proposition}\label{prop:support_fcn_proj_lower}
	Suppose that $K \subset \R^n$ and $L\subset \R^m$ are closed convex cones with $K\neq \{0\}, L\notin \{\{0\},\R^m\}$. Fix $x \in L\cap \partial B_m$. Then there exists some universal constant $C>0$ such that for any $t\geq 1$,
	\begin{align*}
	\Prob\Big(\mathsf{h}_{L\cap G K_n}(x)= 0\Big)\leq\bm{1}\Big(\sqrt{\delta(K)}< \sqrt{\delta(L^\circ)}+C\sqrt{t}\Big)+e^{-t}.
	\end{align*}
	Here $K_n\equiv K\cap B_n$.
\end{proposition}

\begin{proof}
	We shall again drop the subscript in $\mathsf{h}_{L\cap G K_n}$ in the proof. Recall the representation of $\mathsf{h}$ in (\ref{ineq:support_fcn_proj_upper_0}). We shall define a surrogate function of $\mathsf{h}$ as follows: Let for $\epsilon>0$
	\begin{align*}
	\mathsf{h}_{\epsilon}(x)&\equiv \sup_{\mu \in K_n}\inf_{v \in L^\circ} \bigg\{\iprod{x}{G\mu}-\iprod{v}{G\mu}+\frac{\epsilon}{2}\pnorm{v}{}^2\bigg\}\\
	&\stackrel{(\ast)}{=}\sup_{\mu \in K_n}\bigg\{\iprod{x}{G\mu}-\frac{ 1}{2\epsilon}\bigg(\sup_{v \in L^\circ \cap B_m}\iprod{v}{G\mu}\bigg)^2\bigg\}.
	\end{align*}
	Here the last equality $(\ast)$ follows as
	\begin{align*}
	&\inf_{v \in L^\circ}\bigg\{-\iprod{v}{G\mu}+\frac{\epsilon}{2}\pnorm{v}{}^2\bigg\} = -\sup_{\beta \geq 0} \bigg\{\sup_{v \in L^\circ,\pnorm{v}{}=\beta}\iprod{v}{G\mu}-\frac{\epsilon}{2}\cdot \beta^2\bigg\}\\
	& = - \max\bigg\{0, \sup_{\beta>0} \bigg(\beta \sup_{v \in L^\circ \cap \partial B_m}\iprod{v}{G\mu}- \frac{\epsilon}{2}\cdot \beta^2\bigg) \bigg\}\\
	&= - \frac{1}{2\epsilon} \bigg(0\vee\sup_{v \in L^\circ \cap \partial B_m}\iprod{v}{G\mu}\bigg)^2=-\frac{ 1}{2\epsilon}\bigg(\sup_{v \in L^\circ \cap B_m}\iprod{v}{G\mu}\bigg)^2.
	\end{align*}
	The main purpose of $\mathsf{h}_\epsilon$, as will be clear in Step 2 below, is to induce automatic localization over $\inf_{v \in L^\circ}$ that facilitates applications of the Gaussian min-max theorem. By definition of $\mathsf{h}_\epsilon$, clearly $\mathsf{h}(x)\leq \mathsf{h}_{\epsilon}(x)$. 
	
	\noindent (\textbf{Step 1}). We shall first prove
	\begin{align}\label{ineq:support_fcn_proj_lower_0}
	\lim_{\epsilon\downarrow 0} \mathsf{h}_\epsilon(x)=\mathsf{h}(x).
	\end{align}	
	Let $E_0\equiv \{\pnorm{G}{\op}>0\}$. Then $\Prob(E_0)=1$. Note that for any $\mu \in K_n$ such that $\big(\sup_{v \in L^\circ \cap B_m}\iprod{v}{G\mu}\big)^2>4\pnorm{G}{\op}\epsilon$, we have the simple estimate
	\begin{align*}
	\iprod{x}{G\mu}-\frac{1}{2\epsilon}\bigg(\sup\limits_{v \in L^\circ \cap B_m}\iprod{v}{G\mu}\bigg)^2<-\pnorm{G}{\op},
	\end{align*}
	so on the event $E_0$, such $\mu$'s are not feasible maximizers in the definition of $\mathsf{h}_\epsilon$, as $\mathsf{h}_\epsilon(x)\geq 0$. This means on the event $E_0$,
	\begin{align}\label{ineq:support_fcn_proj_lower_1}
	\mathsf{h}(x)\leq \mathsf{h}_{\epsilon}(x) &= \sup_{\mu \in K_n, (\sup\limits_{v \in L^\circ \cap B_m}\iprod{v}{G\mu})^2 \leq 4\pnorm{G}{\op}\epsilon}\bigg\{\iprod{x}{G\mu}-\frac{ 1}{2\epsilon}\bigg(\sup_{v \in L^\circ \cap B_m}\iprod{v}{G\mu}\bigg)^2\bigg\} \nonumber\\
	&\leq  \sup_{\mu \in K_n, (\sup\limits_{v \in L^\circ \cap B_m}\iprod{v}{G\mu})^2 \leq 4\pnorm{G}{\op}\epsilon} \iprod{x}{G\mu}.
	\end{align}
	Note that
	\begin{align*}
	S_\epsilon \equiv \bigg\{\mu \in K_n: \bigg(\sup_{v \in L^\circ \cap B_m}\iprod{v}{G\mu}\bigg)^2\leq 4\pnorm{G}{\op}\epsilon\bigg\}
	\end{align*}
	is a sequence of non-increasing sets as $\epsilon \downarrow 0$, and
	\begin{align*}
	&\bigcap_{\epsilon>0}S_\epsilon\subset \bigg\{\mu \in K_n: \bigg(\sup_{v \in L^\circ \cap B_m}\iprod{v}{G\mu}\bigg)^2=0\bigg\} = \big\{\mu \in K_n: G\mu \in L\big\},
	\end{align*}
	so using (\ref{ineq:support_fcn_proj_lower_1}) and Lemma \ref{lem:sup_set_conv}, we have on the event $E_0$, 
	\begin{align*}
	\mathsf{h}(x)\leq \liminf_{\epsilon \downarrow 0} \mathsf{h}_\epsilon(x)\leq \limsup_{\epsilon \downarrow 0} \mathsf{h}_\epsilon(x) \leq \sup_{\mu \in K_n: G\mu \in L}\iprod{x}{G\mu}= \mathsf{h}(x).
	\end{align*}
	On $E_0^c$, $\pnorm{G}{\op}=0$ so $G=0$ and $\mathsf{h}(x)=\mathsf{h}_\epsilon(x)=0$ are trivial. The claim (\ref{ineq:support_fcn_proj_lower_0}) is proven.

    \noindent (\textbf{Step 2}). Next we shall use the surrogate function $\mathsf{h}_\epsilon$ and the claim (\ref{ineq:support_fcn_proj_lower_0}) in Step 1 to prove that for any $z \in \R$,
    \begin{align}\label{ineq:support_fcn_proj_lower_2}
     \Prob\big(\mathsf{h}(x)\leq z\big)\leq 2\Prob\bigg(\sup_{\mu \in K_n}\inf_{v \in L^\circ} \Big\{-\pnorm{\mu}{}\iprod{h}{v-x}+\pnorm{v-x}{}\iprod{g}{\mu} \Big\}\leq z\bigg).
    \end{align}
	For $R>0$, define the event
	\begin{align*}
	E_1(R)\equiv \big\{\pnorm{G}{\op}\leq R\big\}\cap \big\{\pnorm{g}{}+\pnorm{h}{}\leq R\big\}.
	\end{align*}
	On the event $E_1(R)$ we may restrict the minimum over $v \in L^\circ$ to $v \in L^\circ\cap B_m(R/\epsilon)$, i.e., on the event $E_1(R)$,
	\begin{align*}
	\mathsf{h}_{\epsilon}(x)&\equiv \sup_{\mu \in K_n}\inf_{v \in L^\circ \cap B_m(R/\epsilon)} \bigg\{\iprod{x-v}{G\mu}+\frac{\epsilon}{2}\pnorm{v}{}^2\bigg\}.
	\end{align*}
	By the convex Gaussian min-max theorem (cf. Theorem \ref{thm:CGMT}-(2)), for any $z_0>z$,
	\begin{align*}
	&\Prob\big(\mathsf{h}_{\epsilon}(x)< z_0\big)-\Prob(E_1(R)^c)\\
	&\leq 2\Prob\bigg(\sup_{\mu \in K_n}\inf_{v \in L^\circ \cap B_m(R/\epsilon)} \bigg\{-\pnorm{\mu}{}\iprod{h}{v-x}+\pnorm{v-x}{}\iprod{g}{\mu}+\frac{\epsilon}{2}\pnorm{v}{}^2 \bigg\}< z_0\bigg)\\
	&\leq 2\Prob\bigg(\sup_{\mu \in K_n}\inf_{v \in L^\circ} \Big\{-\pnorm{\mu}{}\iprod{h}{v-x}+\pnorm{v-x}{}\iprod{g}{\mu} \Big\}< z_0\bigg).
	\end{align*}
	Now taking $R\uparrow \infty$ and $\epsilon \downarrow 0$, and finally $z_0\downarrow z$ proves the claim (\ref{ineq:support_fcn_proj_lower_2}) upon using (\ref{ineq:support_fcn_proj_lower_0}). 
	
	\noindent (\textbf{Step 3}). In this step we prove the following key lower bound:
	\begin{align}\label{ineq:support_fcn_proj_lower_3}
	&\sup_{\mu \in K_n}\inf_{v \in L^\circ } \Big\{ -\pnorm{\mu}{}\iprod{h}{v-x}+\pnorm{v-x}{}\iprod{g}{\mu} \Big\}\nonumber\\
	& \geq \frac{1}{2} \bigg\{\bigg(\sup_{\mu \in K\cap B_n}\iprod{g}{\mu}\bigg)^2-\bigg(\sup_{v \in L^\circ\cap B_m} \iprod{h}{v}\bigg)^2\bigg\}_+^{1/2}- \abs{\iprod{h}{x}}.
	\end{align}
	First note that
	\begin{align}\label{ineq:support_fcn_proj_lower_4}
	& \sup_{\mu \in K_n}\inf_{v \in L^\circ } \Big\{-\pnorm{\mu}{}\iprod{h}{v-x}+\pnorm{v-x}{}\iprod{g}{\mu} \Big\}\nonumber\\
	&\geq \sup_{\mu \in K_n}\inf_{v \in L^\circ } \Big\{-\pnorm{\mu}{}\iprod{h}{v}+\pnorm{v-x}{}\iprod{g}{\mu} \Big\} - \abs{\iprod{h}{x}}\nonumber\\
	& = \sup_{\mu \in K_n} \inf_{\beta \geq 0} \inf_{v \in L^\circ,\pnorm{v}{}=\beta} \bigg\{- \pnorm{\mu}{}\iprod{h}{v}+\sqrt{\beta^2-2\iprod{v}{x}+1}\cdot \iprod{g}{\mu} \bigg\}- \abs{\iprod{h}{x}}\nonumber\\
	& = \sup_{\mu \in K_n} \inf_{v \in L^\circ\cap \partial B_m}  \inf_{\beta\geq 0} \bigg\{- \beta\pnorm{\mu}{}\iprod{h}{v}+\sqrt{\beta^2+2\beta\iprod{-v}{x}+1}\cdot \iprod{g}{\mu} \bigg\}- \abs{\iprod{h}{x}}\nonumber\\
	& \equiv \sup_{\mu \in K_n} \inf_{v \in L^\circ\cap \partial B_m}  \inf_{\beta\geq 0}\mathsf{P}(\beta;\mu,v) - \abs{\iprod{h}{x}}.
	\end{align}
	We claim that any $\mu \in K_n$ such that $\iprod{g}{\mu}<0$ is not a feasible solution to $\sup_{\mu \in K_n} \inf_{v \in L^\circ\cap \partial B_m}  \inf_{\beta\geq 0}\mathsf{P}(\beta;\mu,v)$. To see this, for any such $\mu \in K_n$, we have
	\begin{align*}
	\inf_{v \in L^\circ\cap \partial B_m}  \inf_{\beta\geq 0}\mathsf{P}(\beta;\mu,v)< \inf_{v \in L^\circ\cap \partial B_m}  \inf_{\beta\geq 0} \big\{-\beta \pnorm{\mu}{}\iprod{h}{v}\big\} \leq 0,
	\end{align*}
	whereas  the cost optimum $\sup_{\mu \in K_n} \inf_{v \in L^\circ\cap \partial B_m}  \inf_{\beta\geq 0}\mathsf{P}(\beta;\mu,v)\geq 0$ due to $0 \in K_n$. This proves that 
	\begin{align}\label{ineq:support_fcn_proj_lower_5}
	\sup_{\mu \in K_n} \inf_{v \in L^\circ\cap \partial B_m}  \inf_{\beta\geq 0}\mathsf{P}(\beta;\mu,v) = \sup_{\mu \in K_n,\iprod{g}{\mu}\geq 0} \inf_{v \in L^\circ\cap \partial B_m}  \inf_{\beta\geq 0}\mathsf{P}(\beta;\mu,v).
	\end{align}
	As the derivative of $\beta \mapsto \mathsf{P}(\beta;\mu,v)$ is given by
	\begin{align*}
	\mathsf{P}'(\beta;\mu,v)& = -\pnorm{\mu}{}\iprod{h}{v}+ \iprod{g}{\mu}\cdot \frac{\beta+\iprod{-v}{x}}{ \sqrt{\beta^2+2\beta\iprod{-v}{x}+1} },
	\end{align*}
	for any $\mu \in K_n$ with $\iprod{g}{\mu}\geq 0$, the map $\beta \mapsto \mathsf{P}'(\beta;\mu,v)$  is non-decreasing. Furthermore, for such $\mu$'s and any $v \in L^\circ \cap \partial B_m, x \in L\cap \partial B_m$,
	\begin{align*}
	\mathsf{P}'(0;\mu,v)& =  -\pnorm{\mu}{}\iprod{h}{v}+ \iprod{g}{\mu}\cdot \iprod{-v}{x},\\
	\mathsf{P}'(\infty;\mu,v)
	& = -\pnorm{\mu}{}\iprod{h}{v}+ \iprod{g}{\mu}.
	\end{align*}
	Fix $\mu \in K_n$ with $\iprod{g}{\mu}\geq 0$ and 
	\begin{align*}
    \inf_{v \in L^\circ\cap  B_m} \mathsf{P}'(\infty;\mu,v) = -\pnorm{\mu}{}\sup_{v \in L^\circ\cap  B_m}\iprod{h}{v}+\iprod{g}{\mu}\geq 0.	
	\end{align*}
	We consider two cases below.
	
	\noindent \textbf{Case 1}. Suppose $v \in L^\circ \cap \partial B_m$ is such that $\mathsf{P}'(0;\mu,v)\geq 0$. Then 
	\begin{align}\label{ineq:support_fcn_proj_lower_6}
	\inf_{\substack{ v \in L^\circ\cap \partial B_m,\\ \mathsf{P}'(0;\mu,v)\geq 0}}  \inf_{\beta\geq 0}\mathsf{P}(\beta;\mu,v) = \inf_{\substack{ v \in L^\circ\cap \partial B_m,\\ \mathsf{P}'(0;\mu,v)\geq 0}}  \mathsf{P}(0;\mu,v) = \iprod{g}{\mu}.
	\end{align}
	\noindent \textbf{Case 2}. Suppose $v \in L^\circ \cap \partial B_m$ is such that $\mathsf{P}'(0;\mu,v)\leq 0$ and $\mathsf{P}'(\infty;\mu,v)\geq 0$. Then $\iprod{h}{v}\geq 0$ and
	\begin{align}\label{ineq:support_fcn_proj_lower_7}
	&\inf_{\substack{ v \in L^\circ\cap \partial B_m,\\ \mathsf{P}'(0;\mu,v)\leq 0, \mathsf{P}'(\infty;\mu,v)\geq 0 }}  \inf_{\beta\geq 0}\mathsf{P}(\beta;\mu,v) \nonumber \\
	& =\inf_{\substack{ v \in L^\circ\cap \partial B_m,\\ \mathsf{P}'(0;\mu,v)\leq 0, \mathsf{P}'(\infty;\mu,v)\geq 0 }} \bigg\{\sqrt{\big(1-\iprod{-v}{x}^2\big)\big(\iprod{g}{\mu}^2-\pnorm{\mu}{}^2\iprod{h}{v}^2\big)}  +\pnorm{\mu}{}\iprod{h}{v}\iprod{-v}{x}\bigg\}\nonumber\\
	&\geq \inf_{\substack{ v \in L^\circ\cap \partial B_m,\\ \mathsf{P}'(0;\mu,v)\leq 0, \mathsf{P}'(\infty;\mu,v)\geq 0 }} \bigg\{\sqrt{\big(1-\iprod{-v}{x}^2\big)\big(\iprod{g}{\mu}^2-\pnorm{\mu}{}^2\iprod{h}{v}^2\big)}  +\iprod{g}{\mu}\iprod{-v}{x}^2\bigg\}\nonumber\\
	&\geq \inf_{\substack{ v \in L^\circ\cap \partial B_m,\\ \mathsf{P}'(0;\mu,v)\leq 0, \mathsf{P}'(\infty;\mu,v)\geq 0 }}\inf_{\alpha \in [0,1]} \bigg\{\sqrt{\big(1-\alpha\big)\big(\iprod{g}{\mu}^2-\pnorm{\mu}{}^2\iprod{h}{v}^2\big)}  +\iprod{g}{\mu}\alpha\bigg\}\nonumber\\
	&\stackrel{(\ast)}{\geq} \frac{1}{2} \inf_{v \in L^\circ\cap \partial B_m, \iprod{h}{v}\geq 0} \min\bigg\{\sqrt{\big(\iprod{g}{\mu}^2-\pnorm{\mu}{}^2\iprod{h}{v}^2\big)_+}, \iprod{g}{\mu} \bigg\}\nonumber\\
	& = \frac{1}{2} \cdot \bigg\{\iprod{g}{\mu}^2-\pnorm{\mu}{}^2\bigg(\sup_{v \in L^\circ\cap  B_m} \iprod{h}{v}\bigg)^2\bigg\}_+^{1/2}.
	\end{align}
	Here the first identity follows by Lemma \ref{lem:fcn_P} below. The inequality in $(\ast)$ follows by the following simple lower bound: for any $M_1,M_2\geq 0$,
	\begin{align*}
	&\inf_{\alpha \in [0,1]} \big\{\sqrt{1-\alpha} M_1+\alpha M_2\big\}\\
	& = \min\bigg\{\inf_{\alpha \in [0,1/2]} \big(\sqrt{1-\alpha} M_1+\alpha M_2\big), \inf_{\alpha \in [1/2,1]} \big(\sqrt{1-\alpha} M_1+\alpha M_2\big) \bigg\}\\
	&\geq \min \big\{M_1/\sqrt{2}, M_2/2\big\} \geq (M_1\wedge M_2)/2.
	\end{align*}
	The last identity in (\ref{ineq:support_fcn_proj_lower_7}) uses the fact that 
	\begin{align*}
	\sup_{v \in L^\circ \cap \partial B_m, \iprod{h}{v}\geq 0} \iprod{h}{v}^2 &= \bigg(\sup_{v \in L^\circ \cap \partial B_m} \big(0\vee \iprod{h}{v}\big)\bigg)^2  = \bigg(\sup_{v \in L^\circ \cap  B_m} \iprod{h}{v}\bigg)^2.
	\end{align*} 
	Combining (\ref{ineq:support_fcn_proj_lower_4})-(\ref{ineq:support_fcn_proj_lower_7}), we have
	\begin{align*}
	&\sup_{\mu \in K_n}\inf_{v \in L^\circ } \Big\{ -\pnorm{\mu}{}\iprod{h}{v-x}+\pnorm{v-x}{}\iprod{g}{\mu} \Big\}\nonumber\\
	&\geq  \sup_{\substack{\mu \in K_n,\iprod{g}{\mu}\geq 0,\\ \inf_{v \in L^\circ\cap B_m} \mathsf{P}'(\infty;\mu,v)\geq 0}} \inf_{v \in L^\circ\cap \partial B_m}  \inf_{\beta\geq 0}\mathsf{P}(\beta;\mu,v)- \abs{\iprod{h}{x}}\nonumber\\
	& \geq \sup_{\substack{\mu \in K_n,\iprod{g}{\mu}\geq 0,\\ \iprod{g}{\mu}\geq \pnorm{\mu}{}\sup\limits_{v \in L^\circ\cap B_m}\iprod{h}{v}}} \min\bigg\{\iprod{g}{\mu},\frac{1}{2} \bigg[\iprod{g}{\mu}^2-\pnorm{\mu}{}^2\bigg(\sup_{v \in L^\circ\cap B_m} \iprod{h}{v}\bigg)^2\bigg]_+^{1/2}\bigg\}- \abs{\iprod{h}{x}}\nonumber\\
	&\geq \frac{1}{2} \bigg(\sup_{\substack{\mu \in K_n, \iprod{g}{\mu}\geq 0 } }\bigg\{\iprod{g}{\mu}^2-\pnorm{\mu}{}^2\bigg(\sup_{v \in L^\circ\cap B_m} \iprod{h}{v}\bigg)^2\bigg\}_+\bigg)^{1/2}- \abs{\iprod{h}{x}}\nonumber\\
	& = \frac{1}{2} \bigg\{\bigg(\sup_{\mu \in K\cap B_n}\iprod{g}{\mu}\bigg)^2-\bigg(\sup_{v \in L^\circ\cap B_m} \iprod{h}{v}\bigg)^2\bigg\}_+^{1/2}- \abs{\iprod{h}{x}}.
	\end{align*}
	Here the last equality follows as for any $M\geq 0$,
	\begin{align*}
	&\sup_{\substack{\mu \in K_n, \iprod{g}{\mu}\geq 0 } }\big\{\iprod{g}{\mu}^2-\pnorm{\mu}{}^2\cdot M\big\}_+= \sup_{0<\beta\leq 1} \beta^2\cdot \bigg\{\sup_{\mu \in K\cap \partial B_n,\iprod{g}{\mu}\geq 0} \iprod{g}{\mu}^2-M\bigg\}_+\\
	& = \bigg\{\bigg(\sup_{\mu \in K\cap \partial B_n} \big(0\vee \iprod{g}{\mu}\big)\bigg)^2-M\bigg\}_+ = \bigg\{\bigg(\sup_{\mu \in K\cap B_n}\iprod{g}{\mu}\bigg)^2-M\bigg\}_+.
	\end{align*}
	This proves the claimed inequality (\ref{ineq:support_fcn_proj_lower_3}).
	
	\noindent (\textbf{Step 4}).  We now give a probabilistic lower bound for the right hand side of (\ref{ineq:support_fcn_proj_lower_3}). By Gaussian concentration as in Proposition \ref{prop:gaussian_conc_generic}, there exists some universal constant $C_0>0$ such that for $t\geq 1$, on an event $E_0(t)$ with probability at least $1-e^{-t}$,
	\begin{align*}
	\bigabs{\sup_{\mu \in K\cap B_n}\iprod{g}{\mu}-\gw(K\cap B_n)}\vee \bigabs{\sup_{v \in L^\circ\cap B_m} \iprod{h}{v}-\gw(L^\circ\cap B_m)}\leq C_0\sqrt{t}.
	\end{align*}
	Consequently, on the event $E_0(t)$,
    \begin{align*}
    &\bigg\{\bigg(\sup_{\mu \in K\cap B_n}\iprod{g}{\mu}\bigg)^2-\bigg(\sup_{v \in L^\circ\cap B_m} \iprod{h}{v}\bigg)^2\bigg\}_+^{1/2}\\
    & = \bigg(\sup_{\mu \in K\cap B_n}\iprod{g}{\mu}-\sup_{v \in L^\circ\cap B_m} \iprod{h}{v}\bigg)_+^{1/2}\cdot \bigg(\sup_{\mu \in K\cap B_n}\iprod{g}{\mu}+\sup_{v \in L^\circ\cap B_m} \iprod{h}{v}\bigg)_+^{1/2}\\
    & \geq \Big(\gw(K\cap B_n)-\gw(L^\circ\cap B_m)-C_0\sqrt{t}\Big)_+^{1/2}\cdot \Big(\gw(K\cap B_n)+\gw(L^\circ\cap B_m)-C_0\sqrt{t}\Big)_+^{1/2}.
    \end{align*}
	So if $\gw(K\cap B_n)\geq \gw(L^\circ\cap B_m)+2C_0\sqrt{t}$, on the event $E_0(t)$, 
	\begin{align}
	&\bigg\{\bigg(\sup_{\mu \in K\cap B_n}\iprod{g}{\mu}\bigg)^2-\bigg(\sup_{v \in L^\circ\cap B_m} \iprod{h}{v}\bigg)^2\bigg\}_+^{1/2}\nonumber\\
	&\geq \big(C_0\sqrt{t}\big)^{1/2}\cdot \big(\gw(L^\circ\cap B_m)+C_0\sqrt{t}\big)_+^{1/2} \geq C_0\sqrt{t}. 
	\end{align}
	Using the claimed inequality (\ref{ineq:support_fcn_proj_lower_3}), there exists some universal constant $C>0$ such that if $\gw(K\cap B_n)\geq \gw(L^\circ\cap B_m)+C\sqrt{t}$, on the event $E_0(t)$, 
	\begin{align*}
	&\sup_{\mu \in K_n}\inf_{v \in L^\circ } \Big\{ -\pnorm{\mu}{}\iprod{h}{v-x}+\pnorm{v-x}{}\iprod{g}{\mu} \Big\} \geq C^{-1}\sqrt{t} - \abs{\iprod{h}{x}}.
	\end{align*}
	Combined with the claimed inequality (\ref{ineq:support_fcn_proj_lower_2}) for $z=0$, for any $t\geq 1$, if $\gw(K\cap B_n)\geq \gw(L^\circ\cap B_m)+C\sqrt{t}$,
	\begin{align*}
	\Prob\big(\mathsf{h}(x)\leq 0\big)&\leq 2\Prob\bigg(\sup_{\mu \in K_n}\inf_{v \in L^\circ} \Big\{-\pnorm{\mu}{}\iprod{h}{v-x}+\pnorm{v-x}{}\iprod{g}{\mu} \Big\}\leq 0\bigg)\\
	&\leq 2\Prob\big(\abs{\iprod{h}{x}}>\sqrt{t}/C\big)+2\Prob(E_0(t)^c)\leq C e^{-t/C}.
	\end{align*}
	The claim now follows by adjusting constants and using Proposition \ref{prop:stat_dim}-(4). 
\end{proof}

The proof of Proposition \ref{prop:support_fcn_proj_lower} above makes use of the following lemma. 

\begin{lemma}\label{lem:fcn_P}
Let $a=(a_1,a_2,a_3)^\top\in \R_{\geq 0}^3$ with $a_2 \in [0,1]$, and let $\mathsf{P}_a: \R_{\geq 0} \to \R$ be defined by
\begin{align*}
\mathsf{P}_{a}(\beta) \equiv - a_1\cdot \beta+a_3\cdot\sqrt{\beta^2+2a_2 \beta+1}.
\end{align*}
Then its derivative
\begin{align*}
\mathsf{P}_a'(\beta)& = -a_1+a_3\cdot \frac{\beta+a_2}{ \sqrt{\beta^2+2a_2\beta+1}}
\end{align*}
is non-decreasing on $[0,\infty)$, and
\begin{align*}
\inf_{\beta\geq 0} \mathsf{P}_a(\beta)=
\begin{cases}
a_3, & a_1< a_2a_3,\\
\sqrt{(a_3^2-a_1^2)(1-a_2^2)}+a_1a_2, & a_2a_3\leq a_1\leq a_3,\\
-\infty, & a_1>a_3.
\end{cases}
\end{align*}
\end{lemma}
\begin{proof}
The derivative $\mathsf{P}_a'$ may be computed directly. As $a_2 \in [0,1]$, $\mathsf{P}_a'$ is non-decreasing, and we have
\begin{align*}
\mathsf{P}_a'(0) = a_2a_3-a_1, \, \lim_{\beta \uparrow \infty}\mathsf{P}_a'(\beta) = a_3-a_1. 
\end{align*}
\noindent \textbf{Case 1}. Suppose $\mathsf{P}_a'(0)\geq 0$; equivalently $a_2a_3\geq a_1$. Then $\mathsf{P}_a'\geq 0$ globally on $[0,\infty)$ and therefore $\mathsf{P}_a$ is non-decreasing. This means 
\begin{align*}
\inf_{\beta\geq 0} \mathsf{P}_a(\beta) = \mathsf{P}_a(0) = a_3.
\end{align*}
\noindent \textbf{Case 2}. Suppose $\mathsf{P}_a'(0)< 0$ and $\mathsf{P}_a'(\infty)> 0$; equivalently $a_2a_3<a_1<a_3$. To solve the equation $\mathsf{P}_a'(\beta_\ast)=0$ over $\beta_\ast \geq 0$, it is equivalently to solve 
\begin{align*}
&\beta_\ast^2+2a_2 \beta_\ast +\frac{a_2^2a_3^2-a_1^2}{a_3^2-a_1^2} = 0, \beta_\ast \geq 0\, \Leftrightarrow \, \beta_\ast = -a_2 + \sqrt{ \frac{a_1^2(1-a_2^2)}{a_3^2-a_1^2}  }.
\end{align*}
Consequently, with some calculations,
\begin{align*}
\inf_{\beta\geq 0} \mathsf{P}_a(\beta) = \mathsf{P}_a(\beta_\ast)=\sqrt{(a_3^2-a_1^2)(1-a_2^2)}+a_1a_2. 
\end{align*}
\noindent \textbf{Case 3}. Suppose $\mathsf{P}_a'(\infty)\leq 0$; equivalently $a_1\geq a_3$. Then $\mathsf{P}_a'\leq 0$ globally on $[0,\infty)$ so $\mathsf{P}_a$ is non-increasing on $[0,\infty)$. This means
\begin{align*}
\inf_{\beta\geq 0} \mathsf{P}_a(\beta) &= \lim_{\beta \uparrow \infty} \mathsf{P}_a(\beta) \\
&= \lim_{\beta \uparrow \infty} \beta \bigg(-a_1+ a_3\sqrt{1+\frac{2a_2}{\beta}+\frac{1}{\beta^2} }\bigg)=
\begin{cases}
a_1a_2, & a_1=a_3,\\
-\infty, & a_1>a_3.
\end{cases}
\end{align*}
The claim follows by combining the three cases above. 
\end{proof}

\subsection{Completion of Theorem \ref{thm:approx_kinematics}}

\begin{proof}[Proof of Theorem \ref{thm:approx_kinematics}]
When $K\neq \{0\}, L\notin \{\{0\},\R^m\}$, the claimed kinematic formulae follows by combining Propositions \ref{prop:support_fcn_proj_upper} and \ref{prop:support_fcn_proj_lower}. If $K\neq \{0\}$ and $L=\R^m$, then $\delta(L)=m$ and therefore claim (2) holds trivially. 
\end{proof}

\subsection{Proof of Corollary \ref{cor:gordon_escape}}

\begin{lemma}
Fix an integer $\ell \in [n]$. Let $L\subset \R^n$ be any fixed subspace with $\dim(L)=\ell$. Then $O_nL\equald G_n L \equald \nullsp(G_{n-\ell,n})$, where $O_n$ is distributed according to the Haar measure on $\mathrm{O}(n)$, and the matrices $G_n\in \R^{n\times n}$ and $G_{n-\ell,n}\in \R^{(n-\ell)\times n}$ both contain i.i.d. $\mathcal{N}(0,1)$ entries.
\end{lemma}
\begin{proof}
The key fact to note is that there is a unique probability measure on the Grassmannian $\mathscr{G}_\ell(\R^n)$ (with the usual metric induced by the Hausdorff distance between unit balls of subspaces) that is invariant under the action of the orthogonal group $\mathrm{O}(n)$, cf. \cite[Theorem 1.2]{eaton1989group}. It is easy to see that all $O_n L$ and $G_nL$ and $\nullsp(G_{n-\ell,n})$ are invariant under $\mathrm{O}(n)$. 
\end{proof}

\begin{proof}[Proof of Corollary \ref{cor:gordon_escape}]
Using the lemma above, we may realize $O_n L\equald G_n L$ where $G_n\in \R^{n\times n}$ contains i.i.d. $\mathcal{N}(0,1)$ on its entries. Now apply Theorem \ref{thm:approx_kinematics} to conclude. 
\end{proof}

\section{Proofs of Theorems \ref{thm:approx_kinematics_preimage} and \ref{thm:generalized_gordon}}\label{section:proof_generalized_kinetic}

\subsection{Proof of Theorem \ref{thm:approx_kinematics_preimage} }
The following proposition will be used crucially in both the proof of  Theorem \ref{thm:approx_kinematics_preimage}, and that of Theorem \ref{thm:conic_program} to be detailed ahead. Recall the notation $G^{-1}L$ defined in (\ref{def:preimage}).

\begin{proposition}\label{prop:support_fcn_cgmt}
	Suppose $K\subset B_n$ is a compact convex set with $0 \in K$, and $L\subset \R^m$ is a closed convex cone. Then for any $x \in \partial B_n, z \in \R$,
	\begin{align}\label{ineq:support_fcn_cgmt_01}
	\Prob\Big( \mathsf{h}_{K\cap G^{-1}L}(x)\geq z\Big)&\leq 2 \Prob\bigg(\sup_{\mu \in K, \iprod{g}{\mu}\geq \pnorm{\mu}{} \sup\limits_{v \in L^\circ \cap \partial B_m}\iprod{h}{v} } \iprod{x}{\mu}\geq z\bigg),\nonumber\\
	\Prob\big(\mathsf{h}_{K\cap G^{-1}L}(x)< z\big)&\leq 2 \Prob\bigg(\sup_{\mu \in K, \iprod{g}{\mu}\geq \pnorm{\mu}{} \sup\limits_{v \in L^\circ \cap \partial B_m}\iprod{h}{v} } \iprod{x}{\mu}< z\bigg).
	\end{align}
	The first inequality also holds for its uniform version in $x \in \partial B_n$, i.e., for any $z \in \R$,
	\begin{align}\label{ineq:support_fcn_cgmt_02}
	\Prob\Big( \sup_{x \in \partial B_n}\mathsf{h}_{K\cap G^{-1}L }(x)\geq z\Big)&\leq 2 \Prob\bigg(\sup_{x \in \partial B_n, \mu \in K, \iprod{g}{\mu}\geq \pnorm{\mu}{} \sup\limits_{v \in L^\circ \cap \partial B_m}\iprod{h}{v} } \iprod{x}{\mu}\geq z\bigg).
	\end{align}
\end{proposition}
\begin{proof}
	We shall write $\mathsf{h}\equiv \mathsf{h}_{K\cap G^{-1}L }$ for notational convenience in the proof. The method of proof is similar in spirit to that of Proposition \ref{prop:support_fcn_proj_lower}. First, by Lemma \ref{lem:cone_polarity}, we may rewrite $\mathsf{h}(x)$ as
	\begin{align*}
	\mathsf{h}(x)& = \sup_{\mu \in K}\big\{\iprod{x}{\mu}: G\mu \in L\big\} = \sup_{\mu \in K}\inf_{v\in L^\circ} \big\{\iprod{x}{\mu}-\iprod{v}{G\mu}\big\}.
	\end{align*}
	Now let for $\epsilon>0$
	\begin{align*}
	\mathsf{h}_{\epsilon}(x)&\equiv \sup_{\mu \in K}\inf_{v \in L^\circ} \bigg\{\iprod{x}{\mu}-\iprod{v}{G\mu}+\frac{\epsilon}{2}\pnorm{v}{}^2\bigg\}=\sup_{\mu \in K}\bigg\{\iprod{x}{\mu}-\frac{ 1}{2\epsilon}\bigg(\sup_{v \in L^\circ \cap B_m}\iprod{v}{G\mu}\bigg)^2\bigg\}.
	\end{align*}
	On the event 
	\begin{align}\label{ineq:gordon_proof_event}
	E(R)\equiv \big\{\pnorm{G}{\op}\leq R\big\}\cap \big\{\pnorm{g}{}+\pnorm{h}{}\leq R\big\},
	\end{align}
	we may restrict the range of the minimum over $v \in L^\circ$ to $v \in B_m(R/\epsilon)$ (uniformly in $x\in \partial B_n$). In other words, on the event $E(R)$, 
	\begin{align}\label{ineq:support_fcn_1_1}
	\mathsf{h}_{\epsilon}(x)&\equiv \sup_{\mu \in K}\inf_{v \in L^\circ \cap B_m(R/\epsilon)} \bigg\{\iprod{x}{\mu}-\iprod{v}{G\mu}+\frac{\epsilon}{2}\pnorm{v}{}^2\bigg\}.
	\end{align}
	Clearly $\mathsf{h}(x)\leq \mathsf{h}_{\epsilon}(x)$ for any $\epsilon>0$. This means for any $z \in \R$,
	\begin{align*}
	&\Prob\big(\mathsf{h}(x)\geq z\big)\leq \Prob\big(\mathsf{h}_\epsilon(x)\geq z\big)\\
	&\leq \Prob\bigg( \sup_{\mu \in K}\inf_{v \in L^\circ \cap B_m(R/\epsilon)} \bigg\{\iprod{x}{\mu}-\iprod{v}{G\mu}+\frac{\epsilon}{2}\pnorm{v}{}^2\bigg\}\geq z  \bigg) + \Prob\big(E(R)^c\big).
	\end{align*}
	By the one-sided Gaussian min-max theorem (cf. Theorem \ref{thm:CGMT}-(1)), we have
	\begin{align}\label{ineq:support_fcn_2}
	\Prob\big(\mathsf{h}(x)\geq z\big)&\leq 2\Prob\bigg( \sup_{\mu \in K}\inf_{v \in L^\circ\cap B_m(R/\epsilon)} \bigg\{\iprod{x}{\mu}-\pnorm{\mu}{}\iprod{h}{v}+\pnorm{v}{}\iprod{g}{\mu}+\frac{\epsilon}{2}\pnorm{v}{}^2\bigg\}\geq z  \bigg) \nonumber\\
	&\qquad + \Prob\big(E(R)^c\big).
	\end{align}
	On the other hand, as $0 \in K$, on the event $E(R)$,
	\begin{align}\label{ineq:support_fcn_3}
	0&\leq \sup_{\mu \in K}\inf_{v \in L^\circ\cap  B_m(R/\epsilon)} \bigg\{\iprod{x}{\mu}-\pnorm{\mu}{}\iprod{h}{v}+\pnorm{v}{}\iprod{g}{\mu}+\frac{\epsilon}{2}\pnorm{v}{}^2\bigg\}\nonumber\\
	& = \sup_{\mu \in K}\inf_{\beta \in [0,R/\epsilon]} \bigg\{\iprod{x}{\mu}+\beta\Big(\iprod{g}{\mu}-\pnorm{\mu}{}\sup_{v \in L^\circ \cap \partial B_m}\iprod{h}{v}\Big)+\frac{\epsilon}{2}\beta^2\bigg\}\nonumber\\
	& = \sup_{\mu \in K} \bigg\{\iprod{x}{\mu}-\frac{1 }{2\epsilon} \bigg(\iprod{g}{\mu}-\pnorm{\mu}{}\sup_{v \in L^\circ \cap \partial B_m}\iprod{h}{v}\bigg)_-^2\bigg\}.
	\end{align}
	Now as for any $\mu \in K$ such that $\big(\iprod{g}{\mu}-\pnorm{\mu}{}\sup\limits_{v \in L^\circ \cap \partial B_m}\iprod{h}{v}\big)_-^2>2\epsilon$, the right hand side of the above display is $<0$, so such $\mu$'s are not feasible. This means that on the event $E(R)$, 
	\begin{align}\label{ineq:support_fcn_4}
	& \sup_{\mu \in K}\inf_{v \in L^\circ\cap B_m(R/\epsilon)} \bigg\{\iprod{x}{\mu}-\pnorm{\mu}{}\iprod{h}{v}+\pnorm{v}{}\iprod{g}{\mu}+\frac{\epsilon}{2}\pnorm{v}{}^2\bigg\}\nonumber\\
	& \leq \sup_{\substack{\mu \in K, \big(\iprod{g}{\mu}-\pnorm{\mu}{} \sup\limits_{v \in L^\circ \cap \partial B_m}\iprod{h}{v}\big)_-^2\leq 2\epsilon} } \iprod{x}{\mu}. 
	\end{align}
	Combining (\ref{ineq:support_fcn_2})-(\ref{ineq:support_fcn_4}), we have
	\begin{align}\label{ineq:support_fcn_5}
	\Prob\big(\mathsf{h}(x)\geq z\big)&\leq 2\Prob\bigg(\sup_{\substack{\mu \in K, \big(\iprod{g}{\mu}-\pnorm{\mu}{} \sup\limits_{v \in L^\circ \cap \partial B_m}\iprod{h}{v}\big)_-^2\leq 2\epsilon} } \iprod{x}{\mu}\geq z\bigg) + 3\Prob(E(R)^c). 
	\end{align}
	Since the set
	\begin{align*}
	S_\epsilon \equiv \bigg\{\mu \in K, \bigg(\iprod{g}{\mu}-\pnorm{\mu}{}\sup\limits_{v \in L^\circ \cap \partial B_m}\iprod{h}{v}  \bigg)_-^2\leq 2\epsilon\bigg\}
	\end{align*}
	is non-increasing as $\epsilon \downarrow 0$, and 
	\begin{align*}
	&\bigcap_{\epsilon>0}S_\epsilon \subset \bigg\{\mu \in K, \bigg(\iprod{g}{\mu}-\pnorm{\mu}{}\sup\limits_{v \in L^\circ \cap \partial B_m}\iprod{h}{v}  \bigg)_-^2=0\bigg\},
	\end{align*}
	we have by Lemma \ref{lem:sup_set_conv}
	\begin{align*}
	\sup_{\substack{\mu \in K, \big(\iprod{g}{\mu}-\pnorm{\mu}{} \sup\limits_{v \in L^\circ \cap \partial B_m}\iprod{h}{v}\big)_-^2\leq 2\epsilon} } \iprod{x}{\mu} \downarrow  \sup_{\mu \in \cap_{\epsilon>0} S_\epsilon}\iprod{x}{\mu} \leq \sup_{\substack{\mu \in K, \big(\iprod{g}{\mu}-\pnorm{\mu}{} \sup\limits_{v \in L^\circ \cap \partial B_m}\iprod{h}{v}\big)_-^2=0} } \iprod{x}{\mu}
	\end{align*}
	as $\epsilon\downarrow 0$. Now taking $\epsilon \downarrow 0$ followed by $R\uparrow \infty$ on the right hand side of (\ref{ineq:support_fcn_5}), we have for any $z \in \R$,
	\begin{align}\label{ineq:support_fcn_6}
	\Prob\big(\mathsf{h}(x)\geq z\big)\leq 2 \Prob\bigg(\sup_{\mu \in K, \iprod{g}{\mu}\geq\pnorm{\mu}{}  \sup\limits_{v \in L^\circ \cap \partial B_m}\iprod{h}{v} } \iprod{x}{\mu}\geq z\bigg).
	\end{align}
	The uniform version (\ref{ineq:support_fcn_cgmt_02}) follows by working with the additional $\sup_{x \in \partial B_n}$ in the above arguments and a slight modification of the definition of $S_\epsilon$.

	For the other direction, using (\ref{ineq:support_fcn_1_1}) and the convex Gaussian min-max theorem (cf. Theorem \ref{thm:CGMT}-(2)), we have for any $z \in \R$,
	\begin{align*}
	\Prob\big(\mathsf{h}_\epsilon(x)<  z\big)&\leq 2\Prob\bigg( \sup_{\mu \in K}\inf_{v \in L^\circ\cap B_m(R/\epsilon)} \bigg\{\iprod{x}{\mu}-\pnorm{\mu}{}\iprod{h}{v}+\pnorm{v}{}\iprod{g}{\mu}+\frac{\epsilon}{2}\pnorm{v}{}^2\bigg\}<  z  \bigg) \nonumber\\
	&\qquad + \Prob\big(E(R)^c\big).
	\end{align*}
	Combining the above display and the simple lower bound
	\begin{align*}
	& \sup_{\mu \in K}\inf_{v \in L^\circ\cap B_m(R/\epsilon)} \bigg\{\iprod{x}{\mu}-\pnorm{\mu}{}\iprod{h}{v}+\pnorm{v}{}\iprod{g}{\mu}+\frac{\epsilon}{2}\pnorm{v}{}^2\bigg\}\nonumber\\
	& \geq \sup_{\mu \in K}\inf_{v \in L^\circ}\big\{\iprod{x}{\mu}-\pnorm{\mu}{}\iprod{h}{v}+\pnorm{v}{}\iprod{g}{\mu}\big\}\nonumber\\
	& = \sup_{\mu \in K}\inf_{\beta \geq 0}\bigg\{\iprod{x}{\mu}+\beta\bigg(\iprod{g}{\mu}-\pnorm{\mu}{}\sup\limits_{v \in L^\circ \cap \partial B_m}\iprod{h}{v}\bigg)\bigg\}\nonumber\\
	& = \sup_{\mu \in K, \iprod{g}{\mu}\geq\pnorm{\mu}{} \sup\limits_{v \in L^\circ \cap \partial B_m}\iprod{h}{v} } \iprod{x}{\mu},
	\end{align*}
	we have
	\begin{align}\label{ineq:support_fcn_7}
	\Prob\big(\mathsf{h}_\epsilon(x)< z\big)&\leq 2\Prob\bigg( \sup_{\mu \in K, \iprod{g}{\mu}\geq\pnorm{\mu}{} \sup\limits_{v \in L^\circ \cap \partial B_m}\iprod{h}{v} } \iprod{x}{\mu}< z  \bigg)  + \Prob\big(E(R)^c\big).
	\end{align}
	Now noting that using the same method of proof as for the claim (\ref{ineq:support_fcn_proj_lower_0}) in Step 1 in the proof of Proposition \ref{prop:support_fcn_proj_lower}, we have $ \mathsf{h}_\epsilon(x)\downarrow \mathsf{h}(x)$ as $\epsilon\downarrow 0$.	So by taking $\epsilon\downarrow 0$ followed by $R \uparrow \infty$ in (\ref{ineq:support_fcn_7}) we obtain 
	\begin{align}\label{ineq:support_fcn_10}
	\Prob\big(\mathsf{h}(x)< z\big)\leq 2 \Prob\bigg(\sup_{\mu \in K, \iprod{g}{\mu}\geq \pnorm{\mu}{} \sup\limits_{v \in L^\circ \cap \partial B_m}\iprod{h}{v} } \iprod{x}{\mu}< z\bigg).
	\end{align}
	Combining (\ref{ineq:support_fcn_6}) and (\ref{ineq:support_fcn_10}) to conclude. 
\end{proof}

Now we may prove Theorem \ref{thm:approx_kinematics_preimage}. 

\begin{proof}[Proof of Theorem \ref{thm:approx_kinematics_preimage}]
	
	\noindent (1). For the upper estimate, we take the generic $K$ in Proposition \ref{prop:support_fcn_cgmt} as $K_n\equiv K\cap B_n$.  Note that 
	\begin{align*}
	\sup_{x \in \partial B_n, \mu \in K_n, \iprod{g}{\mu}\geq \pnorm{\mu}{} \sup\limits_{v \in L^\circ \cap \partial B_m}\iprod{h}{v} } \iprod{x}{\mu} = \bm{1}\Big(\sup_{\mu \in K\cap \partial B_n} \iprod{g}{\mu}\geq \sup\limits_{v \in L^\circ \cap \partial B_m}\iprod{h}{v}\Big),
	\end{align*}
	so an application of (\ref{ineq:support_fcn_cgmt_02}) yields that for any $\epsilon>0$, 
	\begin{align*}
	\Prob\Big( \sup_{x \in \partial B_n}\mathsf{h}_{K_n\cap G^{-1}L }(x)\geq \epsilon\Big)\leq 2\Prob\bigg(\sup_{\mu \in K\cap \partial B_n} \iprod{g}{\mu}\geq \sup\limits_{v \in L^\circ \cap \partial B_m}\iprod{h}{v}\bigg). 
	\end{align*}
	Taking $\epsilon\downarrow 0$, we have
	\begin{align}\label{ineq:approx_kinematics_preimage_1}
	\Prob\Big( \sup_{x \in \partial B_n}\mathsf{h}_{K_n\cap G^{-1}L }(x)>0\Big)\leq 2\Prob\bigg(\sup_{\mu \in K\cap \partial B_n} \iprod{g}{\mu}\geq \sup\limits_{v \in L^\circ \cap \partial B_m}\iprod{h}{v}\bigg). 
	\end{align}
	By Gaussian concentration in Proposition \ref{prop:gaussian_conc_generic}, there exists some universal constant $C_0>0$ such that for $t\geq 1$, with probability at least $1-e^{-t}$,
	\begin{align*}
	\bigabs{\sup_{\mu \in K\cap \partial B_n}\iprod{g}{\mu}-\gw(K\cap \partial B_n)}\vee \bigabs{\sup_{v \in L^\circ\cap \partial B_m} \iprod{h}{v}-\gw(L^\circ\cap \partial B_m)}\leq C_0\sqrt{t}.
	\end{align*}
	Combined with (\ref{ineq:approx_kinematics_preimage_1}), we have for any $t\geq 1$,
	\begin{align*}
	\Prob\Big( \sup_{x \in \partial B_n}\mathsf{h}_{K_n\cap G^{-1}L }(x)>0\Big)\leq \bm{1}\Big(\gw(K\cap \partial B_n)-\gw(L^\circ\cap  \partial B_m)+C\sqrt{t}>0\Big)+2e^{-t}. 
	\end{align*}
	The claim now follows by Proposition \ref{prop:stat_dim}-(3)(4).
	
	\noindent (2). For the lower estimate, we first assume $K\neq \R^n$. Then we may find some $x \in \partial B_n$ such that $K_x=K$ and $\sup_{\mu \in K}\iprod{x}{\mu}=\infty$. A straightforward modification of (\ref{ineq:conic_program_1}) shows that for this choice of $x$ and any $z >0$,
	\begin{align*}
	\Prob\Big(\mathsf{h}_{K\cap G^{-1}L }(x)< z\Big)&\leq 2 \Prob\bigg(\sup_{\mu \in K} \iprod{x}{\mu}\cdot\bm{1}\Big(\sup_{\mu \in K_x\cap \partial B_n}\iprod{g}{\mu}\geq \sup_{v \in L^\circ\cap B_m} \iprod{h}{v}\Big)< z\bigg)\\
	& \leq 2\Prob\bigg(\sup_{\mu \in K\cap \partial B_n} \iprod{g}{\mu}< \sup\limits_{v \in L^\circ \cap \partial B_m}\iprod{h}{v} \bigg). 
	\end{align*}
	By a similar concentration argument as above, we have for any $z>0$ and $t\geq 1$,
	\begin{align*}
	\Prob\Big(\mathsf{h}_{K\cap G^{-1}L }(x)< z\Big)\leq \bm{1}\Big(\gw(K\cap \partial B_n)-\gw(L^\circ\cap  \partial B_m)-C\sqrt{t}<0\Big)+2e^{-t}.
	\end{align*}
	The claim follows for $K\neq \R^n$. The claim for $K=\R^n$ follows simply by noting that $n-1\leq \delta(K_x)\leq n$. 
\end{proof}

\subsection{Proof of Theorem \ref{thm:generalized_gordon}}

Fix $t\geq 1$. First note that
\begin{align*}
\bigcup_{R>0}\bigg\{\frac{\mu}{\sqrt{\pnorm{\mu}{}^2+\sigma_m^2}}: \mu \in K\cap B_n(R)\bigg\} = \bigg\{\frac{\mu}{\sqrt{\pnorm{\mu}{}^2+\sigma_m^2}}: \mu \in K\bigg\}
\end{align*}
are included in the unit ball $B_n(1)$, so using the continuity Lemma \ref{lem:sup_set_conv} and the monotone convergence theorem, we obtain 
\begin{align*}
\uparrow \lim_{R \to \infty}\mathscr{T}_{K\cap B_n(R)}(\xi)= \mathscr{T}_K(\xi).
\end{align*}
Now we choose $R_t>0$ such that
\begin{align}\label{ineq:generalized_gordon_1}
\sup_{R>R_t}\abs{\mathscr{T}_{K\cap B_n(R)}(\xi)-\mathscr{T}_K(\xi)}\leq C_1\sqrt{t}
\end{align}
for some absolute constant $C_1>0$ to be specified later on. 

\noindent (1). Using the convex Gaussian min-max theorem (cf. Theorem \ref{thm:CGMT}), we have for any $L_v>0$
\begin{align*}
&\Prob\bigg(\min_{\mu \in \R^n} \max_{v \in B_n(L_v)} \Big\{\bm{0}_K(\mu)+\iprod{v}{G\mu-\xi}\Big\} >0\bigg)\\
&= \Prob\bigg(\min_{\mu \in B_n(R_t)} \max_{v \in B_n(L_v)} \Big\{\bm{0}_K(\mu)+\iprod{v}{G\mu-\xi}\Big\} >0\bigg)\\
&\leq 2 \Prob\bigg(\min_{\mu \in B_n(R_t)} \max_{v \in B_n(L_v)}\Big\{-\pnorm{v}{}\iprod{g}{\mu}+\pnorm{\mu}{}\iprod{h}{v}-\iprod{v}{\xi}+\bm{0}_K(\mu) \Big\}>0\bigg).
\end{align*}
The relation in the bracket of the display is equivalent to
\begin{align*}
&\min_{\mu \in B_n(R_t)}\max_{\beta \in [0,L_v]} \bigg\{\beta \Big(\bigpnorm{\pnorm{\mu}{}h-\xi}{}-\iprod{g}{\mu}\Big)+\bm{0}_K(\mu)\bigg\} \\
&= \min_{\mu \in K\cap B_n(R_t)} L_v \Big(\bigpnorm{\pnorm{\mu}{}h-\xi}{}-\iprod{g}{\mu}\Big)_+>0,
\end{align*}
so is further equivalent to 
\begin{align*}
\sup_{\mu \in K\cap B_n(R_t)}\bigg(\frac{\iprod{g}{\mu}}{\sqrt{ \pnorm{\mu}{}^2 + \sigma_m^2}}-\Delta(\mu)\bigg)<  \sqrt{m},
\end{align*}
where 
\begin{align*}
\Delta(\mu)\equiv \frac{\sqrt{\bigpnorm{  \pnorm{\mu}{}h-\xi}{}^2}-\sqrt{\E \bigpnorm{  \pnorm{\mu}{}h-\xi}{}^2 }}{\sqrt{ \pnorm{\mu}{}^2 + \sigma_m^2}}.
\end{align*}
Consequently, 
\begin{align}\label{ineq:generalized_gordon_2}
&\Prob\bigg(\min_{\mu \in \R^n} \max_{v \in B_n(L_v)} \Big\{\bm{0}_K(\mu)+\iprod{v}{G\mu-\xi}\Big\} >0\bigg)\nonumber\\
&\leq 2\Prob\bigg(\sqrt{m}>\sup_{ \mu \in K\cap B_n(R_t)  } \biggiprod{g}{  \frac{\mu}{  \sqrt{\pnorm{\mu}{}^2+\sigma_m^2} } }-\sup_{\mu \in K}\abs{\Delta(\mu)}\bigg). 
\end{align}
Now as
\begin{align*}
\sup_{\mu \in K }\abs{\Delta(\mu)}&\leq \sup_{\mu \in K}\frac{ \pnorm{\mu}{}^2 \abs{ \pnorm{h}{}^2-m}+2\pnorm{\mu}{}\abs{\iprod{h}{\xi} } }{ (\pnorm{\mu}{}^2+\sigma_m^2)\sqrt{m}}\leq \frac{ \abs{\pnorm{h}{}^2-m} }{\sqrt{m}}+\biggabs{\biggiprod{h}{ \frac{\xi}{\pnorm{\xi}{}} }},
\end{align*}
on an event $E_{0,t}$ with $\Prob(E_{0,t})\geq 1-e^{-t}$, there exists some absolute constant $C_0>0$ such that 
\begin{align}\label{ineq:generalized_gordon_3_0}
\sup_{\mu \in K }\abs{\Delta(\mu)}\leq C_0\sqrt{t}.
\end{align}
On the other hand, by Gaussian concentration in Proposition \ref{prop:gaussian_conc_generic}, on an event $E_{1,t}$ with $\Prob(E_{1,t})\geq 1-e^{-t}$, by adjusting the constant $C_0>0$, we have
\begin{align*}
\biggabs{\sup_{ \mu \in K\cap B_n(R_t)  } \biggiprod{g}{  \frac{\mu}{  \sqrt{\pnorm{\mu}{}^2+\sigma_m^2} } }- \mathscr{T}_{K\cap B_n(R_t)}(\xi)}\leq C_0\sqrt{t}/2. 
\end{align*}
By choosing $C_1\equiv C_0/2$ in (\ref{ineq:generalized_gordon_1}), on the event $E_{1,t}$ we have 
\begin{align}\label{ineq:generalized_gordon_3}
\biggabs{\sup_{ \mu \in K\cap B_n(R_t)  } \biggiprod{g}{  \frac{\mu}{  \sqrt{\pnorm{\mu}{}^2+\sigma_m^2} } }- \mathscr{T}_{K}(\xi)}\leq C_0\sqrt{t}. 
\end{align}
Combining (\ref{ineq:generalized_gordon_2}) and the estimates in (\ref{ineq:generalized_gordon_3_0})-(\ref{ineq:generalized_gordon_3}), we conclude that 
\begin{align*}
&\Prob\bigg(\min_{\mu \in \R^n} \max_{v \in B_n(L_v)} \Big\{\bm{0}_K(\mu)+\iprod{v}{G\mu-\xi}\Big\} >0\bigg)\nonumber\\
&\leq \bm{1}\big(\sqrt{m}> \mathscr{T}_{K}(\xi)-2C_0\sqrt{t}\big)+4e^{-t}.
\end{align*}
Equivalently, for $\sqrt{m}\leq \mathscr{T}_{K}(\xi)-2C_0\sqrt{t}$, with probability at least $1-4e^{-t}$,
\begin{align*}
&\min_{\mu \in \R^n} \max_{v \in B_n(L_v)} \Big\{\bm{0}_K(\mu)+\iprod{v}{G\mu-\xi}\Big\} \leq 0 \\
\Leftrightarrow \quad & \min_{\mu \in \R^n} \max_{v \in B_n(L_v)} \Big\{\bm{0}_K(\mu)+\iprod{v}{G\mu-\xi}\Big\}  = 0\quad \hbox{(the objective is $\geq 0$)}\\
\Leftrightarrow \quad & \min_{\mu \in K} \max_{v \in B_n(L_v)} \iprod{v}{G\mu-\xi}  = 0\\
\Leftrightarrow\quad & \xi \in G K.
\end{align*}

\noindent (2). By the convex Gaussian min-max theorem (cf. Theorem \ref{thm:CGMT}), for any $L_v>0$, using the same calculations as above, for any $R>R_t$ we have
\begin{align*}
&\Prob\bigg(\min_{\mu \in \R^n} \max_{v \in \R^n} \Big\{\bm{0}_{K\cap B_n(R)}(\mu)+\iprod{v}{G\mu-\xi}\Big\} \leq 0\bigg)\\
&\leq \Prob\bigg(\min_{\mu \in B_n(R)} \max_{v \in B_n(L_v)} \Big\{\bm{0}_K(\mu)+\iprod{v}{G\mu-\xi}\Big\} \leq 0\bigg)\\
&\leq 2 \Prob\bigg(\sqrt{m}\leq \sup_{ \mu \in B_n(R)  } \biggiprod{g}{  \frac{\mu}{  \sqrt{\pnorm{\mu}{}^2+\sigma_m^2} } }+\sup_{\mu \in K}\abs{\Delta(\mu)}\bigg).
\end{align*}
Using (\ref{ineq:generalized_gordon_3_0})-(\ref{ineq:generalized_gordon_3}), we then conclude that for $\sqrt{m}> \mathscr{T}_{K}(\xi)+2C_0\sqrt{t}$, with probability at least $1-4e^{-t}$,
\begin{align*}
&\min_{\mu \in \R^n} \max_{v \in \R^n} \Big\{\bm{0}_{K\cap B_n(R)}(\mu)+\iprod{v}{G\mu-\xi}\Big\}>0\\
\Leftrightarrow\quad & \min_{\mu \in K\cap B_n(R)} \max_{v \in \R^n} \iprod{v}{G\mu-\xi}=\infty \\
\Leftrightarrow \quad & \xi \notin G(K\cap B_n(R)).
\end{align*}
Finally, as the sequence of non-increasing set $\{\xi \notin G(K\cap B_n(R)):R>R_t\}$ satisfies $\{\xi \notin GK\}= \cap_{R>R_t}\{\xi \notin G(K\cap B_n(R))\}$, we have
\begin{align*}
\Prob(\xi \notin GK)=\downarrow\lim_{R \to \infty} \Prob\big(\xi \notin G(K\cap B_n(R))\big)\geq 1-4e^{-t}.
\end{align*}
The claims follow by adjusting constants. \qed

\section{Proofs for Section \ref{section:applications}}\label{section:proof_application}

\subsection{Proof of Theorem \ref{thm:logistic_cone_MLE}}\label{section:proof_logistic}

The key to the proof of Theorem \ref{thm:logistic_cone_MLE} is the following. 

\begin{proposition}\label{prop:logistic_exist}
	Suppose that $X_1,\ldots,X_m$ are i.i.d. $\mathcal{N}(0,I_n)$ and $\beta_0=0$. Then the following statements hold.
	\begin{enumerate}
		\item $\Prob\big(\hbox{The cone constrained MLE $\hat{\beta}_K$ exists}\big)\leq \Prob\big(GK\cap \R_{> 0}^m = \emptyset\big)$.
		\item There exists some universal constant $c>0$ such that 
		\begin{align*}
		&\Prob\big(\hbox{The cone constrained MLE $\hat{\beta}_K$ does not exist}\big)\\
		&\leq \Prob\big(GK\cap \R_{\geq 0}^m \neq \{0\}\big)+2e^{-c(\sqrt{m}-\sqrt{\delta(K)})^2}. 
		\end{align*}
	\end{enumerate}
	
\end{proposition}
\begin{proof}
	We shall specify a particular construction of $G$ in the proof. Let $G \in \R^{m\times n}$ be the matrix whose rows are given by $\{X_i^\top: i \in [m]\}\subset \R^n$. Then the entries of $G$ are distributed as i.i.d. $\mathcal{N}(0,1)$ due to the normality assumption on $X_i$'s. We further let $V\equiv GK \subset \R^m$. Then $V$ is a closed convex cone in $\R^m$ (which is random due to the randomness of $G$).  Using that under the global null $\beta_0=0$, the distributions of $Y_i$'s are independent of $X_i$ with $\Prob(Y_i=\pm 1)=1/2$ so $\{Y_i\cdot X_i^\top \beta: \beta \in K\}\equald \{X_i^\top \beta: \beta \in K\}$. Now it easily follows that
	\begin{align*}
	\max_{\beta \in K}\ell(\beta) \equald \min_{\beta \in K} \sum_{i=1}^m \log \Big[1+\exp\big(-X_i^\top \beta\big)\Big] = \min_{v \in V} \sum_{i=1}^m \log\big(1+e^{-v_i}\big),
	\end{align*}	
	where recall $\ell(\cdot)$ is the likelihood function defined in (\ref{eqn:logistic_likelihood}). With these definitions and observations, we may prove the claims (1)-(2). 
	
	\noindent (1). Suppose $V\cap \R_{> 0}^m = GK\cap \R_{> 0}^m \neq \emptyset$. Then there exists some $ \bar{v} \in V$ such that $\bar{v} \in \R_{> 0}^m$.  This means
	\begin{align*}
	\min_{v \in V} \sum_{i=1}^m \log\big(1+e^{-v_i}\big)&\leq \min_{\beta \geq 1} \sum_{i=1}^m \log\big(1+e^{-\beta \bar{v}_i}\big)\leq \liminf_{\beta \uparrow \infty} \sum_{i=1}^m \log\big(1+e^{-\beta \bar{v}_i}\big) =0.
	\end{align*}
	As the left hand side of the above display is non-negative, we necessarily have $
	\min_{v \in V} \sum_{i=1}^m \log\big(1+e^{-v_i}\big)=0$. 
	So any minimizer of $v \mapsto \sum_{i=1}^m \log\big(1+e^{-v_i}\big)$ must be reached at $\infty$, and the cone constrained MLE does not exist. Summarizing the above arguments, we have proven
	\begin{align*}
	\big\{ GK\cap \R_{> 0}^m \neq \emptyset\big\}\subset \big\{\hbox{The cone constrained MLE $\hat{\beta}_K$ does not exist}\big\}
	\end{align*}
	and therefore the claim.

	\noindent (2). Suppose $V\cap \R_{\geq 0}^m = GK\cap \R_{\geq 0}^m = \{0\}$. By the closedness of $\R_{\geq 0}^m$, the quantity
	\begin{align*}
	\epsilon_0\equiv \inf_{v \in V, \pnorm{v}{}=1} \dis(v, \R_{\geq 0}^m) >0
	\end{align*}
	is strictly positive. As $
	\dis^2(v, \R_{\geq 0}^m) = \min_{w \in \R_{\geq 0}^m} \pnorm{v-w}{}^2 = \pnorm{v_-}{}^2$, by homogeneity of $V$, for all $v \in V$, we have $\pnorm{v_-}{}\geq \pnorm{v}{}\epsilon_0$. This means for any $R>0$, $\{v \in V, \pnorm{v}{}\geq R\}\subset \{v \in \R^m, \pnorm{v_-}{}\geq R\epsilon_0\}$ and therefore 
	\begin{align}\label{ineq:logistic_exist_prop_1}
	\min_{v \in V, \pnorm{v}{}\geq R} \sum_{i=1}^m \log\big(1+e^{-v_i}\big)&\geq \min_{v \in \R^m, \pnorm{v_-}{}\geq R\epsilon_0 }\sum_{i=1}^m \log\big(1+e^{-v_i}\big)\nonumber\\
	&\geq \log\big(1+ e^{R\epsilon_0}\big)\geq R\epsilon_0. 
	\end{align}
	On the other hand, as $0\in V$, we always have 
	\begin{align}\label{ineq:logistic_exist_prop_2}
	\min_{v \in V}\sum_{i=1}^m \log\big(1+e^{-v_i}\big)\leq m (\log 2).
	\end{align}
	Combining (\ref{ineq:logistic_exist_prop_1}) and (\ref{ineq:logistic_exist_prop_2}), if $R>0$ is chosen such that $R>m(\log 2)/\epsilon_0$, then the set $\{v\in V, \pnorm{v}{}\geq R\}$ is not feasible in the minimization problem $v\mapsto \sum_{i=1}^m \log\big(1+e^{-v_i}\big), v \in V$. So on the event
	\begin{align*}
	E(t)\equiv \bigg\{\inf_{\mu \in K\cap \partial B_n} \pnorm{G\mu}{}\geq \big(\sqrt{m}-\sqrt{\delta(K)}-C\sqrt{t}\big)_+\bigg\},\quad t\geq 1,
	\end{align*}
	the set $\{\beta \in K: \pnorm{\beta}{}\geq R/\big(\sqrt{m}-\sqrt{\delta(K)}-C\sqrt{t}\big)_+ \}$ is not feasible in the minimization problem $\beta \mapsto \sum_{i=1}^m \log\big(1+e^{-X_i^\top \beta}\big)$, so the cone constrained MLE $\hat{\beta}_K$ exists provided that $\sqrt{m}-\sqrt{\delta(K)}-C\sqrt{t}>0$. Summarizing the arguments above, we have shown
	\begin{align*}
	\big\{GK\cap \R_{\geq 0}^m = \{0\}\big\}\cap E(t)\subset \big\{\hbox{The cone constrained MLE $\hat{\beta}_K$ exists}\big\},
	\end{align*}
	provided $t\geq 1$ is chosen such that $\sqrt{m}-\sqrt{\delta(K)}-C\sqrt{t}>0$. Using Lemma \ref{lem:cone_eigenvalue} below, the above display further implies
	\begin{align*}
	&\Prob\big(\hbox{The cone constrained MLE $\hat{\beta}_K$ does not exist}\big)\\
	&\leq \Prob\big(GK\cap \R_{\geq 0}^m \neq \{0\}\big)+e^{-t},
	\end{align*}
	provided $t\geq 1$ is chosen such that $\sqrt{m}-\sqrt{\delta(K)}-C\sqrt{t}>0$. Now taking $t=c \big(\sqrt{m}-\sqrt{\delta(K)}\big)^2$ for a small enough absolute constant $c>0$ to conclude when $\sqrt{m}-\sqrt{\delta(K)}\geq C$. The case $\sqrt{m}-\sqrt{\delta(K)}< C$ follows by simply adjusting constants.
\end{proof}

The proof of the above Proposition \ref{prop:logistic_exist} relies on the following result that gives tight bounds on certain smallest cone-constrained eigenvalues, which may be of independent interest. 

\begin{lemma}\label{lem:cone_eigenvalue}
	Let $K\subset \R^n$ be a non-trivial closed convex cone. There exists some universal constant $C>0$ such that for all $t\geq 1$, 
	\begin{align*}
	\Prob\bigg(\inf_{\mu \in K\cap \partial B_n} \pnorm{G\mu}{}\geq \big(\sqrt{m}-\sqrt{\delta(K)}-C\sqrt{t}\big)_+\bigg)\geq 1-e^{-t}.
	\end{align*}
\end{lemma}
\begin{proof}
	As $\inf_{\mu \in K\cap \partial B_n} \pnorm{G\mu}{} = \inf_{\mu \in K\cap \partial B_n} \sup_{v \in B_m} \iprod{v}{G\mu}$, by using the one-sided Gaussian min-max theorem (cf. Theorem \ref{thm:CGMT}-(1)), we have for any $z \in \R$,
	\begin{align*}
	\Prob\bigg(\inf_{\mu \in K\cap \partial B_n} \pnorm{G\mu}{}\leq z\bigg)&=\Prob\bigg(\inf_{\mu \in K\cap \partial B_n} \sup_{v \in B_m} \iprod{v}{G\mu}\leq z\bigg)\\
	&\leq 2\Prob\bigg(\inf_{\mu \in K\cap \partial B_n} \sup_{v \in B_m}\Big\{\pnorm{\mu}{}\iprod{h}{v}-\pnorm{v}{}\iprod{g}{\mu}\Big\}\leq z\bigg). 
	\end{align*}
	On the other hand, note that
	\begin{align*}
	&\inf_{\mu \in K\cap \partial B_n} \sup_{v \in B_m}\Big\{\pnorm{\mu}{}\iprod{h}{v}-\pnorm{v}{}\iprod{g}{\mu}\Big\}= \inf_{\mu \in K\cap \partial B_n} \sup_{0\leq \beta \leq 1} \beta\big(\pnorm{h}{}-\iprod{g}{\mu}\big)\\
	& = \inf_{\mu \in K\cap \partial B_n}\big(\pnorm{h}{}-\iprod{g}{\mu}\big)_+ = \bigg\{\pnorm{h}{}-\sup_{\mu \in K\cap \partial B_n}\iprod{g}{\mu}\bigg\}_+.
	\end{align*}
	Combining the above two displays, we have for any $z \in \R$,
	\begin{align}\label{ineq:cone_eigenvalue_1}
	\Prob\bigg(\inf_{\mu \in K\cap \partial B_n} \pnorm{G\mu}{}\leq z\bigg)\leq 2 \Prob\bigg(\bigg\{\pnorm{h}{}-\sup_{\mu \in K\cap \partial B_n}\iprod{g}{\mu}\bigg\}_+\leq z\bigg).
	\end{align}
	By Gaussian concentration in Proposition \ref{prop:gaussian_conc_generic}, there exists some universal constant $C_0>0$ such that for any $t\geq 1$, on an event $E(t)$ with probability at least $1-e^{-t}$, 
	\begin{align*}
	\bigabs{\pnorm{h}{}-\sqrt{m}}\vee \bigabs{\sup_{\mu \in K\cap \partial B_n}\iprod{g}{\mu}-\gw(K\cap \partial B_n) }\leq C_0\sqrt{t}. 
	\end{align*}
	Consequently, with $z(t)\equiv \big(\sqrt{m}-\gw(K\cap \partial B_n)-4C_0\sqrt{t}\big)_+$, using (\ref{ineq:cone_eigenvalue_1}), it holds for any $t\geq 1$ that
	\begin{align*}
	\Prob\bigg(\inf_{\mu \in K\cap \partial B_n} \pnorm{G\mu}{}\leq z(t)\bigg)\leq 2\Prob(E(t)^c)\leq 2e^{-t}.
	\end{align*}
	The claim now follows by adjusting constants using Proposition \ref{prop:stat_dim}-(3)(4). 
\end{proof}

Now we may prove Theorem \ref{thm:logistic_cone_MLE}.

\begin{proof}[Proof of Theorem \ref{thm:logistic_cone_MLE}]
	\noindent (1). For $\alpha \in [0,\pi/2)$, let 
	\begin{align*}
	L_+(\alpha)\equiv \Big\{v \in \R^m: \min_{i \in [m]} v_i\geq \sin(\alpha)\cdot \pnorm{v}{}\Big\}.
	\end{align*}
	Then $\{L_+(\alpha): \alpha \in [0,\pi/2]\}$ is a non-decreasing sequence of closed convex cones in $\R^m$ as $\alpha\downarrow 0$, with $L_+(0)=\R_{\geq 0}^m$ being the non-negative orthant cone. Consequently, $\{L_k\equiv L_+(1/k)\cap B_m: k \in \N\}$ is a non-decreasing sequence of compact sets. It is easy to verify that $\mathrm{cl}(\cup_{k \in \N} L_k)=\R_{\geq 0}^m\cap B_m$. Now using Lemma \ref{lem:sup_set_conv} and integrability of the suprema of Gaussian processes, we have
	\begin{align}\label{ineq:logistic_1}
	\delta(L_+(1/k)) = \E \bigg(\sup_{v \in L_k}\iprod{h}{v}\bigg)^2\uparrow  \E \bigg(\sup_{v \in \R_{\geq 0}^m\cap B_m}\iprod{h}{v}\bigg)^2 = \delta(\R_{\geq 0}^m)=\frac{m}{2}. 
	\end{align}
	The last identity $\delta(\R_{\geq 0}^m)=m/2$ follows from, e.g., \cite[Table 3.1]{amelunxen2014living}. On the other hand, as
	\begin{align*}
	\big\{ GK\cap \R_{> 0}^m = \emptyset\big\}\subset \big\{ GK\cap L_+(1/k) = \{0\}\big\}
	\end{align*}
	holds for any $k \in \N$, by Proposition \ref{prop:logistic_exist}-(1), 
	\begin{align}\label{ineq:logistic_2}
	&\Prob\big(\hbox{The cone constrained MLE $\hat{\beta}_K$ exists}\big)\nonumber\\
	&\leq \Prob\big(GK\cap \R_{> 0}^m = \emptyset\big)\leq \Prob\big(GK\cap L_+(1/k) = \{0\}\big).
	\end{align}
	Now using Theorem \ref{thm:approx_kinematics}-(2), there exists some universal constant $C_0>0$ such that for all $t\geq 1$,
	\begin{align}\label{ineq:logistic_3}
	&\Prob\big(GK\cap L_+(1/k) = \{0\}\big)\leq \bm{1}\Big(\sqrt{\delta(K)}<\sqrt{m-\delta(L_+(1/k))}+C_0\sqrt{t}\Big)+e^{-t}.
	\end{align}
	Now under the assumed condition, the convergence in (\ref{ineq:logistic_1}) entails that
	\begin{align*}
	\sqrt{\delta(K)}\geq \sqrt{m/2}+2C_0\sqrt{t}\geq \sqrt{m-\delta(L_+(1/k))}+C_0\sqrt{t}
	\end{align*}
	for all $k$ large enough (that may depend on $m,t$), so combined with (\ref{ineq:logistic_2}) and (\ref{ineq:logistic_3}) we have
	\begin{align*}
	\Prob\big(\hbox{The cone constrained MLE $\hat{\beta}_K$ exists}\big)\leq e^{-t},
	\end{align*}
	as desired. 
	
	\noindent (2). Using Proposition \ref{prop:logistic_exist}-(2) and Theorem \ref{thm:approx_kinematics}-(1) and the fact that $\delta(\R_{\geq 0}^m)=m/2$, there exist some universal constants $C_1,c>0$ such that for all $t\geq 1$,
	\begin{align*}
	&\Prob\big(\hbox{The cone constrained MLE $\hat{\beta}_K$ does not exist}\big)\\
	&\leq \bm{1}\Big(\sqrt{\delta(K)}>\sqrt{m/2}-C_1\sqrt{t}\Big)+e^{-t}+2e^{-c(\sqrt{m}-\sqrt{\delta(K)})^2}. 
	\end{align*}
	Now under the condition that $\sqrt{\delta(K)}\leq \sqrt{m/2}-C_1\sqrt{t}$, the above display can be further simplified as
	\begin{align*}
	\Prob\big(\hbox{The cone constrained MLE $\hat{\beta}_K$ does not exist}\big)\leq e^{-t}+2e^{-(c/4)m}. 
	\end{align*}
	The second term on the right hand side of the above display can be assimilated into the first term upon adjusting constants, as the effective range of $t$ is $1\leq t\leq m/(2C_1^2)$. This completes the proof.
\end{proof}

\subsection{Proofs of Theorem \ref{thm:conic_program}}\label{section:proof_conic_program}

For any $x \in \partial B_n$, recall $
K_x= K\cap \{\mu \in \R^n: \iprod{x}{\mu}\geq 0\}$ defined in (\ref{def:K_x}), and let for any further $a>0$
\begin{align*}
K(x;a)&\equiv \mathrm{cl}\big(\big\{\mu/\pnorm{\mu}{}: \iprod{x}{\mu}=a, \mu \in K\setminus \{0\}\big\}\big)\subset \partial B_n. 
\end{align*}
\begin{lemma}\label{lem:conic_program_1}
	For any $x \in \partial B_n$ with $x \notin K^\circ$, we have $K(x;a)=K_x\cap \partial B_n$ for any $a>0$. 
\end{lemma}
\begin{proof}
	The fact that $K(x;a_1)=K(x;a_2)$ for any $a_1,a_2>0$ follows from the cone property of $K$, so we shall focus on the case $K(x;1)$ with $a=1$ without loss of generality. By definition of $K(x;1)$ we have $K(x;1)\subset K_x\cap \partial B_n$, so we only need to prove the converse direction. Upon using a suitable orthonormal transformation, we may assume without loss of generality that $x = e_1$. This means that
	\begin{align*}
	K_x= K\cap \{v \in \R^n: v_1\geq 0\},\quad K(x;1)=\mathrm{cl}\big(\big\{\mu/\pnorm{\mu}{}: \mu_1=1, \mu \in K\setminus \{0\}\big\}\big).
	\end{align*}
	Take any $v \in K_x\cap \partial B_n$. By the assumption $x \notin K^\circ$, there exists some $\mu \in K$ such that $\mu_1=\iprod{x}{\mu}>0$. This means $K\cap \{v \in \R^n: v_1>0\}\neq \emptyset$. So for $v \in K_x \cap \partial B_n\subset K\cap \{v \in \R^n: v_1\geq 0\}$, we may find some sequence $\{v^{(\ell)}\}_{\ell \in \N}\subset K\cap \{v \in \R^n: v_1>0\}$ such that $v^{(\ell)} \to v$. This convergence implies the norm convergence $\pnorm{v^{(\ell)}}{}\to \pnorm{v}{}=1$, and consequently $v^{(\ell)}/\pnorm{v^{(\ell)}}{}\to v$. Now by defining $\mu^{(\ell)}\equiv v^{(\ell)}/v^{(\ell)}_1$, we have $\mu^{(\ell)}_1=1,\mu^{(\ell)}\in K\setminus \{0\}$ and $\mu^{(\ell)}/\pnorm{\mu^{(\ell)}}{}= v^{(\ell)}/\pnorm{v^{(\ell)}}{}\to v$. This verifies that $v \in K(x;1)$, as desired.
\end{proof}

With all the preparations, we may prove Theorem \ref{thm:conic_program}. 
\begin{proof}[Proof of Theorem \ref{thm:conic_program}]
	For any $R>0$, let $K_n(R)\equiv K\cap B_n(R)$. Using (\ref{ineq:support_fcn_cgmt_01}) in Proposition \ref{prop:support_fcn_cgmt}, we have for any $x \in \partial B_n$ and $z \in \R$,
	\begin{align}\label{ineq:conic_program_0}
	\Prob\Big(\max_{\mu \in K_n(R): G\mu \in L}\iprod{x}{\mu}\geq z\Big)&\leq 2 \Prob\bigg(\sup_{\mu \in K_n(R), \iprod{g}{\mu}\geq \pnorm{\mu}{} \pnorm{\proj_L^\perp h}{}} \iprod{x}{\mu}\geq z\bigg),\nonumber\\
	\Prob\Big(\max_{\mu \in K_n(R): G\mu \in L }\iprod{x}{\mu}< z\Big)&\leq 2 \Prob\bigg(\sup_{\mu \in K_n(R), \iprod{g}{\mu}\geq \pnorm{\mu}{} \pnorm{\proj_L^\perp h}{}} \iprod{x}{\mu}< z\bigg).
	\end{align}	
	Note that $\max_{\mu \in K_n(R): G\mu \in L}\iprod{x}{\mu}$ and $\sup_{\mu \in K_n(R), \iprod{g}{\mu}\geq \pnorm{\mu}{} \pnorm{\proj_L^\perp h}{}} \iprod{x}{\mu}$ either take value $0$ for all $R>0$, or blow up to $\infty$ as $R\uparrow \infty$. Further note that 
	\begin{align*}
	\sup_{\mu \in K, \iprod{g}{\mu}\geq \pnorm{\mu}{} \pnorm{\proj_L^\perp h}{}} \iprod{x}{\mu}& = 0\vee \sup \bigg\{a>0: \sup_{\mu \in K, \iprod{x}{\mu}=a} \big(\iprod{g}{\mu}-\pnorm{\mu}{}\pnorm{\proj_L^\perp h}{}\big)\geq 0\bigg\}\\
	& = 0 \vee \sup\bigg\{a>0: \sup_{\mu \in K(x;a)}\iprod{g}{\mu}\geq \pnorm{\proj_L^\perp h}{}\bigg\} \\
	&= \sup_{\mu \in K} \iprod{x}{\mu}\cdot\bm{1}\big(\sup_{\mu \in K_x\cap \partial B_n}\iprod{g}{\mu}\geq \pnorm{\proj_L^\perp h}{}\big),
	\end{align*}
	where in the last equality we used Lemma \ref{lem:conic_program_1}, and we interpret $\infty\cdot 0=0$.  Now taking limit as $R\uparrow \infty$ for free in (\ref{ineq:conic_program_0}), we have for any $x \in \partial B_n$ and $z \in \R$,
	\begin{align}\label{ineq:conic_program_1}
	\Prob\Big(\max_{\mu \in K: G\mu \in L}\iprod{x}{\mu}\geq z\Big)&\leq 2 \Prob\bigg(\sup_{\mu \in K} \iprod{x}{\mu}\cdot\bm{1}\big(\sup_{\mu \in K_x\cap \partial B_n}\iprod{g}{\mu}\geq \pnorm{\proj_L^\perp h}{}\big)\geq z\bigg),\nonumber\\
	\Prob\Big(\max_{\mu \in K: G\mu \in L}\iprod{x}{\mu}< z\Big)&\leq 2 \Prob\bigg(\sup_{\mu \in K} \iprod{x}{\mu}\cdot\bm{1}\big(\sup_{\mu \in K_x\cap \partial B_n}\iprod{g}{\mu}\geq \pnorm{\proj_L^\perp h}{}\big)< z\bigg).
	\end{align}
	Similar arguments using (\ref{ineq:support_fcn_cgmt_01}) in Proposition \ref{prop:support_fcn_cgmt} leads the one-sided inequality for the uniform version: for any $z \in \R$,
	\begin{align}\label{ineq:conic_program_1_1}
	\Prob\Big(\max_{\mu \in K: G\mu \in L}\pnorm{\mu}{}\geq z\Big)&\leq 2 \Prob\bigg(\sup_{\mu \in K, \iprod{g}{\mu}\geq \pnorm{\mu}{} \pnorm{\proj_L^\perp h}{}} \pnorm{\mu}{}\geq z\bigg).
	\end{align}

	\noindent (\textbf{Homogeneous case $b=0$}). In this case $L=\{0\}$ so $\proj_L^\perp h=h$. Clearly for $x \in K^\circ$, $\max_{\mu \in K: G\mu=0}\iprod{x}{\mu}=0$ regardless of the relationship between $m$ and $K$. Now we consider the case $x \notin K^\circ$. By Lemma \ref{lem:cone_polarity}, $\sup_{\mu \in K}\iprod{x}{\mu}=\infty$, so by taking $z=\epsilon>0$ in the two displays in (\ref{ineq:conic_program_1}), 
	\begin{align}\label{ineq:conic_program_2}
	\Prob\Big( \max_{\mu \in K: G\mu=0}\iprod{x}{\mu}\geq \epsilon\Big)&\leq 2\Prob\bigg(\sup_{\mu \in K_x\cap \partial B_n}\iprod{g}{\mu}\geq \pnorm{h}{}\bigg),\nonumber\\
	\Prob\Big( \max_{\mu \in K: G\mu=0}\iprod{x}{\mu}<\epsilon\Big)&\leq 2\Prob\bigg(\sup_{\mu \in K_x\cap \partial B_n}\iprod{g}{\mu}< \pnorm{h}{}\bigg).
	\end{align}
	As $x\notin K^\circ$ and $K$ is non-trivial, $K_x$ is also non-trivial. By Gaussian concentration in Proposition \ref{prop:gaussian_conc_generic} and Proposition \ref{prop:stat_dim}-(3)(4), there exists some universal constant $C>0$ such that for any $t\geq 1$,
	\begin{align}\label{ineq:conic_program_3}
	\bigabs{\sup_{\mu \in K_x\cap \partial B_n}\iprod{g}{\mu}- \sqrt{\delta(K_x)}}\vee \bigabs{\pnorm{h}{}-\sqrt{m}}\leq C\sqrt{t}.
	\end{align}
	Combining the above two displays (\ref{ineq:conic_program_2})-(\ref{ineq:conic_program_3}) and noting that $\max_{\mu \in K: G\mu=0}\iprod{x}{\mu}\in \{0,\infty\}$, for $x \notin K^\circ$ and any $ t\geq 1$, 
	\begin{align}\label{ineq:conic_program_4}
	\sqrt{m}\geq \sqrt{\delta(K_x)}+C\sqrt{t} \,\Rightarrow\, \Prob\Big( \max_{\mu \in K: G\mu=0}\iprod{x}{\mu}=0\Big)\geq 1-e^{-t},\nonumber\\
	\sqrt{m}\leq \sqrt{\delta(K_x)}-C\sqrt{t}\,\Rightarrow\, \Prob\Big( \max_{\mu \in K: G\mu=0}\iprod{x}{\mu}=\infty\Big)\geq 1-e^{-t}. 
	\end{align}
	The claim now follows by noting that the case $x\in K^\circ$ can be merged into the first inequality above. 
	
	\noindent (\textbf{In-homogeneous case $b\neq 0$}). In this case we take $L\equiv L_b\equiv \{tb:t\geq 0\}$. As 
	\begin{align*}
	\sup_{\mu \in K, \iprod{g}{\mu}\geq \pnorm{\mu}{} \pnorm{\proj_{L_b}^\perp h}{}} \pnorm{\mu}{} =
	\begin{cases}
	0,& \sup_{\mu \in K\cap \partial B_n} \iprod{g}{\mu}< \pnorm{\proj_{L_b}^\perp h}{}\\
	\infty, & \sup_{\mu \in K\cap \partial B_n} \iprod{g}{\mu}\geq \pnorm{\proj_{L_b}^\perp h}{}
	\end{cases}
	,
	\end{align*}
	by using (\ref{ineq:conic_program_1_1}) it follows that for any $\epsilon>0$,
	\begin{align*}
	\Prob\Big(\sup_{\mu \in K: G\mu \in L_b}\pnorm{\mu}{}\geq \epsilon\Big)&\leq 2 \Prob\bigg(\sup_{\mu \in K\cap \partial B_n} \iprod{g}{\mu}\geq \pnorm{\proj_{L_b}^\perp h}{}\bigg).
	\end{align*}
	Taking $\epsilon\downarrow 0$ and using similar Gaussian concentration as in (\ref{ineq:conic_program_3}), there exists some universal constant $C>0$ such that for any $t\geq 1$,
	\begin{align}\label{ineq:conic_program_5}
	\sqrt{m}\geq \sqrt{\delta(K)}+C\sqrt{t}\,\Rightarrow\, \Prob\Big(\sup_{\mu \in K: G\mu \in L_b}\pnorm{\mu}{}=0\Big)\geq 1-e^{-t}. 
	\end{align}
	On the other hand, for $K\neq \R^n$, we may find some $x \in \partial B_n$ such that $K_x=K$, so the second inequality of (\ref{ineq:conic_program_4}) entails 
	\begin{align}\label{ineq:conic_program_6}
	\sqrt{m}\leq \sqrt{\delta(K)}-C\sqrt{t}\,\Rightarrow\, \Prob\Big(\sup_{\mu \in K: G\mu \in L_b}\pnorm{\mu}{}=\infty\Big)\geq 1-e^{-t}. 
	\end{align}
	If $K=\R^n$, the above display also holds as $G \R^n \cap L_b \neq \{0\}$ with the prescribed probability by Theorem \ref{thm:approx_kinematics}. Now combining (\ref{ineq:conic_program_4})-(\ref{ineq:conic_program_6}), we have the following:
	\begin{enumerate}
		\item If $\sqrt{m}\geq \sqrt{\delta(K)}+C\sqrt{t}$, then
		\begin{align*}
		\Prob\Big(\big\{\mu \in K: G\mu \in L_b\big\}=\{0\}\Big)\geq 1-e^{-t}.
		\end{align*}
		In this case, the CP (\ref{def:conic_program}) is infeasible. 
		\item If $\sqrt{\delta(K_x)}+C\sqrt{t}\leq \sqrt{m}\leq \sqrt{\delta(K)}-C\sqrt{t}$, then 
		\begin{align*}
		\Prob\Big(\big\{\mu \in K: G\mu \in L_b\big\}\neq\{0\}, \max_{\mu \in K: G\mu \in L_b}\iprod{x}{\mu}=0\Big)\geq 1-e^{-t}.
		\end{align*}
		In this case, the CP (\ref{def:conic_program}) is feasible with a non-positive cost optimum. 
		\item If $\sqrt{m}\leq \sqrt{\delta(K_x)}-C\sqrt{t}$, then 
		\begin{align*}
		\Prob\Big( G K_x \cap L_b\neq \{0\}, \max_{\mu \in K: G\mu=0}\iprod{x}{\mu}=\infty
		\Big)\geq 1-e^{-t}.
		\end{align*}
		In this case, the CP (\ref{def:conic_program}) is feasible with $\max_{\mu \in K: G\mu=b}\iprod{x}{\mu}=\infty$. 
	\end{enumerate}
    The claim follows. 
\end{proof}

\subsection{Proof of Theorem \ref{thm:local_dm}}\label{section:proof_DM}

We first give an upper estimate for $\mathsf{h}_{L\cap G K_n}(x)$.

\begin{proposition}\label{prop:DM_upper}
	Suppose that $K \subset \R^n$ is a closed convex cone and $L\subset \R^m$ is a subspace with $K\neq \{0\}, L\neq \{0\}$. Fix $x \in L\cap \partial B_m$. Then there exists some universal constant $C>0$ such that for any $t\geq 1$,
	\begin{align*}
	\Prob\Big(\mathsf{h}^2_{L\cap G K_n}(x)>\Big\{\delta(K)-\delta(L^\circ)+C\sqrt{t}\cdot \Big(\sqrt{\delta(K)}\vee \sqrt{\delta(L^\circ)}\Big)+C t\Big\}_+\Big)\leq e^{-t}. 
	\end{align*}
	Here $K_n\equiv K\cap B_n$. 
\end{proposition}
\begin{proof}
	We shall write $\mathsf{h}\equiv \mathsf{h}_{L\cap G K_n}$ for notational convenience in the proof.  Using the same arguments as in the proof of (\ref{ineq:pointwise_support_fcn}) (a pointwise version of Proposition \ref{prop:support_fcn_proj_upper}) up to (\ref{ineq:support_fcn_proj_upper_1_2}),
	for any $z \in \R$ and $R>0$,
	\begin{align}\label{ineq:DM_upper_1}
	\Prob\big(\mathsf{h}(x)>z\big)
	&\leq 2 \Prob\bigg(\sup_{\mu \in K_n}  \inf_{1\leq \beta \leq R} \beta\bigg\{\iprod{g}{\mu}- \pnorm{\mu}{}\sup_{w \in L^\circ(\beta;x)} \iprod{h}{w} \bigg\}>z   \bigg),
	\end{align}
	where recall $L^\circ(\beta;x)=\{(v-x)/\pnorm{v-x}{}:v \in L^\circ, \pnorm{v-x}{}=\beta\}$ is defined in (\ref{ineq:support_fcn_proj_upper_1_1}). For subspace $L$ and $x \in L\cap \partial B_m$, $L^\circ(\beta;x)$ can be computed exactly: With $\ell^\circ\equiv \dim(L^\circ)+1\leq m$, there exists some orthogonal matrix $O_x \in \R^{m\times m}$ such that 
	\begin{align*}
	L^\circ(\beta;x)=O_x\big\{v \in \R^m: \pnorm{v}{}=1, v|_{(\ell^\circ:m]}=0, v_1 = 1/\beta\big\}.
	\end{align*}
	Consequently, with $h^x\equiv O_x^\top h \sim \mathcal{N}(0,I_m)$,
	\begin{align}\label{ineq:DM_upper_2}
	&\sup_{\mu \in K_n} \inf_{1\leq \beta \leq R} \beta\bigg\{\iprod{g}{\mu}- \pnorm{\mu}{}\sup_{w \in L^\circ(\beta;x)} \iprod{h}{w} \bigg\}\nonumber\\
	&=\sup_{\mu \in K_n} \inf_{1\leq \beta \leq R} \beta\bigg\{\iprod{g}{\mu}- \pnorm{\mu}{}\cdot \bigg(\frac{1}{\beta}h_1^x + \sqrt{1-\frac{1}{\beta^2}} \cdot \pnorm{h_{[2:\ell^\circ]}^x }{}\bigg) \bigg\}\nonumber\\
	&\leq \sup_{\mu \in K,\iprod{\mu}{g}\geq 0} \inf_{1\leq \beta \leq R} \Big\{ \beta\iprod{g}{\mu}- \sqrt{\beta^2-1}\cdot \pnorm{\mu}{} \pnorm{h_{[2:\ell^\circ]}^x }{} \Big\}+\abs{h_1^x}.
	\end{align}
	By Lemma \ref{lem:fcn_Q} below, the first term above can be further bounded from above by
	\begin{align*}
	&\bigg\{\sup_{\mu \in K_n,\iprod{\mu}{g}\geq 0}\Big(\iprod{g}{\mu}^2-\pnorm{\mu}{}^2\pnorm{h_{[2:\ell^\circ]}^x }{}^2\Big)_+ \bigg\}^{1/2}+ \frac{C \pnorm{h_{[2:\ell^\circ]}^x }{} }{R}\\
	& = \bigg\{\bigg(\sup_{\mu \in K_n}\iprod{g}{\mu}\bigg)^2- \pnorm{h_{[2:\ell^\circ]}^x }{}^2\bigg\}_+^{1/2}+ \frac{C \pnorm{h_{[2:\ell^\circ]}^x }{} }{R}.
	\end{align*}
	Combined with (\ref{ineq:DM_upper_2}), we have
	\begin{align*}
	&\sup_{\mu \in K_n} \inf_{1\leq \beta \leq R} \beta\bigg\{\iprod{g}{\mu}- \pnorm{\mu}{}\sup_{w \in L^\circ(\beta;x)} \iprod{h}{w} \bigg\}\nonumber\\
	&\leq \bigg\{\bigg(\sup_{\mu \in K_n}\iprod{g}{\mu}\bigg)^2- \pnorm{h_{[2:\ell^\circ]}^x }{}^2\bigg\}_+^{1/2} + \frac{C \pnorm{h_{[2:\ell^\circ]}^x }{} }{R}+\abs{h_1^x}.
	\end{align*}
	Using (\ref{ineq:DM_upper_1}) and the above display, we have 
	\begin{align*}
	\Prob\big(\mathsf{h}(x)>z\big)
	&\leq 2 \Prob\bigg(\bigg\{\bigg(\sup_{\mu \in K_n}\iprod{g}{\mu}\bigg)^2- \pnorm{h_{[2:\ell^\circ]}^x }{}^2\bigg\}_+^{1/2} + \frac{C \pnorm{h_{[2:\ell^\circ]}^x }{} }{R}+\abs{h_1^x}>z   \bigg).
	\end{align*}
	Taking $R\uparrow \infty$, we obtain the estimate
	\begin{align}\label{ineq:DM_upper_3}
	\Prob\big(\mathsf{h}(x)>z\big)
	&\leq 2 \Prob\bigg(\bigg\{\bigg(\sup_{\mu \in K_n}\iprod{g}{\mu}\bigg)^2- \bigg(\sup_{v \in L^\circ\cap B_m} \iprod{h}{v}\bigg)^2\bigg\}_+^{1/2} +\abs{h_1^x}\geq z   \bigg).
	\end{align}
	Using Gaussian concentration as in Proposition \ref{prop:gaussian_conc_generic}, there exists some universal constant $C_0>0$ such that for any $t\geq 1$, with probability at least $1-e^{-t}$,
	\begin{align*}
	&\bigg\{\bigg(\sup_{\mu \in K\cap B_n}\iprod{g}{\mu}\bigg)^2-\bigg(\sup_{v \in L^\circ\cap B_m} \iprod{h}{v}\bigg)^2\bigg\}_+^{1/2}\\
	& \leq \Big(\gw(K\cap B_n)-\gw(L^\circ\cap B_m)+C_0\sqrt{t}\Big)_+^{1/2}\cdot \Big(\gw(K\cap B_n)+\gw(L^\circ\cap B_m)+C_0\sqrt{t}\Big)_+^{1/2}\\
	& \leq \Big\{\delta(K)-\delta(L^\circ)+C_0\sqrt{t}\cdot \Big(\sqrt{\delta(K)}\vee \sqrt{\delta(L^\circ)}\Big)+C_0 t\Big\}_+^{1/2}.
	\end{align*}
	Consequently the right hand side of (\ref{ineq:DM_upper_3}) can be bounded by
	\begin{align*}
	\bm{1}\bigg(\Big\{\delta(K)-\delta(L^\circ)+C\sqrt{t}\cdot \Big(\sqrt{\delta(K)}\vee \sqrt{\delta(L^\circ)}\Big)+C t\Big\}_+^{1/2}\geq z \bigg)+ Ce^{-t/C}.
	\end{align*}
	The proof is complete by possibly adjusting constants. 
\end{proof}

The proof of Proposition \ref{prop:DM_upper} above makes use of the following result. 

\begin{lemma}\label{lem:fcn_Q}
	Let $a=(a_1,a_2)^\top \in \R_{\geq 0}^2$, and let $\mathsf{Q}_a:\R_{\geq 1}\to \R$ be defined by 
	\begin{align*}
	\mathsf{Q}_a(\beta)\equiv a_1\beta-a_2\sqrt{\beta^2-1},\, \beta\geq 1.
	\end{align*}
	Then there exists some universal constant $C>0$ such that for any $R>1$,
	\begin{align*}
	\bigg(\inf_{1\leq \beta \leq R}\mathsf{Q}_a(\beta)-\sqrt{(a_1^2-a_2^2)_+}\bigg)_+\leq \frac{Ca_2}{R}. 
	\end{align*}
	Moreover,
	\begin{align*}
	\inf_{\beta \geq 1}\mathsf{Q}_a(\beta)=
	\begin{cases}
	\sqrt{a_1^2-a_2^2},& a_1\geq a_2,\\
	-\infty, & a_1<a_2.
	\end{cases}
	\end{align*}
\end{lemma}
\begin{proof}
	The derivative of $\mathsf{Q}_a$ is easily computed as
	\begin{align*}
	\mathsf{Q}_a'(\beta)= a_1-a_2\frac{\beta}{\sqrt{\beta^2-1}},\, \beta>1.
	\end{align*}
	So $\mathsf{Q}_a'$ is non-decreasing with
	\begin{align*}
	\lim_{\beta\downarrow 1} \mathsf{Q}_a'(\beta)=-\infty,\quad \mathsf{Q}_a'(R)=a_1-a_2\frac{R}{\sqrt{R^2-1}},\quad\lim_{\beta \uparrow \infty} \mathsf{Q}_a'(\beta)=a_1-a_2. 
	\end{align*}
	\noindent \textbf{Case 1}. Suppose $\mathsf{Q}_a'(R)\geq 0$. Then the global infimum of $\mathsf{Q}_a$ is attained at some $\beta_\ast \in [1,R]$ that solves $\mathsf{Q}_a'(\beta_\ast)=0$, i.e., $\beta_\ast=\sqrt{a_1^2/(a_1^2-a_2^2)}$. In other words,
	\begin{align*}
	\inf_{1\leq \beta \leq R} \mathsf{Q}_a'(\beta) = \mathsf{Q}_a'(\beta_\ast)=\sqrt{a_1^2-a_2^2}.
	\end{align*}
	\noindent \textbf{Case 2}. Suppose $\mathsf{Q}_a'(R)<0$. Then $\mathsf{Q}_a'$ is globally non-positive on $[1,R]$, and therefore for any $R>1$, as $a_1<a_2(1+O(1/R^2))$,
	\begin{align*}
	\inf_{1\leq \beta \leq R} \mathsf{Q}_a(\beta) = \mathsf{Q}_a(R) =R\bigg(a_1-a_2 \sqrt{1-\frac{1}{R^2}}\bigg)\leq \frac{C a_2}{R}. 
	\end{align*}
	In this case, as $\sqrt{(a_1^2-a_2^2)_+}\leq Ca_2/R$, 
	\begin{align*}
	\bigg(\inf_{1\leq \beta \leq R}\mathsf{Q}_a(\beta)-\sqrt{(a_1^2-a_2^2)_+}\bigg)_+\leq \frac{Ca_2}{R}. 
	\end{align*}
	Combining the two cases concludes the proof for the claim for finite $R>1$. The case for $R=\infty$ can be argued similarly. In particular, for $a_1>a_2$, the global minimum is computed exactly in the same way as in Case 1 above. For the case $a_1\leq a_2$, 
	\begin{align*}
	\inf_{\beta \geq 1}\mathsf{Q}_a(\beta)&= \lim_{\beta \uparrow\infty} \mathsf{Q}_a(\beta) = \lim_{\beta \uparrow\infty} \beta\cdot \Big(a_1-a_2\sqrt{1-\beta^{-2}}\Big)=
	\begin{cases}
	0,& a_1=a_2;\\
	-\infty, & a_1<a_2.
	\end{cases}
	\end{align*}
	The claim follows by noting that the case $a_1=a_2$ can be assimilated into the first case. 
\end{proof}

We next give a lower estimate for $\mathsf{h}_{L\cap G K_n}(x)$. 
\begin{proposition}\label{prop:DM_lower}
	Suppose that $K \subset \R^n$ is a closed convex cone and $L\subset \R^m$ is a subspace with $K\neq \{0\}, L\neq \{0\}$. Fix $x \in L\cap \partial B_m$. Then there exists some universal constant $C>0$ such that for any $t\geq 1$,
	\begin{align*}
	\Prob\Big(\mathsf{h}^2_{L\cap G K_n}(x)<\Big\{\delta(K)-\delta(L^\circ)-C\sqrt{t}\cdot \Big(\sqrt{\delta(K)}\vee \sqrt{\delta(L^\circ)}\Big)-C t\Big\}_+\Big)\leq e^{-t}. 
	\end{align*}
	Here $K_n\equiv K\cap B_n$. 
\end{proposition}
\begin{proof}
	We will use the estimate in (\ref{ineq:support_fcn_proj_lower_2}) to compute the quantity in the right hand side therein: Recall $L^\circ(\beta;x)$ and $h^x$ as in the proof of Proposition \ref{prop:DM_upper}. We have
	\begin{align*}
	&\sup_{\mu \in K_n}\inf_{v \in L^\circ} \Big\{-\pnorm{\mu}{}\iprod{h}{v-x}+\pnorm{v-x}{}\iprod{g}{\mu} \Big\}\\
	& =\sup_{\mu \in K_n} \bigg[ \inf_{\beta \geq 1} \beta\bigg\{\iprod{g}{\mu}- \pnorm{\mu}{}\sup_{w \in L^\circ(\beta;x)} \iprod{h}{w} \bigg\}\bigg]\\
	&\geq \sup_{\mu \in K,\iprod{\mu}{g}\geq 0} \inf_{\beta \geq 1} \Big\{ \beta\iprod{g}{\mu}- \sqrt{\beta^2-1}\cdot \pnorm{\mu}{} \pnorm{h_{[2:\ell^\circ]}^x }{} \Big\}-\abs{h_1^x}.
	\end{align*}
	Here the last inequality follows from similar arguments as in (\ref{ineq:DM_upper_2}) in the proof of Proposition \ref{prop:DM_upper}. By Lemma \ref{lem:fcn_Q}, the first term above equals
	\begin{align*}
	&\bigg\{\sup_{\mu \in K_n,\iprod{\mu}{g}\geq 0}\Big(\iprod{g}{\mu}^2-\pnorm{\mu}{}^2\pnorm{h_{[2:\ell^\circ]}^x }{}^2\Big)_+ \bigg\}^{1/2} = \bigg\{\bigg(\sup_{\mu \in K_n}\iprod{g}{\mu}\bigg)^2- \pnorm{h_{[2:\ell^\circ]}^x }{}^2\bigg\}_+^{1/2}.
	\end{align*}
	The remaining proof follows from the same lines as in Proposition \ref{prop:DM_upper}. 
\end{proof}

Now we are in a good position to prove Theorem \ref{thm:local_dm}. 
\begin{proof}[Proof of Theorem \ref{thm:local_dm}]
	We write $S_k\equiv \Pi_{m\to k} L$. By Propositions \ref{prop:DM_upper} and \ref{prop:DM_lower}, for any $x \in S_k\cap \partial B_k$, there exists some $C_0=C_0(\tau)>0$ such that for any $t\geq 1$, with probability at least $1-e^{-t}$, we have
	\begin{align*}
	\bigabs{\mathsf{h}_{L\cap G K_n}(x)- \sqrt{\delta(K)-\delta(L^\circ)}}\leq C_0\sqrt{t}. 
	\end{align*}
	Now choosing $t\equiv \epsilon^2(\delta(K)-\delta(L^\circ))/(4C_0^2)\geq c \epsilon^2 \delta(K)$, for any $x \in S_k\cap \partial B_k$ with probability at least $1-\exp(-c\epsilon^2\delta(K))$,
	\begin{align}\label{ineq:local_dm_1}
	\bigg(1-\frac{\epsilon}{2}\bigg)\sqrt{\delta(K)-\delta(L^\circ)}\leq \mathsf{h}_{L\cap G K_n}(x)\leq \bigg(1+\frac{\epsilon}{2}\bigg) \sqrt{\delta(K)-\delta(L^\circ)}.
	\end{align}
	Let $T_\epsilon$  be a minimal $(\epsilon/2)\sqrt{\delta(K)-\delta(L^\circ)}$ net of $S_k\cap \partial B_k$ under the Euclidean metric. Then an easy volume estimate shows that the log cardinality of $T_\epsilon$ can be bounded by
	\begin{align*}
	\log T_\epsilon\lesssim k \log \bigg(\frac{1}{\epsilon \sqrt{\delta(K)-\delta(L^\circ)}}\bigg)\lesssim_\tau k \log\bigg(\frac{1}{\epsilon^2\delta(K)}\bigg),
	\end{align*}
	where the last inequality follows from the assumption that $\delta(K)-\delta(L^\circ)\gtrsim_\tau \delta(K)\gtrsim 1$. 
	As $\mathsf{h}_{L\cap G K_n}$ is $1$-Lipschitz, we have
	\begin{align}\label{ineq:local_dm_2}
	\sup_{x \in S_k\cap \partial B_k}\inf_{x' \in T_\epsilon}\bigabs{\mathsf{h}_{L\cap G K_n}(x)-\mathsf{h}_{L\cap G K_n}(x') }\leq \frac{\epsilon}{2}\sqrt{\delta(K)-\delta(L^\circ)}.
	\end{align}
	Combining (\ref{ineq:local_dm_1}) and (\ref{ineq:local_dm_2}) using a union bound, we have 
	\begin{align*}
	(1-\epsilon)\sqrt{\delta(K)-\delta(L^\circ)}\leq \mathsf{h}_{L\cap G K_n}(x)\leq (1+\epsilon) \sqrt{\delta(K)-\delta(L^\circ)},\quad \forall x \in S_k\cap \partial B_k,
	\end{align*}
	with probability at least 
	\begin{align*}
	1- \exp\bigg[-c\epsilon^2\delta(K)+C k \log\bigg(\frac{1}{\epsilon^2\delta(K)}\bigg)\bigg]. 
	\end{align*}
	The claim follows from the assumption on the dimension $k$. 
\end{proof}

\section*{Acknowledgments}
The research of Q. Han is partially supported by NSF grants DMS-1916221 and DMS-2143468. The authors would like to thank two referees and an Associate Editor for numerous helpful comments and suggestions that significantly improved the quality of the paper.

\bibliographystyle{amsalpha}
\bibliography{mybib}

@article{han2023distribution,
	title={The distribution of Ridgeless least squares interpolators},
	author={Han, Qiyang and Xu, Xiaocong},
	journal={arXiv preprint arXiv:2307.02044},
	year={2023}
}

@article {montanari2022interpolation,
	AUTHOR = {Montanari, Andrea and Zhong, Yiqiao},
	TITLE = {The interpolation phase transition in neural networks:
	memorization and generalization under lazy training},
	JOURNAL = {Ann. Statist.},
	FJOURNAL = {The Annals of Statistics},
	VOLUME = {50},
	YEAR = {2022},
	NUMBER = {5},
	PAGES = {2816--2847},
	ISSN = {0090-5364},
	MRCLASS = {62J07 (62J05)},
	MRNUMBER = {4500626},
	DOI = {10.1214/22-aos2211},
	URL = {https://doi-org.proxy.libraries.rutgers.edu/10.1214/22-aos2211},
}

@article{wu2020optimal,
	title={On the Optimal Weighted $\ell_2$ Regularization in Overparameterized Linear Regression},
	author={Wu, Denny and Xu, Ji},
	journal={Advances in Neural Information Processing Systems},
	volume={33},
	pages={10112--10123},
	year={2020}
}

@article {bartlett2020benign,
	AUTHOR = {Bartlett, Peter L. and Long, Philip M. and Lugosi, G\'{a}bor and
	Tsigler, Alexander},
	TITLE = {Benign overfitting in linear regression},
	JOURNAL = {Proc. Natl. Acad. Sci. USA},
	FJOURNAL = {Proceedings of the National Academy of Sciences of the United
	States of America},
	VOLUME = {117},
	YEAR = {2020},
	NUMBER = {48},
	PAGES = {30063--30070},
	ISSN = {0027-8424},
	MRCLASS = {62J05 (62M45)},
	MRNUMBER = {4263288},
	DOI = {10.1073/pnas.1907378117},
	URL = {https://doi-org.proxy.libraries.rutgers.edu/10.1073/pnas.1907378117},
}

@article {bartlett2021deep,
	AUTHOR = {Bartlett, Peter L. and Montanari, Andrea and Rakhlin,
	Alexander},
	TITLE = {Deep learning: a statistical viewpoint},
	JOURNAL = {Acta Numer.},
	FJOURNAL = {Acta Numerica},
	VOLUME = {30},
	YEAR = {2021},
	PAGES = {87--201},
	ISSN = {0962-4929},
	MRCLASS = {62J02 (60E15 60F15 62G20 62H30 62M45)},
	MRNUMBER = {4295218},
	MRREVIEWER = {Suman Sanyal},
	DOI = {10.1017/S0962492921000027},
	URL = {https://doi-org.proxy.libraries.rutgers.edu/10.1017/S0962492921000027},
}

@article{oymak2010new,
	title={New null space results and recovery thresholds for matrix rank minimization},
	author={Oymak, Samet and Hassibi, Babak},
	journal={arXiv preprint arXiv:1011.6326},
	year={2010}
}

@article{stojnic2009various,
	title={Various thresholds for $\ell_1$-optimization in compressed sensing},
	author={Stojnic, Mihailo},
	journal={arXiv preprint arXiv:0907.3666},
	year={2009}
}

@book {eaton1989group,
	AUTHOR = {Eaton, Morris L.},
	TITLE = {Group invariance applications in statistics},
	SERIES = {NSF-CBMS Regional Conference Series in Probability and
	Statistics},
	VOLUME = {1},
	PUBLISHER = {Institute of Mathematical Statistics, Hayward, CA; American
	Statistical Association, Alexandria, VA},
	YEAR = {1989},
	PAGES = {vi+133},
	ISBN = {0-940600-15-3},
	MRCLASS = {62A05 (62C05 62H15)},
	MRNUMBER = {1089423},
	MRREVIEWER = {N. Giri},
}

@article{han2024exact,
	title={Exact bounds for some quadratic empirical processes with applications},
	author={Han, Qiyang},
	journal={arXiv preprint arXiv:2207.13594v3},
	year={2024}
}

@article{lopes2011more,
	title={A more powerful two-sample test in high dimensions using random projection},
	author={Lopes, Miles and Jacob, Laurent and Wainwright, Martin J},
	journal={Advances in Neural Information Processing Systems},
	volume={24},
	year={2011}
}

@article {baraniuk2009random,
	AUTHOR = {Baraniuk, Richard G. and Wakin, Michael B.},
	TITLE = {Random projections of smooth manifolds},
	JOURNAL = {Found. Comput. Math.},
	FJOURNAL = {Foundations of Computational Mathematics. The Journal of the
	Society for the Foundations of Computational Mathematics},
	VOLUME = {9},
	YEAR = {2009},
	NUMBER = {1},
	PAGES = {51--77},
	ISSN = {1615-3375},
	MRCLASS = {94A12 (65C50)},
	MRNUMBER = {2472287},
	MRREVIEWER = {Steven B. Damelin},
	DOI = {10.1007/s10208-007-9011-z},
	URL = {https://doi.org/10.1007/s10208-007-9011-z},
}

@inproceedings{dasgupta2008random,
	title={Random projection trees and low dimensional manifolds},
	author={Dasgupta, Sanjoy and Freund, Yoav},
	booktitle={Proceedings of the fortieth annual ACM symposium on Theory of computing},
	pages={537--546},
	year={2008}
}

@inproceedings{dahl2013large,
	title={Large-scale malware classification using random projections and neural networks},
	author={Dahl, George E and Stokes, Jack W and Deng, Li and Yu, Dong},
	booktitle={2013 IEEE International Conference on Acoustics, Speech and Signal Processing},
	pages={3422--3426},
	year={2013},
	organization={IEEE}
}

@incollection {achlioptas2003database,
	AUTHOR = {Achlioptas, Dimitris},
	TITLE = {Database-friendly random projections:
	{J}ohnson-{L}indenstrauss with binary coins},
	NOTE = {Special issue on PODS 2001 (Santa Barbara, CA)},
	JOURNAL = {J. Comput. System Sci.},
	FJOURNAL = {Journal of Computer and System Sciences},
	VOLUME = {66},
	YEAR = {2003},
	NUMBER = {4},
	PAGES = {671--687},
	ISSN = {0022-0000},
	MRCLASS = {68P15 (60C05 68W20)},
	MRNUMBER = {2005771},
	DOI = {10.1016/S0022-0000(03)00025-4},
	URL = {https://doi.org/10.1016/S0022-0000(03)00025-4},
}

@article {cannings2017random,
	AUTHOR = {Cannings, Timothy I. and Samworth, Richard J.},
	TITLE = {Random-projection ensemble classification},
	NOTE = {With discussions and a reply by the authors},
	JOURNAL = {J. R. Stat. Soc. Ser. B. Stat. Methodol.},
	FJOURNAL = {Journal of the Royal Statistical Society. Series B.
	Statistical Methodology},
	VOLUME = {79},
	YEAR = {2017},
	NUMBER = {4},
	PAGES = {959--1035},
	ISSN = {1369-7412},
	MRCLASS = {62H30},
	MRNUMBER = {3689307},
	DOI = {10.1111/rssb.12228},
	URL = {https://doi.org/10.1111/rssb.12228},
}

@article {durrant2015random,
	AUTHOR = {Durrant, Robert J. and Kab\'{a}n, Ata},
	TITLE = {Random projections as regularizers: learning a linear
	discriminant from fewer observations than dimensions},
	JOURNAL = {Mach. Learn.},
	FJOURNAL = {Machine Learning},
	VOLUME = {99},
	YEAR = {2015},
	NUMBER = {2},
	PAGES = {257--286},
	ISSN = {0885-6125},
	MRCLASS = {62H30 (62-07 62G05)},
	MRNUMBER = {3335521},
	DOI = {10.1007/s10994-014-5466-8},
	URL = {https://doi.org/10.1007/s10994-014-5466-8},
}

@incollection {vempala2001random,
	AUTHOR = {Vempala, Santosh},
	TITLE = {The random projection method},
	BOOKTITLE = {Handbook of randomized computing, {V}ol. {I}, {II}},
	SERIES = {Comb. Optim.},
	VOLUME = {9},
	PAGES = {651--671},
	PUBLISHER = {Kluwer Acad. Publ., Dordrecht},
	YEAR = {2001},
	MRCLASS = {68W99 (68W25 90C15)},
	MRNUMBER = {1966914},
	DOI = {10.1007/978-1-4615-0013-1\_16},
	URL = {https://doi.org/10.1007/978-1-4615-0013-1_16},
}

@article {jiang2006how,
	AUTHOR = {Jiang, Tiefeng},
	TITLE = {How many entries of a typical orthogonal matrix can be
	approximated by independent normals?},
	JOURNAL = {Ann. Probab.},
	FJOURNAL = {The Annals of Probability},
	VOLUME = {34},
	YEAR = {2006},
	NUMBER = {4},
	PAGES = {1497--1529},
	ISSN = {0091-1798},
	MRCLASS = {60B15 (15A52 60B10 60F05 62H10)},
	MRNUMBER = {2257653},
	MRREVIEWER = {Florent Benaych-Georges},
	DOI = {10.1214/009117906000000205},
	URL = {https://doi.org/10.1214/009117906000000205},
}

@article {jiang2019distances,
	AUTHOR = {Jiang, Tiefeng and Ma, Yutao},
	TITLE = {Distances between random orthogonal matrices and independent
	normals},
	JOURNAL = {Trans. Amer. Math. Soc.},
	FJOURNAL = {Transactions of the American Mathematical Society},
	VOLUME = {372},
	YEAR = {2019},
	NUMBER = {3},
	PAGES = {1509--1553},
	ISSN = {0002-9947},
	MRCLASS = {15B52 (28C10 51F25 60B15 62E17)},
	MRNUMBER = {3976569},
	MRREVIEWER = {Dominique L\'{e}pingle},
	DOI = {10.1090/tran/7470},
	URL = {https://doi.org/10.1090/tran/7470},
}

@article {tropp2012comparison,
	AUTHOR = {Tropp, Joel A.},
	TITLE = {A comparison principle for functions of a uniformly random
	subspace},
	JOURNAL = {Probab. Theory Related Fields},
	FJOURNAL = {Probability Theory and Related Fields},
	VOLUME = {153},
	YEAR = {2012},
	NUMBER = {3-4},
	PAGES = {759--769},
	ISSN = {0178-8051},
	MRCLASS = {60B20 (15B52)},
	MRNUMBER = {2948692},
	MRREVIEWER = {Estelle L. Basor},
	DOI = {10.1007/s00440-011-0360-9},
	URL = {https://doi.org/10.1007/s00440-011-0360-9},
}

@inproceedings{thrampoulidis2015isotropically,
	title={Isotropically random orthogonal matrices: Performance of lasso and minimum conic singular values},
	author={Thrampoulidis, Christos and Hassibi, Babak},
	booktitle={2015 IEEE International Symposium on Information Theory (ISIT)},
	pages={556--560},
	year={2015},
	organization={IEEE}
}

@article{milman1971new,
	title={A new proof of the theorem of {A}. {D}voretzky on sections of convex bodies},
	author={Milman, Vitali D.},
	journal={Funct. Anal. Appl},
	volume={5},
	pages={28--37},
	year={1971}
}

@article {dvoretzky1959theorem,
	AUTHOR = {Dvoretzky, Aryeh},
	TITLE = {A theorem on convex bodies and applications to {B}anach
	spaces},
	JOURNAL = {Proc. Nat. Acad. Sci. U.S.A.},
	FJOURNAL = {Proceedings of the National Academy of Sciences of the United
	States of America},
	VOLUME = {45},
	YEAR = {1959},
	PAGES = {223--226; erratum, 1554},
	ISSN = {0027-8424},
	MRCLASS = {52.00 (46.00)},
	MRNUMBER = {105652},
	MRREVIEWER = {H. G. Eggleston},
	DOI = {10.1073/pnas.45.10.1554},
	URL = {https://doi.org/10.1073/pnas.45.10.1554},
}

@inproceedings {dvoretzky1961some,
	AUTHOR = {Dvoretzky, Aryeh},
	TITLE = {Some results on convex bodies and {B}anach spaces},
	BOOKTITLE = {Proc. {I}nternat. {S}ympos. {L}inear {S}paces ({J}erusalem,
	1960)},
	PAGES = {123--160},
	PUBLISHER = {Jerusalem Academic Press, Jerusalem; Pergamon, Oxford},
	YEAR = {1961},
	MRCLASS = {52.30},
	MRNUMBER = {0139079},
	MRREVIEWER = {M. M. Day},
}

@incollection {johnson1984extensions,
	AUTHOR = {Johnson, William B. and Lindenstrauss, Joram},
	TITLE = {Extensions of {L}ipschitz mappings into a {H}ilbert space},
	BOOKTITLE = {Conference in modern analysis and probability ({N}ew {H}aven,
	{C}onn., 1982)},
	SERIES = {Contemp. Math.},
	VOLUME = {26},
	PAGES = {189--206},
	PUBLISHER = {Amer. Math. Soc., Providence, RI},
	YEAR = {1984},
	MRCLASS = {46B20},
	MRNUMBER = {737400},
	MRREVIEWER = {Yehoram Gordon},
	DOI = {10.1090/conm/026/737400},
	URL = {https://doi.org/10.1090/conm/026/737400},
}

@book {artstein2015asymptotic,
	AUTHOR = {Artstein-Avidan, Shiri and Giannopoulos, Apostolos and Milman,
	Vitali D.},
	TITLE = {Asymptotic geometric analysis. {P}art {I}},
	SERIES = {Mathematical Surveys and Monographs},
	VOLUME = {202},
	PUBLISHER = {American Mathematical Society, Providence, RI},
	YEAR = {2015},
	PAGES = {xx+451},
	ISBN = {978-1-4704-2193-9},
	MRCLASS = {52A21 (28Axx 46-02 46Bxx 52A23 52A40)},
	MRNUMBER = {3331351},
	MRREVIEWER = {Artem Zvavitch},
	DOI = {10.1090/surv/202},
	URL = {https://doi.org/10.1090/surv/202},
}

@book {schneider2008stochastic,
	AUTHOR = {Schneider, Rolf and Weil, Wolfgang},
	TITLE = {Stochastic and integral geometry},
	SERIES = {Probability and its Applications (New York)},
	PUBLISHER = {Springer-Verlag, Berlin},
	YEAR = {2008},
	PAGES = {xii+693},
	ISBN = {978-3-540-78858-4},
	MRCLASS = {60-02 (52A22 60D05 60G55 62M30)},
	MRNUMBER = {2455326},
	MRREVIEWER = {V. K. Ohanyan},
	DOI = {10.1007/978-3-540-78859-1},
	URL = {https://doi.org/10.1007/978-3-540-78859-1},
}

@book {bental2001lectures,
	AUTHOR = {Ben-Tal, Aharon and Nemirovski, Arkadi},
	TITLE = {Lectures on modern convex optimization},
	SERIES = {MPS/SIAM Series on Optimization},
	PUBLISHER = {MPS/SIAM, Philadelphia, PA},
	YEAR = {2001},
	PAGES = {xvi+488},
	ISBN = {0-89871-491-5},
	MRCLASS = {90-01 (49-01 90C05 90C22 90C51)},
	MRNUMBER = {1857264},
	MRREVIEWER = {Jos F. Sturm},
	DOI = {10.1137/1.9780898718829},
	URL = {https://doi.org/10.1137/1.9780898718829},
}

@article {amelunxen2015intrinsic,
	AUTHOR = {Amelunxen, Dennis and B\"{u}rgisser, Peter},
	TITLE = {Intrinsic volumes of symmetric cones and applications in
	convex programming},
	JOURNAL = {Math. Program.},
	FJOURNAL = {Mathematical Programming},
	VOLUME = {149},
	YEAR = {2015},
	NUMBER = {1-2, Ser. A},
	PAGES = {105--130},
	ISSN = {0025-5610},
	MRCLASS = {15B48 (28A75 52A20 53C65 60D05 90C22)},
	MRNUMBER = {3300458},
	MRREVIEWER = {Roland Hildebrand},
	DOI = {10.1007/s10107-013-0740-2},
	URL = {https://doi.org/10.1007/s10107-013-0740-2},
}

@article {candes2020phase,
	AUTHOR = {Cand\`es, Emmanuel J. and Sur, Pragya},
	TITLE = {The phase transition for the existence of the maximum
	likelihood estimate in high-dimensional logistic regression},
	JOURNAL = {Ann. Statist.},
	FJOURNAL = {The Annals of Statistics},
	VOLUME = {48},
	YEAR = {2020},
	NUMBER = {1},
	PAGES = {27--42},
	ISSN = {0090-5364},
	MRCLASS = {62J12 (62E20)},
	MRNUMBER = {4065151},
	MRREVIEWER = {Yi-Hau Chen},
	DOI = {10.1214/18-AOS1789},
	URL = {https://doi.org/10.1214/18-AOS1789},
}

@article{cover1965geometrical,
	title={Geometrical and statistical properties of systems of linear inequalities with applications in pattern recognition},
	author={Cover, Thomas M},
	journal={IEEE Transactions on Electronic Computers},
	number={3},
	pages={326--334},
	year={1965},
	publisher={IEEE}
}

@article {klartag2005empirical,
	AUTHOR = {Klartag, B. and Mendelson, S.},
	TITLE = {Empirical processes and random projections},
	JOURNAL = {J. Funct. Anal.},
	FJOURNAL = {Journal of Functional Analysis},
	VOLUME = {225},
	YEAR = {2005},
	NUMBER = {1},
	PAGES = {229--245},
	ISSN = {0022-1236},
	MRCLASS = {60D05 (46B09)},
	MRNUMBER = {2149924},
	MRREVIEWER = {Almut Burchard},
	DOI = {10.1016/j.jfa.2004.10.009},
	URL = {https://doi.org/10.1016/j.jfa.2004.10.009},
}

@book {vershynin2018high,
	AUTHOR = {Vershynin, Roman},
	TITLE = {High-dimensional probability: An introduction with applications in data science},
	SERIES = {Cambridge Series in Statistical and Probabilistic Mathematics},
	VOLUME = {47},
	PUBLISHER = {Cambridge University Press, Cambridge},
	YEAR = {2018},
	PAGES = {xiv+284},
	ISBN = {978-1-108-41519-4},
	MRCLASS = {60-01 (60B05 60B20 60E15 60Fxx 62H25)},
	MRNUMBER = {3837109},
	MRREVIEWER = {Sasha Sodin},
	DOI = {10.1017/9781108231596},
	URL = {https://doi.org/10.1017/9781108231596},
}

@article {han2023universality,
	AUTHOR = {Han, Qiyang and Shen, Yandi},
	TITLE = {Universality of regularized regression estimators in high
	dimensions},
	JOURNAL = {Ann. Statist.},
	FJOURNAL = {The Annals of Statistics},
	VOLUME = {51},
	YEAR = {2023},
	NUMBER = {4},
	PAGES = {1799--1823},
	ISSN = {0090-5364,2168-8966},
	MRCLASS = {62J05 (60F17 62E17 62J07)},
	MRNUMBER = {4658577},
	DOI = {10.1214/23-aos2309},
	URL = {https://doi.org/10.1214/23-aos2309},
}

@article{montanari2023generalization,
	title={The generalization error of max-margin linear classifiers: Benign overfitting and high-dimensional asymptotics in the overparametrized regime},
	author={Montanari, Andrea and Ruan, Feng and Sohn, Youngtak and Yan, Jun},
	journal={arXiv preprint arXiv:1911.01544v3},
	year={2023}
}

@article {han2023noisy,
	AUTHOR = {Han, Qiyang},
	TITLE = {Noisy linear inverse problems under convex constraints:
	{E}xact risk asymptotics in high dimensions},
	JOURNAL = {Ann. Statist.},
	FJOURNAL = {The Annals of Statistics},
	VOLUME = {51},
	YEAR = {2023},
	NUMBER = {4},
	PAGES = {1611--1638},
	ISSN = {0090-5364,2168-8966},
	MRCLASS = {62G08 (62E17 62G20)},
	MRNUMBER = {4658570},
	DOI = {10.1214/23-aos2301},
	URL = {https://doi.org/10.1214/23-aos2301},
}

@article {hastie2022surprises,
	AUTHOR = {Hastie, Trevor and Montanari, Andrea and Rosset, Saharon and
	Tibshirani, Ryan J.},
	TITLE = {Surprises in high-dimensional ridgeless least squares
	interpolation},
	JOURNAL = {Ann. Statist.},
	FJOURNAL = {The Annals of Statistics},
	VOLUME = {50},
	YEAR = {2022},
	NUMBER = {2},
	PAGES = {949--986},
	ISSN = {0090-5364},
	MRCLASS = {62J05 (62F12 62J02 62J07)},
	MRNUMBER = {4404925},
	DOI = {10.1214/21-aos2133},
	URL = {https://doi.org/10.1214/21-aos2133},
}

@article {celentano2023lasso,
	AUTHOR = {Celentano, Michael and Montanari, Andrea and Wei, Yuting},
	TITLE = {The {L}asso with general {G}aussian designs with applications to hypothesis testing},
	JOURNAL = {Ann. Statist.},
	FJOURNAL = {The Annals of Statistics},
	VOLUME = {51},
	YEAR = {2023},
	NUMBER = {5},
	PAGES = {2194--2220},
	ISSN = {0090-5364,2168-8966},
	MRCLASS = {62J07 (62E17 62F05 62F12)},
	MRNUMBER = {4678801},
	DOI = {10.1214/23-aos2327},
	URL = {https://doi.org/10.1214/23-aos2327},
}

@article {candes2006near,
	AUTHOR = {Candes, Emmanuel J. and Tao, Terence},
	TITLE = {Near-optimal signal recovery from random projections:
	universal encoding strategies?},
	JOURNAL = {IEEE Trans. Inform. Theory},
	FJOURNAL = {Institute of Electrical and Electronics Engineers.
	Transactions on Information Theory},
	VOLUME = {52},
	YEAR = {2006},
	NUMBER = {12},
	PAGES = {5406--5425},
	ISSN = {0018-9448},
	MRCLASS = {94A12 (41A25 94A13)},
	MRNUMBER = {2300700},
	MRREVIEWER = {L. L. Campbell},
	DOI = {10.1109/TIT.2006.885507},
	URL = {https://doi.org/10.1109/TIT.2006.885507},
}

@article {miolane2021distribution,
	AUTHOR = {Miolane, L\'{e}o and Montanari, Andrea},
	TITLE = {The distribution of the {L}asso: uniform control over sparse
	balls and adaptive parameter tuning},
	JOURNAL = {Ann. Statist.},
	FJOURNAL = {The Annals of Statistics},
	VOLUME = {49},
	YEAR = {2021},
	NUMBER = {4},
	PAGES = {2313--2335},
	ISSN = {0090-5364},
	MRCLASS = {62J05 (62J07)},
	MRNUMBER = {4319252},
	DOI = {10.1214/20-aos2038},
	URL = {https://doi.org/10.1214/20-aos2038},
}

@article{stojnic2013framework,
	title={A framework to characterize performance of lasso algorithms},
	author={Stojnic, Mihailo},
	journal={arXiv preprint arXiv:1303.7291},
	year={2013}
}

@incollection {gordon1988milman,
	AUTHOR = {Gordon, Yehoram},
	TITLE = {On {M}ilman's inequality and random subspaces which escape
	through a mesh in {${\bf R}^n$}},
	BOOKTITLE = {Geometric aspects of functional analysis (1986/87)},
	SERIES = {Lecture Notes in Math.},
	VOLUME = {1317},
	PAGES = {84--106},
	PUBLISHER = {Springer, Berlin},
	YEAR = {1988},
	MRCLASS = {46B20 (52A40 60B11 60G15)},
	MRNUMBER = {950977},
	MRREVIEWER = {Jerzy Szulga},
	DOI = {10.1007/BFb0081737},
	URL = {https://doi.org/10.1007/BFb0081737},
}

@article {gordon1985some,
	AUTHOR = {Gordon, Yehoram},
	TITLE = {Some inequalities for {G}aussian processes and applications},
	JOURNAL = {Israel J. Math.},
	FJOURNAL = {Israel Journal of Mathematics},
	VOLUME = {50},
	YEAR = {1985},
	NUMBER = {4},
	PAGES = {265--289},
	ISSN = {0021-2172},
	MRCLASS = {60G15 (52A22)},
	MRNUMBER = {800188},
	MRREVIEWER = {Naresh C. Jain},
	DOI = {10.1007/BF02759761},
	URL = {https://doi.org/10.1007/BF02759761},
}

@book {schneider2014convex,
	AUTHOR = {Schneider, Rolf},
	TITLE = {Convex bodies: the {B}runn-{M}inkowski theory},
	SERIES = {Encyclopedia of Mathematics and its Applications},
	VOLUME = {151},
	EDITION = {expanded},
	PUBLISHER = {Cambridge University Press, Cambridge},
	YEAR = {2014},
	PAGES = {xxii+736},
	ISBN = {978-1-107-60101-7},
	MRCLASS = {52-02 (52A20 52A39)},
	MRNUMBER = {3155183},
	MRREVIEWER = {Andrea Colesanti},
}

@incollection{thrampoulidis2015recovering,
	title={Recovering structured signals in noise: Least-squares meets compressed sensing},
	author={Thrampoulidis, Christos and Oymak, Samet and Hassibi, Babak},
	booktitle={Compressed Sensing and Its Applications},
	pages={97--141},
	year={2015},
	publisher={Springer}
}

@article {thrampoulidis2018precise,
	AUTHOR = {Thrampoulidis, Christos and Abbasi, Ehsan and Hassibi, Babak},
	TITLE = {Precise error analysis of regularized {$M$}-estimators in high
	dimensions},
	JOURNAL = {IEEE Trans. Inform. Theory},
	FJOURNAL = {Institute of Electrical and Electronics Engineers.
	Transactions on Information Theory},
	VOLUME = {64},
	YEAR = {2018},
	NUMBER = {8},
	PAGES = {5592--5628},
	ISSN = {0018-9448},
	MRCLASS = {94A12 (62H12 90C25)},
	MRNUMBER = {3832326},
	DOI = {10.1109/TIT.2018.2840720},
	URL = {https://doi.org/10.1109/TIT.2018.2840720},
}

@article {mccoy2014from,
	AUTHOR = {McCoy, Michael B. and Tropp, Joel A.},
	TITLE = {From {S}teiner formulas for cones to concentration of
	intrinsic volumes},
	JOURNAL = {Discrete Comput. Geom.},
	FJOURNAL = {Discrete \& Computational Geometry. An International Journal
	of Mathematics and Computer Science},
	VOLUME = {51},
	YEAR = {2014},
	NUMBER = {4},
	PAGES = {926--963},
	ISSN = {0179-5376},
	MRCLASS = {52A22 (52A20 60D05)},
	MRNUMBER = {3216671},
	MRREVIEWER = {Jan Rataj},
	DOI = {10.1007/s00454-014-9595-4},
	URL = {https://doi.org/10.1007/s00454-014-9595-4},
}

@article {goldstein2017gaussian,
	AUTHOR = {Goldstein, Larry and Nourdin, Ivan and Peccati, Giovanni},
	TITLE = {Gaussian phase transitions and conic intrinsic volumes:
	{S}teining the {S}teiner formula},
	JOURNAL = {Ann. Appl. Probab.},
	FJOURNAL = {The Annals of Applied Probability},
	VOLUME = {27},
	YEAR = {2017},
	NUMBER = {1},
	PAGES = {1--47},
	ISSN = {1050-5164},
	MRCLASS = {60D05 (60F05 62F30)},
	MRNUMBER = {3619780},
	MRREVIEWER = {Christoph Th\"{a}le},
	DOI = {10.1214/16-AAP1195},
	URL = {https://doi-org.proxy.libraries.rutgers.edu/10.1214/16-AAP1195},
}

@article {amelunxen2014living,
	AUTHOR = {Amelunxen, Dennis and Lotz, Martin and McCoy, Michael B. and
	Tropp, Joel A.},
	TITLE = {Living on the edge: phase transitions in convex programs with
	random data},
	JOURNAL = {Inf. Inference},
	FJOURNAL = {Information and Inference. A Journal of the IMA},
	VOLUME = {3},
	YEAR = {2014},
	NUMBER = {3},
	PAGES = {224--294},
	ISSN = {2049-8764},
	MRCLASS = {62F10 (62G05 62H25 94A12 94A15)},
	MRNUMBER = {3311453},
	DOI = {10.1093/imaiai/iau005},
	URL = {https://doi.org/10.1093/imaiai/iau005},
}

@book {pisier1989volume,
	AUTHOR = {Pisier, Gilles},
	TITLE = {The volume of convex bodies and {B}anach space geometry},
	SERIES = {Cambridge Tracts in Mathematics},
	VOLUME = {94},
	PUBLISHER = {Cambridge University Press, Cambridge},
	YEAR = {1989},
	PAGES = {xvi+250},
	ISBN = {0-521-36465-5; 0-521-66635-X},
	MRCLASS = {52A21 (46B20 52A07)},
	MRNUMBER = {1036275},
	MRREVIEWER = {Mikhail Ostrovskii},
	DOI = {10.1017/CBO9780511662454},
	URL = {https://doi.org/10.1017/CBO9780511662454},
}

@article {chatterjee2007multivariate,
	AUTHOR = {Chatterjee, Sourav and Meckes, Elizabeth},
	TITLE = {Multivariate normal approximation using exchangeable pairs},
	JOURNAL = {ALEA Lat. Am. J. Probab. Math. Stat.},
	FJOURNAL = {ALEA. Latin American Journal of Probability and Mathematical
	Statistics},
	VOLUME = {4},
	YEAR = {2008},
	PAGES = {257--283},
	ISSN = {1980-0436},
	MRCLASS = {60F05 (60B12 60B15)},
	MRNUMBER = {2453473},
}

@book {ledoux2013probability,
	AUTHOR = {Ledoux, Michel and Talagrand, Michel},
	TITLE = {Probability in {B}anach {S}paces},
	SERIES = {Classics in Mathematics},
	NOTE = {Isoperimetry and processes,
	Reprint of the 1991 edition},
	PUBLISHER = {Springer-Verlag, Berlin},
	YEAR = {2011},
	PAGES = {xii+480},
	ISBN = {978-3-642-20211-7},
	MRCLASS = {60B11 (46N30 60Fxx 60Gxx)},
	MRNUMBER = {2814399},
}

@book {gine2015mathematical,
	AUTHOR = {Gin\'{e}, Evarist and Nickl, Richard},
	TITLE = {Mathematical foundations of infinite-dimensional statistical
	models},
	SERIES = {Cambridge Series in Statistical and Probabilistic Mathematics,
	[40]},
	PUBLISHER = {Cambridge University Press, New York},
	YEAR = {2016},
	PAGES = {xiv+690},
	ISBN = {978-1-107-04316-9},
	MRCLASS = {62-02 (60F05 60F17 60G15 62C20 62G07 62G08 62Gxx)},
	MRNUMBER = {3588285},
	MRREVIEWER = {Natalie Neumeyer},
	DOI = {10.1017/CBO9781107337862},
	URL = {https://doi-org.proxy.libraries.rutgers.edu/10.1017/CBO9781107337862},
}

@article {rudelson2008sparse,
	AUTHOR = {Rudelson, Mark and Vershynin, Roman},
	TITLE = {On sparse reconstruction from {F}ourier and {G}aussian
	measurements},
	JOURNAL = {Comm. Pure Appl. Math.},
	FJOURNAL = {Communications on Pure and Applied Mathematics},
	VOLUME = {61},
	YEAR = {2008},
	NUMBER = {8},
	PAGES = {1025--1045},
	ISSN = {0010-3640},
	MRCLASS = {94A20 (46B07 46B09 60B11 90C90)},
	MRNUMBER = {2417886},
	MRREVIEWER = {Alberto Portal},
	DOI = {10.1002/cpa.20227},
	URL = {http://dx.doi.org/10.1002/cpa.20227},
}

@article {chandrasekaran2012convex,
	AUTHOR = {Chandrasekaran, Venkat and Recht, Benjamin and Parrilo, Pablo
	A. and Willsky, Alan S.},
	TITLE = {The convex geometry of linear inverse problems},
	JOURNAL = {Found. Comput. Math.},
	FJOURNAL = {Foundations of Computational Mathematics. The Journal of the
	Society for the Foundations of Computational Mathematics},
	VOLUME = {12},
	YEAR = {2012},
	NUMBER = {6},
	PAGES = {805--849},
	ISSN = {1615-3375},
	MRCLASS = {41A45 (52A41 65K10 90C22)},
	MRNUMBER = {2989474},
	MRREVIEWER = {Qinghong Zhang},
	DOI = {10.1007/s10208-012-9135-7},
	URL = {https://doi-org.proxy.libraries.rutgers.edu/10.1007/s10208-012-9135-7},
}

@book {rockafellar1997convex,
	AUTHOR = {Rockafellar, R. Tyrrell},
	TITLE = {Convex {A}nalysis},
	SERIES = {Princeton Landmarks in Mathematics},
	NOTE = {Reprint of the 1970 original,
	Princeton Paperbacks},
	PUBLISHER = {Princeton University Press, Princeton, NJ},
	YEAR = {1997},
	PAGES = {xviii+451},
	ISBN = {0-691-01586-4},
	MRCLASS = {49-02 (26-02 46-02 58-02 90-02)},
	MRNUMBER = {1451876 (97m:49001)},
}

@book {boyd2004convex,
	AUTHOR = {Boyd, Stephen and Vandenberghe, Lieven},
	TITLE = {Convex {O}ptimization},
	PUBLISHER = {Cambridge University Press, Cambridge},
	YEAR = {2004},
	PAGES = {xiv+716},
	ISBN = {0-521-83378-7},
	MRCLASS = {90-01 (90C05 90C25 90C46 90C51)},
	MRNUMBER = {2061575 (2005d:90002)},
	MRREVIEWER = {Dan Butnariu},
	DOI = {10.1017/CBO9780511804441},
	URL = {http://dx.doi.org/10.1017/CBO9780511804441},
}

@book {boucheron2013concentration,
	AUTHOR = {Boucheron, St\'ephane and Lugosi, G\'abor and Massart, Pascal},
	TITLE = {Concentration inequalities: {A} nonasymptotic theory of independence},
	PUBLISHER = {Oxford University Press, Oxford},
	YEAR = {2013},
	PAGES = {x+481},
	ISBN = {978-0-19-953525-5},
	MRCLASS = {60E15 (60D05 60G15 60G50 60G70 62G20)},
	MRNUMBER = {3185193},
	MRREVIEWER = {Sreenivasan Ravi},
	DOI = {10.1093/acprof:oso/9780199535255.001.0001},
	URL = {http://dx.doi.org/10.1093/acprof:oso/9780199535255.001.0001},
}

\end{document}